\documentclass[a4paper]{amsart}
\usepackage[top=3cm, bottom=3cm, left=3cm, right=3cm]{geometry}

\usepackage{enumerate}
\usepackage{amsmath}
\usepackage{amssymb}
\usepackage{amsthm}

\usepackage{math}
\usepackage{thm}
\usepackage{ref}

\title
	[Multiplication in Vector-Valued Anisotropic Function Spaces]
	{Multiplication \\ in Vector-Valued Anisotropic Function Spaces and \\ Applications to Non-Linear Partial Differential Equations}

\author
	[Matthias K{\"o}hne]
	{Matthias K{\"o}hne}

\address
	{Mathematisches Institut --
	 Heinrich-Heine-Universit{\"a}t D{\"u}sseldorf \newline\indent
	 Universit{\"a}tsstr.~1, 40225 D{\"u}sseldorf, Germany}
	
\email
	{Matthias.Koehne@uni-duesseldorf.de}

\author
	[J{\"u}rgen Saal]
	{J{\"u}rgen Saal}

\address
	{Mathematisches Institut --
	 Heinrich-Heine-Universit{\"a}t D{\"u}sseldorf \newline\indent
	 Universit{\"a}tsstr.~1, 40225 D{\"u}sseldorf, Germany}
	
\email
	{Juergen.Saal@uni-duesseldorf.de}



\date
	{\today}

\begin{document}
\setlength{\parskip}{0.5\baselineskip}
\setlength{\parindent}{0pt}
\renewcommand{\baselinestretch}{1.125}
\setlength{\marginparwidth}{2.0cm}
\normalsize
\begin{abstract}
	We study multiplication as well as Nemytskij operators in an\-iso\-tro\-pic vector-valued Besov spaces $B^{s, \omega}_p$,
	Bessel potential spaces $H^{s, \omega}_p$, and Sobolev-Slobodeckij spaces $W^{s, \omega}_p$.
	Concerning multiplication we obtain optimal estimates, which constitute generalizations and improvements of known estimates in the isotropic/scalar-valued case.
	Concerning Nemytskij operators we consider the acting of analytic functions on supercritial anisotropic vector-valued function spaces of the above type.
	Moreover, we show how the given estimates may be used in order to improve results on quasilinear evolution equations as well as their proofs.
\end{abstract}
\maketitle
\vspace*{-1.5\baselineskip}

\section*{Introduction}
The systematic treatment of quasilinear evolution equations is often based on optimal estimates for suitable linearizations, i.\,e.\ maximal regularity,
and the handling of the related non-linear terms in the functional analytic setting prescribed by the linear problems.
Typical examples of such quasilinear systems are given by free boundary value problems such as Stefan problems or the two-phase Navier-Stokes equations,
see \Secref{Applications}.

Concerning elliptic problems the approach via maximal regularity for suitable linearizations and a fixed point argument is well-established nowadays.
In an $L_p$-setting the functional analytic framework is then based on {\em isotropic} Bessel potential and Sobolev-Slobodeckij spaces
and the study of the non-linear perturbations typically boil down to estimates on multiplication and Nemytskij operators in these scales of function spaces.
Results on multiplication in isotropic scalar-valued Besov and Sobolev spaces have been derived by several authors,
cf.~\cite{Mazya-Shaposhnikova:Sobolev-Multipliers, Miyachi:Multiplication, Triebel:Multiplication, Valent:Multiplication, Zolesio:Multiplication},
while the vector-valued case has first been considered by {\scshape H.~Amann} in \cite{Amann:Multiplication}.
A coincise study of Nemytskij operators in isotropic scalar-valued Besov and Triebel-Lizorkin spaces may
e.\,g.\ be found in the book by {\scshape T.~Runst} and {\scshape W.~Sickel} \cite{Runst-Sickel:Function-Spaces}.
We do not want to recall all known results in this direction and instead refer to \cite{Amann:Multiplication, Runst-Sickel:Function-Spaces}
and the references therein.

While handling evolution equations that lead to parabolic linearizations, 
however, we are usually faced to the fact that the solutions exhibit different regularities in temporal and spatial variables.
Thus, in an $L_p$-setting a suitable functional analytic framework has to be based on {\em anisotropic} $L_p$-based function spaces.
{\scshape J.~Johnsen} developed in \cite{Johnsen:Multiplication} the corresponding theory for anisotropic scalar-valued Besov and Triebel-Lizorkin spaces
based on estimates for {\itshape paraproducts}.
It is well-known that the scale of vector-valued Bessel potential spaces is not included in the (standard) scale of vector-valued  Triebel-Lizorkin spaces;
in fact $H^s_p(\bR^n,\,X) = F^s_{p, 2}(\bR^n,\,X)$, if and only if $X$ is a Hilbert Space, cf.~\cite{Han-Meyer:Littlewood-Paley}.
For a systematic treatment of quasilinear evolution problems, however, vector-valued  Bessel potential and Sobolev-Slobodeckij scales play a fundamental role.
This is also demonstrated by the examples that will be discussed in \Secref{Applications}.
For those scales of function spaces a sytematic theory still seems to be missing in existing literature.

This is also underlined by the fact that existing well-posedness results for free boundary value problems in the $L_p$-setting 
are not optimal with respect to the range of $p$.
Let us briefly explain this for the two examples discussed later in more detail.
In the $L_p$-maximal regularity framework the Neumann trace space (for second order problems) is given as
\begin{equation*}
	\bY(a) := W^{1/2 - 1/2p}_p(J,\,L_p(\bR^n)) \cap L_p(J,\,W^{1 - 1/p}_p(\bR^n))
\end{equation*}
for a time interval $J = (0,\,a)$ with $0 < a \leq \infty$ and $1 < p < \infty$.
A typical non-linearity appearing in the Stefan condition of the (one-phase) Stefan problem
or on the boundary of the free boundary Navier-Stokes equations with surface tension reads as
\begin{equation}
	\eqnlabel{Example-Non-Linearity}
	G(u,\,h) = |\nabla_\Sigma h|^2\,\partial_\nu u,
\end{equation}
which then has to be estimated in $\bY(a)$.
Here $u$ denotes the temperature or the fluid velocity and $h$ denotes a height function that parametrizes the free boundary.
The height function is defined on a fixed reference manifold $\Sigma$ and $\partial_\nu$ denotes the normal derivative w.\,r.\,t.\ $\Sigma$,
cf.~\Secref{Applications} where the (two-phase) Stefan problem is discussed.
In previous approaches to the Stefan problem,
see e.\,g.\ \cite{Escher-Pruess-Simonett:Stefan-Analytic, Pruess-Saal-Simonett:Stefan-Analytic-Classical,
Pruess-Saal-Simonett:Stefan-Singular-Limits, Pruess-Simonett:Stefan-Stability, Solonnikov-Frolova:Stefan-Problem},
or to the two-phase Navier-Stokes equations,
see e.\,g.\ \cite{Denisova:Two-Phase-Navier-Stokes, Denk-Geissert-Hieber-Saal-Sawada:Spin-Coating, Koehne-Pruess-Wilke:Two-Phase-Navier-Stokes,
Pruess-Simonett:Two-Phase-Navier-Stokes, Pruess-Simonett:Two-Phase-Navier-Stokes-Analytic, Shibata-Shimizu:Two-Phase-Navier-Stokes, Solonnikov:Navier-Stokes-Free-Surface},
the desired mapping property for $G$ is obtained by requiring $\bY(a)$ to be a multiplication algebra.
Then $G$ is commonly estimated as
\begin{equation}
	\eqnlabel{Example-Estimate}
	\|G(u,\,h)\|_{\bY(a)}
		\leq C \|\nabla_\Sigma h\|^2_{\bY(a)}\,\|\partial_\nu u\|_{\bY(a)}
		\leq C \|\nabla_\Sigma h\|^2_{\nabla_\Sigma \bX_h(a)} \|u\|_{\bX_u(a)},
\end{equation}
where $\bX_u(a)$ and $\bX_h(a)$ denote the maximal regularity spaces for $u$ and $h$, respectively,
and $\bX_h(a)$ satisfying $\nabla_\Sigma \bX_h(a) \hookrightarrow \bY(a)$ denotes the corresponding maximal regularity space for $h$;
see \Subsecref{Example-Stefan} for the concrete setting in case of the Stefan problem
and \Subsecref{Example-NVS} for the concrete setting in case of the two-phase Navier-Stokes equations.
Note that the algebra property
\begin{equation}
	\eqnlabel{Example-Algebra}
	\bY(a) \cdot \bY(a) \hookrightarrow \bY(a)
\end{equation}
is only available under the constraint
\begin{equation*}
	p > n + 2
\end{equation*}
for the integrability exponent $p$, cf.~\Thmref{Algebra-Isotropic}.
This, however, is far from being optimal.
In fact, from estimate \eqnref{Example-Estimate} we see that already the validity of the embedding
\begin{equation}
	\eqnlabel{Example-Embedding}
	\nabla_\Sigma \bX_h(a) \cdot \nabla_\Sigma \bX_h(a) \cdot \bY(a) \hookrightarrow \bY(a)
\end{equation}
is sufficient for \eqnref{Example-Estimate} to hold.
Applying \Thmref{Multiplication-Anisotropic}, one of the main results of this paper,
the embedding \eqnref{Example-Embedding} directly follows under the improved and sharp constraint
\begin{equation*}
	p > \frac{2 (n + 1)}{5}, \qquad \qquad p > \frac{n + 2}{2}, \qquad \qquad \textrm{respectively},
\end{equation*}
in case of the Stefan problem and the two-phase Navier-Stokes equations, respectively.
In fact, the discussion in Subsections~\ref{subsec:Example-Stefan} and \ref{subsec:Example-NVS} will show that,
based on our main results, for most of the existing results on strong well-posedness
for these two problems in the $L_p$-setting the constraint $p > n + 2$
can be improved to $p \geq (n + 2) / 2$ and $p > (n + 2) / 2$, respectively.

We want to remark that the optimal range for $p$ might be of crucial importance for the treatment of certain quasilinear evolution equations.
This could be the case, for instance, while investigating corresponding problems in non-smooth domains. 
Here, depending on the roughness of the domain, the well-posedness range for $p$ is a-priori restricted to a certain neighborhood of $p = 2$.
It might also be significant while passing from a parabolic equation to a singular hyperbolic limit, since then the case $p=2$ can be crucial.

The above discussion motivates the generalization of known results concerning multiplication and Nemytskij operators for the isotropic case to the anisotropic case.
And this is precisely the purpose of this paper.
We derive optimal results on multiplication of anisotropic vector-valued Besov, Bessel potential and Sobolev-Slobodeckij spaces
and on analytic Nemytskij operators acting on those scales, see \Thmref{Multiplication-Anisotropic} and \Thmref{Nemytskij-Anisotropic}.
Besides giving optimal lower bounds for the integrability parameter $p$ as explained above,
this approach yields several further advantages for the treatment of non-linear problems:
\vspace*{-0.5em}
\begin{itemize}
	\item it economizes the study of mapping properties of non-linearities, since merely the Sobolev index,
		see \eqnref{Anisotropic-Index}, of the involved anisotropic spaces has to be computed;
	\item it yields a rigorous but elementary verification method for the mapping properties of certain Nemytskij operators acting on anisotropic function spaces;
	\item it includes vector-valued function spaces for a large class of Banach spaces $E$.
\end{itemize}
\vspace*{-0.5em}
The usefulness of this approach is demonstrated by its application to the examples discussed in \Secref{Applications}.

We also emphasize that the vector-valued case already appears in the treatment of scalar-valued equations
such as the free boundary problems discussed in \Secref{Applications}.
This is due to the fact that the non-linearities in the bulk phases involve quantities that are defined in the bulk (like e.\,g.\ temperature, velocity)
as well as quantities that are defined on the boundary (like e.\,g.\ height functions, interfacial quantities).
In order to obtain optimal lower bounds for the integrability parameter $p$ in these cases,
it is necessary to treat the scalar-valued functions defined in the bulk as vector-valued functions defined on the interface; see \Secref{Applications} for the details.
Furthermore, in certain applications it is rather natural to work in a vector-valued setting.
This is the case for e.\,g.\ coagulation-fragmentation systems, cf.~\cite{Walker:Coagulation-Fragmentation}.

The paper is organized as follows.
In \Secref{Main} we introduce notation, recall known facts for isotropic and anisotropic function spaces
and state our main results on anisotropic vector-valued Besov, Bessel potential and Sobolev-Slobodeckij spaces. 
Those are given by the multiplication result \Thmref{Multiplication-Anisotropic} and
the result on analytic Nemytskij operators \Thmref{Nemytskij-Anisotropic}.
The proofs of all results of \Secref{Main} are given in Sections~\ref{sec:Multiplication-Proof}, \ref{sec:Multiplier-Proof}, and \ref{sec:Nemytskij-Proof}.

In \Secref{Applications} we apply the developed theory to two pertinent examples.
\Subsecref{Example-Stefan} deals with the Stefan problem with Gibbs-Thomson correction.
By applying our results we will demonstrate how proofs can be economized and results can be optimized,
as for instance given in \cite{Escher-Pruess-Simonett:Stefan-Analytic,Pruess-Saal-Simonett:Stefan-Analytic-Classical}.
In the same way in \Subsecref{Example-NVS} we demonstrate how proofs and results can be improved for the two-phase Navier-Stokes system
as compared e.\,g. to \cite{Pruess-Simonett:Two-Phase-Navier-Stokes,Pruess-Simonett:Two-Phase-Navier-Stokes-Analytic,Shibata-Shimizu:Two-Phase-Navier-Stokes,
Denk-Geissert-Hieber-Saal-Sawada:Spin-Coating,Koehne-Pruess-Wilke:Two-Phase-Navier-Stokes}.

Note that in this paper we do not consider the most general case.
We restrict our considerations to anisotropic vector-valued Bessel potential and Sobolev-Slobodeckij spaces,
since these are the most significant scales for the applications we have in mind.
Moreover, we consider anisotropic vector-valued Besov spaces that depend on one parameter only, cf.~\Subsecref{Anisotropic-Spaces},
since these naturally appear in some borderline cases that are not covered by the two scales above.
Furthermore, the collection of available embedding and interpolation results is much larger for the Besov scale than for the Sobolev-Slobodeckij scale,
cf.~the results collected in Appendix~A.
These results play a fundamental role in the proof of \Thmref{Multiplication-Anisotropic}, which gives a second reason to include the (restricted) Besov scale.
We remark, however, that generalizations of our main results to other scales of anisotropic vector-valued function spaces such as the (full) Besov scale
or the Triebel-Lizorkin scale are possible; see also Remarks~\ref{rem:Multiplication-Anisotropic} and \ref{rem:Nemytskij-Anisotropic}.

Concerning Nemytskij operators we restrict our considerations to the supercritical case,
where the involved function spaces are continuously embedded into spaces of bounded and continuous functions.
Moreover, we only consider Nemytskij operators that are defined by an analytic function.
As it will be demonstrated by the two examples in \Secref{Applications},
this is already sufficient for the treatment of quasilinear problems arising in the theory of free boundary value problems.
We remark, however, that generalizations of our result concerning Nemytskij operators to the subcritical case
and under much weaker conditions on the defining function are possible; see also \Remref{Nemytskij-Anisotropic}.

\section{Notation and Main Results}\seclabel{Main}
\subsection{Isotropic Function Spaces}\subseclabel{Isotropic-Spaces}
Recall that the vector-valued Bessel potential spaces are defined as
\begin{equation*}
	\begin{array}{rcl}
		        H^s_p(\bR^n,\,E) & := & \Big\{\,u \in \cS^\prime(\bR^n,\,E)\,:\,u = \cB^{- s} f,\ f \in L_p(\bR^n,\,E)\,\Big\},        \\[1.0em]
		\|u\|_{H^s_p(\bR^n,\,E)} & := & \|f\|_{L_p(\bR^n,\,E)}, \quad u \in H^s_p(\bR^n,\,E),\ f \in L_p(\bR^n,\,E),\ u = \cB^{- s} f,
	\end{array}
	\qquad
	\begin{array}{c}
		-\infty < s < \infty, \\[0.25em]
		1 < p < \infty,
	\end{array}
\end{equation*}
where $n \in \bN$ and $\cS^\prime(\bR^n,\,E) := \cL(\cS(\bR^n),\,E)$ denotes the space of tempered distributions.
The space $E$ may be an arbitrary Banach space and we denote by $\cL(X,\,Y)$ the space of all linear, continuous mappings between two topological vector spaces $X$ and $Y$.
We set $\cL(X) := \cL(X,\,X)$ and $\cS(\bR^n,\,E)$ denotes the Fr{\'e}chet space of all rapidly decreasing smooth functions
with values in the Banach space $E$, i.\,e.\ the $E$-valued Schwartz space.
The Bessel potentials are given as
\begin{equation*}
	\begin{array}{rclrcl}
		 \cB^s u & := & \cF^{-1} (\xi \mapsto B^s(\xi) \cF u(\xi)), & \qquad u   & \in & \cS(\bR^n,\,E), \\[0.5em]
		B^s(\xi) & := & (1 + |\xi|^2)^{s / 2},                      & \qquad \xi & \in & \bR^n,
	\end{array}
	\qquad - \infty < s < \infty,
\end{equation*}
and it is well-known that $\cB^s \in \cL\mbox{aut}(\cS(\bR^n,\,E))$, i.\,e.\ $\cB^s$ is a linear automorphism of $\cS(\bR^n,\,E)$, for all $- \infty < s < \infty$.
By duality we obtain $\cB^s \in \cL\mbox{aut}(\cS^\prime(\bR^n,\,E))$ for all $-\infty < s < \infty$, too,
cf.~\cite[Section~2.3.4]{Triebel:Interpolation} for the scalar-valued case and \cite[Section~2.3]{Amann:Maximal-Regularity} for the vector-valued (anisotropic) case.

The above definition relies on the Fourier transformation $\cF: \cS(\bR^n,\,E) \longrightarrow \cS(\bR^n,\,E)$,
which is defined as usual via
\begin{equation*}
	\cF u(\xi) := \int_{\bR^n} e^{- i x \cdot \xi}\,u(x)\,\mbox{d}x, \qquad \xi \in \bR^n, \qquad \qquad u \in \cS(\bR^n,\,E),
\end{equation*}
and its inverse $\cF^{-1}: \cS(\bR^n,\,E) \longrightarrow \cS(\bR^n,\,E)$,
which is given by the formula $(2 \pi)^n \cF^{-1} u = \cF u(-\,\cdot\,)$ for $u \in \cS(\bR^n,\,E)$.
The Bessel potential spaces $H^s_p$ are isomorphic to the Sobolev spaces
\begin{equation*}
	\begin{array}{rcl}
		        W^s_p(\bR^n,\,E) & := & \Big\{\,u \in L_p(\bR^n,\,E)\,:\,\partial^\alpha u \in L_p(\bR^n,\,E),\ |\alpha| \leq s\,\Big\},                                       \\[1.0em]
		\|u\|_{W^s_p(\bR^n,\,E)} & := & {\left( {\displaystyle{\sum_{|\alpha| \leq s}}} \|\partial^\alpha u\|^p_{L_p(\bR^n,\,E)} \right)}^{1/p}, \quad u \in W^s_p(\bR^n,\,E),
	\end{array}
	\qquad
	\begin{array}{c}
		s \in \bN_0, \\[0.25em]
		1 \leq p < \infty,
	\end{array}
\end{equation*}
if the Fourier multiplication operators $\partial^\alpha \cB^{- s} \in \cL(\cS(\bR^n,\,E))$
extend to bounded linear operators in $L_p(\bR^n,\,E)$ for $s \in \bN_0$ and $\alpha \in \bN^n_0$ with $|\alpha| \leq s$,
which is well-known to be true, provided that $E$ is a UMD-space; see e.\,g.\ the pertinent monograph \cite{Hytoenen-vanNeerven-Veraar-Weis:Banach-Spaces-1}
for a definition and the basic theory of UMD-Spaces and of Banach spaces that have the so-called property $(\alpha)$, which will be employed later.
Indeed, if $E$ happens to be a UMD-space, then
\begin{equation}
	\eqnlabel{HW-Isotropic}
	H^s_p(\bR^n,\,E) \doteq W^s_p(\bR^n,\,E), \qquad s \in \bN_0,\ 1 < p < \infty,
\end{equation}
which follows e.\,g.\ from \cite[Prop.~3]{Zimmermann:Fourier-Multipliers}.
Here and in the following we write $X \doteq Y$, if two Banach spaces $X$ and $Y$ are identical up to equivalence of norms.

The Sobolev scale $W^s_p$ is further extended by the Sobolev-Slobodeckij spaces
\begin{equation*}
	\begin{array}{rcl}
		        W^s_p(\bR^n,\,E) & := & \Big\{\,u \in W^{[s]}_p(\bR^n,\,E)\,:\,|u|_{\dot{W}^s_p(\bR^n,\,E)} < \infty\,\Big\},                                 \\[1.0em]
		\|u\|_{W^s_p(\bR^n,\,E)} & := & {\left( \|u\|^p_{W^{[s]}_p(\bR^n,\,E)} + |u|^p_{\dot{W}^s_p(\bR^n,\,E)} \right)}^{1/p}, \quad u \in W^s_p(\bR^n,\,E),
	\end{array}
	\qquad
	\begin{array}{c}
		s \in (0,\,\infty) \setminus \bN, \\[0.25em]
		1 \leq p < \infty,
	\end{array}
\end{equation*}
where $[s] := \max\,\{\,m \in \bN_0\,:\,m \leq s\,\}$ and
\begin{subequations}\eqnlabel{W-Norm-Isotropic}
\begin{equation}
	\eqnlabel{W-Norm-Isotropic-1}
	|u|_{\dot{W}^s_p(\bR^n,\,E)} := {\left( \sum_{|\alpha| = [s]}\ \int_{\bR^n} \int_{\bR^n} \frac{\|\partial^\alpha u(x) - \partial^\alpha u(y)\|^p_E}{|x - y|^{n + (s - [s])p}}\,\mbox{d}x\,\mbox{d}y \right)}^{1/p},
	\qquad u \in W^s_p(\bR^n,\,E),
\end{equation}
for all $0 < s < \infty$ with $s \notin \bN$ and $1 \leq p < \infty$.
It is well-known that the norms
\begin{equation*}
	\Big( \|u\|^p_{L_p(\bR^n,\,E)} + |u|^p_{\dot{W}^s_p(\bR^n,\,E)} \Big)^{1/p}, \qquad \Big( \|u\|^p_{L_p(\bR^n,\,E)} + [u]^p_{\dot{W}^s_p(\bR^n,\,E)} \Big)^{1/p},
	\qquad u \in \cS^\prime(\bR^n,\,E),
\end{equation*}
are equivalent to $\|\,\!\cdot\,\!\|_{W^s_p(\bR^n,\,E)}$, where
\begin{equation}
	\eqnlabel{W-Norm-Isotropic-2}
	[u]_{\dot{W}^s_p(\bR^n,\,E)} := {\left( \sum_{|\alpha| = [s]} \Big\| |h|^{[s] - s} \|\Delta^h \partial^\alpha u\|_{L_p(\bR^n,\,E)} \Big\|^p_{L_p(\bR^n,\,|h|^{-n} \textrm{d}h)} \right)}^{1/p},
	\qquad u \in W^s_p(\bR^n,\,E),
\end{equation}
\end{subequations}
for all $0 < s < \infty$ with $s \notin \bN$ and $1 \leq p < \infty$.
Here, we employ the convention to denote by $\Delta^h := \tau^h - 1 \in \cL(\cS(\bR^n,\,E)) \cap \cL(\cS^\prime(\bR^n,\,E))$ the difference operators
that are defined based on the translations $\tau^h \in \cL(\cS(\bR^n,\,E)) \cap \cL(\cS^\prime(\bR^n,\,E))$ given as
\begin{equation*}
	\tau^h u := u(\,\cdot + h), \quad u \in \cS(\bR^n,\,E), \qquad
	\langle \varphi,\,\tau^h u \rangle := \langle \tau^{-h} \varphi,\,u \rangle, \quad \varphi \in \cS(\bR^n),\ u \in \cS^\prime(\bR^n,\,E),
\end{equation*}
for $h \in \bR^n$.
Note that Fubini's theorem implies $|u|_{\dot{W}^s_p(\bR^n,\,E)} = [u]_{\dot{W}^s_p(\bR^n,\,E)}$ for $u \in \cS(\bR^n,\,E)$,
i.\,e.\ in order to obtain the above claimed equivalence of norms it is sufficient to estimate the intermediate derivatives $\partial^\alpha u$
with $0 < |\alpha| \leq [s]$ in $L_p(\bR^n,\,E)$ for $u \in \cS(\bR^n,\,E)$ via interpolation inequalities.

Finally, the Besov scale, which is usually defined based on dyadic spectral decompositions, may equivalently be defined as
\begin{equation*}
	B^s_{p, q}(\bR^n,\,E) := \big( H^{s - \epsilon}_p(\bR^n,\,E),\ H^{s + \epsilon}_p(\bR^n,\,E) \big)_{1/2, q}, \qquad
	\begin{array}{c}
		- \infty < s < \infty,\ \epsilon > 0, \\[0.25em] 1 < p < \infty,\ 1 \leq q \leq \infty,
	\end{array}
\end{equation*}
where we denote by $(\,\cdot\,,\,\cdot\,)_{\theta,\,\cdot}$ the family of real interpolation functors,
cf.~\cite[Sections~1.3--1.8]{Triebel:Interpolation}.
Concerning the equivalence of the usual and our definition we refer to
\cite[Remark~2.4.2/4]{Triebel:Interpolation} for the scalar-valued case and to \cite[Theorem~3.7.1~(iv)]{Amann:Maximal-Regularity} for the (anisotropic) vector-valued case.
With the abbreviation $B^s_p := B^s_{p, p}$ it is well-known that
\begin{equation}
	\eqnlabel{BW-Isotropic}
	B^s_p(\bR^n,\,E) \doteq W^s_p(\bR^n,\,E), \qquad s \in (0,\,\infty) \setminus \bN,\ 1 < p < \infty,
\end{equation}
cf.~\cite[Eq.~(5.8)]{Amann:Vector-Valued-Besov-Spaces}.
In particular, the semi-norms \eqnref{W-Norm-Isotropic} may be used to define an equivalent norm
for the vector-valued Besov scale $B^s_p$ for $0 < s < \infty$ with $s \notin \bN$.

\subsection{Multiplication in Isotropic Function Spaces}\subseclabel{Isotropic-Multiplication}
Concerning multiplication in the isotropic scales $B^s_p$, $H^s_p$, and $W^s_p$ the following well-known theorem,
which is essentially proved in \cite{Amann:Multiplication, Amann:Maximal-Regularity}, often serves as a starting point.
\begin{theorem}
	\thmlabel{Algebra-Isotropic}
	Let $n \in \bN$ and let $E$ be a UMD-space.
	Moreover, let $X \in \{\,B,\,H\,\}$, let $0 < s < \infty$, and let $1 < p < \infty$ such that
	\begin{equation*}
		\mbox{\upshape ind}(X^s_p(\bR^n,\,E)) > 0.
	\end{equation*}
	Then we have
	\begin{enumerate}[(a)]
		\item $X^s_p(\bR^n,\,E) \hookrightarrow C_0(\bR^n,\,E)$;
		\item $X^s_p(\bR^n,\,E)$ is a multiplication algebra, if $E$ is a Banach algebra, and $s \in \bN$ for $X = H$. \qeddiamond
	\end{enumerate}
\end{theorem}
Hence, the basic properties of the spaces $B^s_{p, q}$, $H^s_p$, and $W^s_p$ are closely related to their so-called Sobolev index
\begin{equation}\label{def_iso_sop_ind}
	s - n/p =: \left\{
	\begin{array}{lcc}
		\mbox{ind}(B^s_{p, q}(\bR^n,\,E)), & \qquad & -\infty < s < \infty,\ 1 < p < \infty, 1 \leq q \leq \infty, \\[0.5em]
		\mbox{ind}(H^s_p(\bR^n,\,E)),      & \qquad & -\infty < s < \infty,\ 1 < p < \infty,                       \\[0.5em]
		\mbox{ind}(W^s_p(\bR^n,\,E)),      & \qquad & 0 \leq s < \infty,\ 1 \leq p < \infty.
	\end{array}
	\right.
\end{equation}
Indeed, as has been proved in \cite[Theorem 4.6.4/1]{Runst-Sickel:Function-Spaces} the condition $\mbox{ind}(X^s_p(\bR^n)) > 0$
is even equivalent to both assertions of \Thmref{Algebra-Isotropic} for a large class of scalar-valued Besov and Triebel-Lizorkin spaces.
In the vector-valued case \Thmref{Algebra-Isotropic}~(a) follows from \cite[Theorem~3.9.1]{Amann:Maximal-Regularity};
note that the case $X = W$ may then be obtained with the aid of \eqnref{HW-Isotropic} and \eqnref{BW-Isotropic}.
\Thmref{Algebra-Isotropic}~(b) is a consequence of \cite[Theorems~2.1~\&~4.1]{Amann:Multiplication} (which also include the case $X = W$) and \eqnref{HW-Isotropic},
cf.~also \cite[Remarks~2.2~(a)~\&~4.2~(a)]{Amann:Multiplication}.
On the other hand, \Thmref{Algebra-Isotropic}~(b) is a direct consequence of \Thmref{Algebra-Anisotropic}~(b) below.

Note that the space
\begin{equation*}
	\begin{array}{rcl}
		        C_0(\bR^n,\,E) & := & \Big\{\,u \in C(\bR^n,\,E)\,:\,{\displaystyle{\lim_{\rho \rightarrow \infty} \sup_{|x| > \rho} \|u(x)\|_E}} = 0\,\Big\}, \\[1.0em]
		\|u\|_{C_0(\bR^n,\,E)} & := & {\displaystyle{\sup_{x \in \bR^n} \|u(x)\|_E}}, \qquad u \in C_0(\bR^n,\,E),
	\end{array}
\end{equation*}
is a closed subspace of the space $BUC(\bR^n,\,E)$ of bounded, uniformly continuous functions.
To be precise, it coincides with the closure of the space $\cS(\bR^n,\,E)$ in $BUC(\bR^n,\,E)$ as well as in $L_\infty(\bR^n,\,E)$,
see also \cite[Section 3.9]{Amann:Maximal-Regularity}.

A more complex situation arises, if a finite family of Banach spaces $E_1,\,\dots,\,E_m$ is given along with a continuous, $m$-linear map
\begin{equation}
	\eqnlabel{Multiplication}
	\bullet: E_1 \times \dots \times E_m \longrightarrow E,
\end{equation}
which may be interpreted as a multiplication of the elements of the $E_j$ with values in a Banach space $E$.
Such a multiplication induces a continuous, $m$-linear map
\begin{equation}
	\eqnlabel{Multiplication-S}
	\bullet: \cS(\bR^n,\,E_1) \times \dots \times \cS(\bR^n,\,E_m) \longrightarrow \cS(\bR^n,\,E)
\end{equation}
in a canonical way and it is a natural question whether this last multiplication extends to a continuous, $m$-linear map between suitable Besov, Bessel potential, and/or Sobolev-Slobodeckij spaces.
A comprehensive answer to this question is given by the following theorem.
\begin{theorem}
	\thmlabel{Multiplication-Isotropic}
	Let $m,\,n \in \bN$ and let $E_1,\,\dots,\,E_m$ and $E$ be UMD-spaces that allow for a multiplication \eqnref{Multiplication}.
	Moreover, let $X_1,\,\dots,\,X_m,\,X \in \{\,B,\,H\,\}$, let $0 \leq s_1,\,\dots,\,s_m,\,s < \infty$, and let $1 < p_1,\,\dots,\,p_m,\,p < \infty$ with
	\begin{equation*}
		\textrm{(i)} \quad s \leq \min\,\{\,s_1,\,\dots,\,s_m\,\} \qquad \mbox{and} \qquad
		\textrm{(ii)} \quad {\displaystyle{\frac{1}{p} \leq \sum^m_{j = 1} \frac{1}{p_j}}}.
	\end{equation*}
	Furthermore, let
	\begin{equation*}
		\mbox{\upshape ind}_j := \mbox{\upshape ind}([X_j]^{s_j}_{p_j}(\bR^n,\,E_j)), \quad j = 1,\,\dots,\,m, \qquad \mbox{and} \qquad \mbox{\upshape ind} := \mbox{\upshape ind}(X^s_p(\bR^n,\,E)),
	\end{equation*}
	and assume
	\begin{equation*}
		\textrm{(iii)} \quad \mbox{\upshape ind} \leq \left\{
		\begin{array}{ll}
			\min\,\{\,\mbox{\upshape ind}_1,\,\dots,\,\mbox{\upshape ind}_m\,\},                             & \qquad \mbox{\upshape ind}_1,\,\dots,\,\mbox{\upshape ind}_m \geq 0, \\[0.5em]
			{\displaystyle{\sum^m_{\substack{j = 1 \\ \textrm{\upshape ind}_j < 0}}}} \mbox{\upshape ind}_j, & \qquad \mbox{otherwise}.
		\end{array}
		\right.
	\end{equation*}
	Finally, assume that
	\begin{itemize}
		\item[{\itshape (a)}] $s_j > s$ for all $j \in \{\,1,\,\dots,\,m\,\}$ for which $X_j \neq X$;
		\item[{\itshape (b)}] in case $X = B$: $s > 0$ and $p_j = p$ for all $j \in \{\,1,\,\dots,\,m\,\}$ for which $s_j = s$;
		\item[{\itshape (c)}] in case $X = B$: the inequality (iii) is strict or $\max\,\{\,p_1,\,\dots,\,p_m\,\} \leq p$;
		\item[{\itshape (d)}] in case $X = H$: $s \in \bN_0$, or inequality (i) is strict, or there is equality in (ii);
		\item[{\itshape (e)}] at least one of the inequalities (ii), (iii) is strict, if $X_j \neq X$ for some $j \in \{\,1,\,\dots,\,m\,\}$;
		\item[{\itshape (f)}] the inequality (iii) is strict, if $\textrm{\upshape ind}_j = 0$ for some $j \in \{\,1,\,\dots,\,m\,\}$.
	\end{itemize}
	Then the multiplication \eqnref{Multiplication-S} extends to a continuous, $m$-linear operator
	\begin{equation*}
		\bullet: [X_1]^{s_1}_{p_1}(\bR^n,\,E_1) \times \dots \times [X_1]^{s_m}_{p_m}(\bR^n,\,E_m) \longrightarrow X^s_p(\bR^n,\,E)
	\end{equation*}
	in a unique way. \qeddiamond
\end{theorem}
Again, \Thmref{Multiplication-Isotropic} is known to a large extent.
Scalar-valued (anisotropic) Besov and Triebel-Lizorkin spaces for $m = 2$ are considered in \cite{Johnsen:Multiplication}.
The results presented there even characterize the continuity of the induced multiplication in terms of the orders of differentiability,
integrability, and in terms of the Sobolev indices of the involved function spaces.
Thus, these results are on the one hand more general than \Thmref{Multiplication-Isotropic} (note that the scalar-valued Bessel potential scale is included in the scalar-valued Triebel-Lizorkin scale);
on the other hand, vector-valued spaces are not considered in \cite{Johnsen:Multiplication}.
Moreover, for scalar-valued function spaces \Thmref{Multiplication-Isotropic} is contained in \cite[Theorem~4.5.2 \& Corollary~4.5.2]{Runst-Sickel:Function-Spaces}.
Vector-valued Sobolev and Besov spaces are considered in \cite{Amann:Multiplication}.
The results \cite[Theorems~2.1~\&~4.1]{Amann:Multiplication} for $X_1 = \ldots = X_m = X \in \{\,B,\,W\,\}$ together with \eqnref{HW-Isotropic} already provide a large part
of the assertions of \Thmref{Multiplication-Isotropic} in these cases, where the formulation for Besov spaces $B^s_{p, q}$
additionally allows for a constrained choice of parameters $1 \leq q_1,\,\dots,\,q_m,\,q \leq \infty$.
Thus, these results are on the one hand again more general than \Thmref{Multiplication-Isotropic};
on the other hand, Bessel potential spaces are not considered in \cite{Amann:Multiplication}
and the constraints on the parameters $q_1,\,\dots,\,q_m,\,q$ in \cite[Theorem~4.1]{Amann:Multiplication}
are not optimal in the UMD-space setting with the consequence that certain borderline cases of \Thmref{Multiplication-Isotropic} are excluded;
cf.~also \Remref{Multiplication-Anisotropic} below.
The main reason for this is that Sobolev type embeddings for Besov spaces can be used in the UMD-space setting
that seem to be unavailable for general Banach spaces.
See also \cite[Sec.~VII.6]{Amann:Parabolic-Problems-2} for recent results for (bilinear) multiplication between (anisotropic) Sobolev and Besov spaces
with values in general Banach spaces.
We do not include a proof of \Thmref{Multiplication-Isotropic} here,
since it follows as a direct consequence of \Thmref{Multiplication-Anisotropic} below.
Note that one can directly obtain similar results for $X_j = W$ for some/all $j \in \{\,1,\,\dots,\,m\,\}$ and/or $X = W$
thanks to \eqnref{HW-Isotropic} and \eqnref{BW-Isotropic}.
Also note the remarks given for \Thmref{Multiplication-Anisotropic}, cf.~\Remref{Multiplication-Anisotropic},
which also apply to \Thmref{Multiplication-Isotropic} and which include further discussions on the constraints (a)--(f).

\subsection{Anisotropic Function Spaces}\subseclabel{Anisotropic-Spaces}
In order to generalize \Thmref{Algebra-Isotropic} and \Thmref{Multiplication-Isotropic}
to the case of anisotropic Besov, Bessel potential and Sobolev-Slobodeckij spaces, we fix the following notation.
We denote by $\nu \in \bN$ the number of slices in which the Euclidean space is divided to allow for different regularities in space.
Moreover, we denote by $n = (n_1,\,\dots,\,n_\nu) \in \bN^\nu$ the dimensions of the slices and use the abbreviation
\begin{equation*}
	\bR^n := \bR^{n_1} \times \dots \times \bR^{n_\nu}.
\end{equation*}
By $\omega=(\omega_1,\ldots,\omega_\nu) \in \bN^\nu$ we denote 
an arbitrary weight vector and we define
the vector-valued anisotropic Bessel potential spaces as
\begin{equation*}
	\begin{array}{rcl}
		        H^{s,\omega}_p(\bR^n,\,E) & := & \Big\{\,u \in \cS^\prime(\bR^n,\,E)\,:\,u = \cB^{- s, \omega} f,\ f \in L_p(\bR^n,\,E)\,\Big\},                  \\[1.0em]
		\|u\|_{H^{s,\omega}_p(\bR^n,\,E)} & := & \|f\|_{L_p(\bR^n,\,E)}, \quad u \in H^{s, \omega}_p(\bR^n,\,E),\ f \in L_p(\bR^n,\,E),\ u = \cB^{- s, \omega} f,
	\end{array}
	\ \ 
	\begin{array}{c}
		-\infty < s < \infty, \\[0.25em]
		1 < p < \infty,
	\end{array}
\end{equation*}
where the space $E$ may be an arbitrary Banach space.
Here the anisotropic Bessel potentials are given as
\begin{equation*}
	\begin{array}{rclrcl}
		 \cB^{s, \omega} u & := & \cF^{-1} (\xi \mapsto B^{s,\omega}(\xi) \cF u(\xi)), & \qquad u   & \in & \cS(\bR^n,\,E),                                                 \\[0.5em]
		B^{s, \omega}(\xi) & := & \left(1 + {\displaystyle{\sum^\nu_{k = 1} |\xi_k|^{2 \dot{\omega} / \omega_k}}} \right)^{s / 2 \dot{\omega}}, & \qquad \xi & \in & \bR^n,
	\end{array}
	\qquad -\infty < s < \infty,
\end{equation*}
where we use the notation $\xi = (\xi_1,\,\dots,\,\xi_\nu) \in \bR^{n_1} \times \dots \times \bR^{n_\nu} = \bR^n$.
Moreover,
\begin{equation*}
	\dot{\omega} := \mbox{lcm}\,\{\,\omega_1,\,\dots,\,\omega_\nu\,\}
\end{equation*}
denotes the least common multiple of the weight entries 
$\omega_1,\,\dots,\,\omega_\nu$.
These definitions coincide with those used in \cite[Section~3.7]{Amann:Maximal-Regularity} and \cite[Section~VII.4.1]{Amann:Parabolic-Problems-2}.
Note, however, that the terminology $X^{s, \omega}_p$ for anisotropic spaces is used e.\,g.\ in \cite{Triebel:Function-Spaces-3},
whereas in \cite{Amann:Maximal-Regularity} the notation $X^{s / \omega}_p$ is employed.
As in the isotropic case we have $\cB^{s, \omega} \in \cL\mbox{aut}(\cS(\bR^n,\,E)) \cap \cL\mbox{aut}(\cS^\prime(\bR^n,\,E))$,
cf.~\cite[Section 2.3]{Amann:Maximal-Regularity}.

It is not surprising that the anisotropic Bessel potential spaces
$H^{s,\omega}_p$ are related to the anisotropic Sobolev spaces
\begin{equation*}
	\begin{array}{rcl}
		        W^{s,\omega}_p(\bR^n,\,E) & := & \left\{\,u \in L_p(\bR^n,\,X)\,: \begin{array}{c} \partial^\alpha_k u \in L_p(\bR^n,\,E),\ \alpha \in \bN^{n_k}_0 \\[0.5em] |\alpha| \leq s / \omega_k,\ k = 1,\,\dots,\,\nu \end{array} \right\}, \\[2.0em]
		\|u\|_{W^{s,\omega}_p(\bR^n,\,E)} & := & {\left( {\displaystyle{\sum^\nu_{k = 1} \sum_{|\alpha| \leq s / \omega_k}}} \|\partial^\alpha_k u\|^p_{L_p(\bR^n,\,E)} \right)}^{1/p}, \qquad u \in W^{s,\omega}_p(\bR^n,\,E),
	\end{array}
	\qquad
	\begin{array}{c}
		s \in \dot{\omega} \cdot \bN_0, \\[0.25em]
		1 \leq p < \infty,
	\end{array}
\end{equation*}
if the Fourier multiplication operators $\partial^\alpha_k \cB^{- s, \omega} \in \cL(\cS(\bR^n,\,E))$
extend to bounded linear operators in $L_p(\bR^n,\,E)$ for $s \in \dot{\omega} \cdot \bN_0$
and $\alpha \in \bN^{n_k}_0$ with $|\alpha| \leq s / \omega_k$ for all $k = 1,\,\dots,\,\nu$.
This is true in the case that $E$ is a UMD-space that has property $(\alpha)$ if $\omega \neq \dot{\omega} \cdot (1,\,\dots,\,1)$,
as follows e.\,g.\ from \cite[Prop.~3]{Zimmermann:Fourier-Multipliers}.
As usual we denote by $\partial^\alpha_k = \partial^{|\alpha|} / \partial x^{\alpha}_k$ the partial derivative
w.\,r.\,t.\ the $k$-th component $x_k \in \bR^{n_k}$ of $x = (x_1,\,\dots,\,x_\nu) \in \bR^n$.
Note that $s / \omega_k \in \bN$ for $s \in \dot{\omega} \cdot \bN$ and all $k = 1,\,\dots,\,\nu$
and that $H^{0, \omega}_p = W^{0, \omega}_p = L_p$ by definition, independent of the properties of the Banach space $E$.
Furthermore, \cite[Theorem 3.7.1 (ii)]{Amann:Maximal-Regularity} guarantees the following identification to be valid,
which is the analogue of \eqnref{HW-Isotropic} in the anisotropic case.
If $E$ is a UMD-space, that has property $(\alpha)$ if $\omega \neq \dot{\omega} \cdot (1,\,\dots,\,1)$, then
\begin{equation}
	\eqnlabel{HW-Anisotropic}
	H^{s, \omega}_p(\bR^n,\,E) \doteq W^{s,\omega}_p(\bR^n,\,E), \qquad s \in \dot{\omega} \cdot \bN_0,\ 1 < p < \infty.
\end{equation}
Now, in applications the an\-iso\-tro\-pic spaces $H^{s, \omega}_p$ and $W^{s, \omega}_p$ frequently appear
in form of an intersection of isotropic vector-valued spaces, see \Secref{Applications}. 
Therefore the following characterization is important, since it shows that both representations are equivalent.
\begin{proposition}
	\proplabel{Characterization-Anisotropic}
	Let $\nu \in \bN$ and $n,\,\omega \in \bN^\nu$.
	Moreover, let $E$ be a UMD-space, that has property~$(\alpha)$ if $\omega \neq \dot{\omega} \cdot (1,\,\dots,\,1)$.
	Then the characterizations
	\begin{equation*}
		\begin{array}{rcl}
			H^{s, \omega}_p(\bR^n,\,E) & = & H^{s / \omega_1}_p(\bR^{n_1},\,L_p(\bR^{n^\prime_1},\,E)) \cap L_p(\bR^{n_1},\,H^{s, \omega^\prime_1}_p(\bR^{n^\prime_1},\,E))                                  \\[0.5em]
			                           & = & {\displaystyle{\bigcap^\nu_{k = 1}}} H^{s / \omega_k}_p(\bR^{n_k},\,L_p(\bR^{n^\prime_k},\,E)), \qquad \qquad 0 < s < \infty,\ 1 < p < \infty,                  \\[1.5em]
			W^{s, \omega}_p(\bR^n,\,E) & = & W^{s/\omega_1}_p(\bR^{n_1},\,L_p(\bR^{n^\prime_1},\,E)) \cap L_p(\bR^{n_1},\,W^{s, \omega^\prime_1}_p(\bR^{n'_1},\,E))                                          \\[0.5em]
			                           & = & {\displaystyle{\bigcap^\nu_{k = 1}}} W^{s / \omega_k}_p(\bR^{n_k},\,L_p(\bR^{n^\prime_k},\,E)), \qquad \qquad s \in \dot{\omega} \cdot \bN,\ 1 \leq p < \infty,
		\end{array}
	\end{equation*}
	are valid. \qeddiamond
\end{proposition}
Of course, here
\begin{equation*}
	\begin{array}{rclcl}
		     n^\prime_k & := & (n_1,\,\dots,\,n_{k - 1},\,n_{k + 1},\,\dots,\,n_\nu)                     & \in & \bN^{\nu - 1}, \\[0.5em]
		\omega^\prime_k & := & (\omega_1,\,\dots,\,\omega_{k - 1},\,\omega_{k + 1},\,\dots,\,\omega_\nu) & \in & \bN^{\nu - 1},
	\end{array}
	\qquad \qquad k = 1,\,\dots,\,\nu, \quad n,\,\omega \in \bN^\nu.
\end{equation*}
The above characterizations could also serve as (recursive) definitions of the anisotropic function spaces.
Note that the second characterizations in \Propref{Characterization-Anisotropic} implicitly employ Fubini's theorem
to obtain a suitable rearrangement of the slices in the product $\bR^n$, see also \cite[Remark 3.6.2]{Amann:Maximal-Regularity}.
\Propref{Characterization-Anisotropic} for $1 < p < \infty$ is a special case of \cite[Theorems 3.7.2 \& 3.7.3]{Amann:Maximal-Regularity}
in conjunction with \eqnref{HW-Isotropic} and \eqnref{HW-Anisotropic}.
Moreover, the characterizations of the spaces $W^{s, \omega}_p(\bR^n,\,X)$ for $s \in \dot{\omega} \cdot \bN$ and $1 \leq p < \infty$
may also be easily obtained by Fubini's theorem.
Thus, we do not include a detailed proof of \Propref{Characterization-Anisotropic}.
Note that the spaces $H^{s / \omega_k}_p(\bR^{n_k},\,\dots)$, which appear in \Propref{Characterization-Anisotropic},
are the isotropic vector-valued Bessel potential spaces on the slice $\bR^{n_k}$ and the spaces $W^{s / \omega_k}_p(\bR^{n_k},\,\dots)$
are the isotropic vector-valued Sobolev spaces, respectively, cf.~\Subsecref{Isotropic-Spaces}.

\Propref{Characterization-Anisotropic} motivates the extension of the anisotropic Sobolev scale
by suitable anisotropic Sobolev-Slobodeckij spaces via
\begin{equation*}
	\begin{array}{rrl}
		W^{s, \omega}_p(\bR^n,\,E) & := & W^{s / \omega_1}_p(\bR^{n_1},\,L_p(\bR^{n^\prime_1},\,E)) \cap L_p(\bR^{n_1},\,W^{s,\omega^\prime_1}_p(\bR^{n^\prime_1},\,E)) \\[0.5em]
		                           &  = & {\displaystyle{\bigcap^\nu_{k = 1}}} W^{s / \omega_k}_p(\bR^{n_k},\,L_p(\bR^{n^\prime_k},\,X)), \qquad \qquad 0 < s < \infty,\ 1 \leq p < \infty.
	\end{array}
\end{equation*}
This is consistent with the above definition of the anisotropic Sobolev spaces $W^{s, \omega}_p$ for \mbox{$s \in \dot{\omega} \cdot \bN$},
which follows from \Propref{Characterization-Anisotropic}.
Moreover, both representations are indeed equivalent thanks to \cite[Theorem 3.8.5]{Amann:Maximal-Regularity},
provided that $E$ is a UMD-space, that has property~$(\alpha)$ if $\omega \neq \dot{\omega} \cdot (1,\,\dots,\,1)$.
Based on the second representation above and the considerations in \Subsecref{Isotropic-Spaces} we infer that
\begin{equation*}
	\Big( \|u\|^p_{L_p(\bR^n,\,E)} + [u]^p_{\dot{W}^{s, \omega}_p(\bR^n,\,E)} \Big)^{1/p}, \qquad u \in \cS^\prime(\bR^n,\,E),
\end{equation*}
defines an equivalent norm on $W^{s, \omega}_p(\bR^n,\,E)$ for $0 < s < \infty$ with $s / \omega_k \notin \bN$ for all $k = 1,\,\dots,\,\nu$
and for $1 \leq p < \infty$, where
\begin{equation}
	\eqnlabel{W-Norm-Anisotropic}
	[u]_{\dot{W}^{s, \omega}_p(\bR^n,\,E)} := {\left( \sum^\nu_{k = 1} \sum_{|\alpha| = [s / \omega_k]} \Big\| |h|^{|\alpha| - s / \omega_k} \|\Delta^h_k \partial^\alpha_k u\|_{L_p(\bR^n,\,E)} \Big\|^p_{L_p(\bR^{n_k},\,|h|^{-n_k} \textrm{d}h)} \right)}^{1/p}
\end{equation}
for $u \in W^{s, \omega}_p(\bR^n,\,E)$.
Here, we denote by $\Delta^h_k := \tau^h_k - 1 \in \cL(\cS(\bR^n,\,E)) \cap \cL(\cS^\prime(\bR^n,\,E))$
the anisotropic difference operators, where the translations
$\tau^h_k \in \cL(\cS(\bR^n,\,E)) \cap \cL(\cS^\prime(\bR^n,\,E))$ are given as
\begin{equation*}
	\begin{array}{ll}
		\tau^h_k u := u(x_1,\,\dots,\,x_{k - 1},\,x_k + h,\,x_{k + 1},\,\dots,\,x_\nu),  & \qquad u \in \cS(\bR^n,\,E), \\[0.5em]
		\langle \varphi,\,\tau^h_k u \rangle := \langle \tau^{-h}_k \varphi,\,u \rangle, & \qquad \varphi \in \cS(\bR^n),\ u \in \cS^\prime(\bR^n,\,E),
	\end{array}
\end{equation*}
for $h \in \bR^{n_k}$ and $k = 1,\,\dots,\,\nu$.

Finally, the anisotropic Besov scale, which is usually defined based on anisotropic dyadic spectral decompositions, may equivalently be defined as
\begin{equation*}
	B^{s, \omega}_{p, q}(\bR^n,\,E) := \big( H^{s - \epsilon, \omega}_p(\bR^n,\,E),\ H^{s + \epsilon, \omega}_p(\bR^n,\,E) \big)_{1/2, q}, \qquad
	\begin{array}{c}
		- \infty < s < \infty,\ \epsilon > 0, \\[0.25em] 1 < p < \infty,\ 1 \leq q \leq \infty.
	\end{array}
\end{equation*}
Concerning the equivalence of the usual and our definition we refer to \cite[Theorem~3.7.1~(iv)]{Amann:Maximal-Regularity}.
Now, with the abbreviation $B^{s, \omega}_p := B^{s, \omega}_{p, p}$ there is an analog of \Propref{Characterization-Anisotropic} for anisotropic Besov spaces,
which reads as follows.
\begin{proposition}
	\proplabel{Characterization-Anisotropic-Besov}
	Let $\nu \in \bN$ and $n,\,\omega \in \bN^\nu$.
	Moreover, let $E$ a Banach space.
	Then the characterization
	\begin{equation*}
		\begin{array}{rcl}
			B^{s, \omega}_p(\bR^n,\,E) & = & B^{s / \omega_1}_p(\bR^{n_1},\,L_p(\bR^{n^\prime_1},\,E)) \cap L_p(\bR^{n_1},\,B^{s, \omega^\prime_1}_p(\bR^{n^\prime_1},\,E))                                  \\[0.5em]
			                           & = & {\displaystyle{\bigcap^\nu_{k = 1}}} B^{s / \omega_k}_p(\bR^{n_k},\,L_p(\bR^{n^\prime_k},\,E)), \qquad \qquad 0 < s < \infty,\ 1 < p < \infty,
		\end{array}
	\end{equation*}
	is valid. \qeddiamond
\end{proposition}
This follows from \cite[Theorems~3.6.3~\&~3.6.7]{Amann:Maximal-Regularity} and allows for an analog of \eqnref{BW-Isotropic} for anisotropic spaces.
Indeed, the second representation above and \eqnref{BW-Isotropic} imply that
\begin{equation}
	\eqnlabel{BW-Anisotropic}
	B^{s, \omega}_p(\bR^n,\,E) \doteq W^{s, \omega}_p(\bR^n,\,E), \qquad 0 < s < \infty\ \textrm{with}\ s / \omega_k \notin \bN\ \textrm{for}\ k = 1,\,\dots,\,\nu,\ 1 < p < \infty.
\end{equation}
In particular, the semi-norm \eqnref{W-Norm-Anisotropic} may be used to define an equivalent norm
for the vector-valued Besov scale $B^{s, \omega}_p$ for $0 < s < \infty$ with $s / \omega_k \notin \bN$ for all $k = 1,\,\dots,\,\nu$.
Note that \Propref{Characterization-Anisotropic-Besov} and its proof employ Fubini's theorem.
Therefore, similar characterizations are not obvious for the spaces $B^{s, \omega}_{p, q}$ with $p \neq q$.
However, \cite[Theorem~3.6.1]{Amann:Maximal-Regularity} shows that
\begin{equation*}
	\Big( \|u\|^p_{L_p(\bR^n,\,E)} + [u]^p_{\dot{B}^{s, \omega}_{p, q}(\bR^n,\,E)} \Big)^{1/p}, \qquad u \in \cS^\prime(\bR^n,\,E),
\end{equation*}
defines an equivalent norm on $B^{s, \omega}_{p, q}(\bR^n,\,E)$ for all $0 < s < \infty$ and for all $1 \leq p,\,q < \infty$, where
\begin{equation}
	\eqnlabel{B-Norm-Anisotropic}
	[u]_{\dot{B}^{s, \omega}_{p, q}(\bR^n,\,E)} := {\left( \sum^\nu_{k = 1} \Big\| |h|^{- s / \omega_k} \|(\Delta^h_k)^{[s / \omega_k] + 1} u\|_{L_p(\bR^n,\,E)} \Big\|^q_{L_q(\bR^{n_k},\,|h|^{-n_k} \textrm{d}h)} \right)}^{1/q}
\end{equation}
for $u \in B^{s, \omega}_{p, q}(\bR^n,\,E)$.
This implies at least that
\begin{equation}
	\eqnlabel{Besov-Embedding-Partial}
	B^{s, \omega}_{p, q}(\bR^n,\,E) \hookrightarrow B^{s / \omega_k}_{p, q}(\bR^{n_k},\,L_p(\bR^{n^\prime_k},\,E)),
	\qquad 0 < s < \infty,\ 1 \leq p,\,q < \infty,\ k = 1,\,\dots,\,\nu,
\end{equation}
i.\,e.\ one of the inclusions of the second characterization in \Propref{Characterization-Anisotropic-Besov}
is also valid for the spaces $B^{s, \omega}_{p, q}$ with $p \neq q$.

Now, a careful inspection of the regularity of the functions belonging to one of the anisotropic spaces
$B^{s, \omega}_{p, q}$, $H^{s,\omega}_p$, or $W^{s,\omega}_p$ reveals that the definition
\begin{equation}
	\eqnlabel{Anisotropic-Index}
	\frac{1}{\dot{\omega}} \left( s - \frac{\omega \cdot n}{p} \right) =: \left\{
	\begin{array}{rcc}
		\mbox{ind}(B^{s, \omega}_{p, q}(\bR^n,\,E)), & \qquad & -\infty < s < \infty,\ 1 < p < \infty,\ 1 \leq q \leq \infty, \\[0.5em]
		\mbox{ind}(H^{s, \omega}_p(\bR^n,\,E)),      & \qquad & -\infty < s < \infty,\ 1 < p < \infty,                        \\[0.5em]
		\mbox{ind}(W^{s, \omega}_p(\bR^n,\,E)),      & \qquad & 0 \leq s < \infty,\ 1 \leq p < \infty
	\end{array}
	\right.
\end{equation}
is the correct generalization of the isotropic Sobolev index \eqref{def_iso_sop_ind} to the anisotropic setting.
Here,
\begin{equation*}
	\omega \cdot n := \sum^\nu_{k = 1} \omega_k\,n_k, \qquad n,\,\omega \in \bN^\nu.
\end{equation*}
Note that in the isotropic case, i.\,e.\ for anisotropies $\omega = \dot{\omega} \cdot (1,\,\dots,\,1)$ with $\dot{\omega} \in \bN$, we have
\begin{equation*}
	\mbox{ind}(X^{s, \omega}_p(\bR^n,\,E)) = \frac{1}{\dot{\omega}} \left( s - \frac{\omega \cdot n}{p} \right)
		= s / \dot{\omega} - \frac{|n|}{p} = \mbox{ind}(X^{s / \dot{\omega}}_p(\bR^n,\,E)),
\end{equation*}
where $|n| = n_1 + \dots + n_\nu = \dim\,\bR^n$.
Hence, the definition of the anisotropic index is consistent with its isotropic version.
In Appendix~A we collect numerous Sobolev type embeddings and interpolation identities
for anisotropic function spaces, which we employ in the proofs of our main theorems, and,
which again show that the above definition of the anisotropic Sobolev index is the correct adaption to the anisotropic setting.
It is worthwhile to mention that embeddings in the anisotropic setting into $L^p(\bR^n,E)$ require to
regard this space as a corresponding anisotropic space.
To be precise, in case an anisotropy with weight $\omega$ is given, one has to use
\begin{equation*}
	\mbox{ind}_\omega(L_p(\bR^n,\,E)) := - \frac{\omega \cdot n}{\dot{\omega}} \frac{1}{p}
		\neq - \frac{|n|}{p} = \textrm{ind}(L_p(\bR^n,\,E)), \quad 1 \leq p < \infty,\ \omega \neq \dot{\omega} \cdot (1,\,\dots,\,1),
\end{equation*}
for the spaces $H^{0, \omega}_p(\bR^n,\,E) = W^{0, \omega}_p(\bR^n,\,E) = L_p(\bR^n,\,E)$.
This also applies to Sobolev type embeddings in the anisotropic setting,
cf.~\eqnref{Bessel-Potential-Lebesgue-Embedding} and \eqnref{Besov-Lebesgue-Embedding}.

We want to remark that there are several alternative approaches
to define anisotropic function spaces of Besov, Bessel potential, Sobolev-Slobodeckij or Triebel-Lizorkin type
that are in use in the literature.
For instance, one can allow for $\omega \in [1,\,\infty)^\nu$ as in \cite{Johnsen:Multiplication} instead of $\omega \in \bN^\nu$.

\subsection{Multiplication in Anisotropic Function Spaces}\subseclabel{Anisotropic-Multiplication}
Utilizing the anisotropic So\-bo\-lev index defined above we obtain the analog of \Thmref{Algebra-Isotropic}.
This also confirms that \eqnref{Anisotropic-Index} seems to be the correct generalization of the Sobolev index.
\begin{theorem}
	\thmlabel{Algebra-Anisotropic}
	Let $\nu \in \bN$ and $n,\,\omega \in \bN^\nu$.
	Moreover, let $E$ be a UMD-space, that has property~$(\alpha)$ if $\omega \neq \dot{\omega} \cdot (1,\,\dots,\,1)$.
	Let $X \in \{\,B,\,H\,\}$, let $0 < s < \infty$ and let $1 < p < \infty$ such that
	\begin{equation*}
		\mbox{\upshape ind}(X^{s, \omega}_p(\bR^n,\,E)) > 0.
	\end{equation*}
	Then we have
	\begin{enumerate}[(a)]
		\item $X^{s, \omega}_p(\bR^n,\,E) \hookrightarrow C_0(\bR^n,\,E)$;
		\item $X^{s, \omega}_p(\bR^n,\,E)$ is a multiplication algebra, if $E$ is a Banach algebra, and $s \in \dot{\omega} \cdot \bN$ for $X = H$.
	\end{enumerate}
\end{theorem}
\vspace*{-1.5em}
\begin{proof}
	Assertion (a) is proved as part of \cite[Theorem 3.9.1]{Amann:Maximal-Regularity}.
	Assertion (b) is a direct consequence of our first main result, \Thmref{Multiplication-Anisotropic}, below.
\end{proof}
\begin{remark}
\remlabel{Algebra-Anisotropic}
(a) The results of \Thmref{Algebra-Anisotropic} may be easily transferred to anisotropic function spaces on domains $D \subseteq \bR^n$.
Indeed, they remain true if there is a bounded, linear extension operator
$\textrm{ext}: X^{s, \omega}_p(D,\,E) \longrightarrow X^{s, \omega}_p(\bR^n,\,E)$.
\vspace*{-0.25em}

(b) A combination of \Thmref{Algebra-Anisotropic} with \eqnref{HW-Anisotropic} and \eqnref{BW-Anisotropic} yields analogous results
for anisotropic Sobolev-Slobodeckij spaces $W^{s, \omega}_p$.
Note, however, that not all values of $0 < s < \infty$ are admissible
for the identifications \eqnref{HW-Anisotropic} and \eqnref{BW-Anisotropic};
cf.~also \Remref{Multiplication-Anisotropic}~(b) below.
\vspace*{-0.25em}

(c) The result \cite[Theorem 3.9.1]{Amann:Maximal-Regularity}, which is used for the proof of assertion (a),
is formulated in particular for the full Besov scale $B^{s, \omega}_{p, q}$ with parameters $1 \leq p,\,q < \infty$.
A similar generalization is possible for \Thmref{Multiplication-Anisotropic}, cf.~\Remref{Multiplication-Anisotropic}~(g),
and hence also for \Thmref{Algebra-Anisotropic}.
\end{remark}
\vspace*{-0.5em}

Our first main result is the analog of the isotropic 
\Thmref{Multiplication-Isotropic}
for anisotropic Besov and Bessel-potential scales.
\begin{theorem}
	\thmlabel{Multiplication-Anisotropic}
	Let $m,\,\nu \in \bN$ and $n,\,\omega \in \bN^\nu$.
	Let $E_1,\,\dots,\,E_m$ and $E$ be UMD-spaces that allow for a multiplication \eqnref{Multiplication},
	and that have property~$(\alpha)$ if $\omega \neq \dot{\omega} \cdot (1,\,\dots,\,1)$.
	Moreover, let $X_1,\,\dots,\,X_m,\,X \in \{\,B,\,H\,\}$,
	let $0 \leq s_1,\,\dots,\,s_m,\,s < \infty$ and let $1 < p_1,\,\dots,\,p_m,\,p < \infty$ such that
	\begin{equation*}
		\textrm{(i)} \quad s \leq \min\,\{\,s_1,\,\dots,\,s_m\,\} \qquad \textrm{and} \qquad
		\textrm{(ii)} \quad {\displaystyle{\frac{1}{p} \leq \sum^m_{j = 1} \frac{1}{p_j}}}.
	\end{equation*}
	Furthermore, let
	\begin{equation*}
		\textrm{\upshape ind}_j := \textrm{\upshape ind}([X_j]^{s_j, \omega}_{p_j}(\bR^n,\,E_j)), \quad j = 1,\,\dots,\,m \qquad \textrm{and} \qquad
		\textrm{\upshape ind} := \textrm{\upshape ind}(X^{s, \omega}_p(\bR^n,\,E))
	\end{equation*}
	and assume that
	\begin{equation*}
		\textrm{(iii)} \quad \textrm{\upshape ind} \leq \left\{
		\begin{array}{ll}
			\min\,\{\,\textrm{\upshape ind}_1,\,\dots,\,\textrm{\upshape ind}_m\,\},                           & \qquad \textrm{\upshape ind}_1,\,\dots,\,\textrm{\upshape ind}_m \geq 0, \\[0.5em]
			{\displaystyle{\sum^m_{\substack{j = 1 \\ \textrm{\upshape ind}_j < 0}}}} \textrm{\upshape ind}_j, & \qquad \textrm{otherwise}.
		\end{array}
		\right.
	\end{equation*}
	Finally, assume that
	\begin{itemize}
		\item[{\itshape (a)}] $s_j > s$ for all $j \in \{\,1,\,\dots,\,m\,\}$ for which $X_j \neq X$;
		\item[{\itshape (b)}] in case $X = B$: $s > 0$ and $p_j = p$ for all $j \in \{\,1,\,\dots,\,m\,\}$ for which $s_j = s$;
		\item[{\itshape (c)}] in case $X = B$: the inequality (iii) is strict or $\max\,\{\,p_1,\,\dots,\,p_m\,\} \leq p$;
		\item[{\itshape (d)}] in case $X = H$: $s \in \dot{\omega} \cdot \bN_0$, or inequality (i) is strict, or there is equality in (ii);
		\item[{\itshape (e)}] at least one of the inequalities (ii), (iii) is strict, if $X_j \neq X$ for some $j \in \{\,1,\,\dots,\,m\,\}$;
		\item[{\itshape (f)}] the inequality (iii) is strict, if $\textrm{\upshape ind}_j = 0$ for some $j \in \{\,1,\,\dots,\,m\,\}$.
	\end{itemize}
	Then the multiplication \eqnref{Multiplication-S} extends to a continuous, $m$-linear operator
	\begin{equation*}
		\bullet: [X_1]^{s_1, \omega}_{p_1}(\bR^n,\,E_1) \times \dots \times [X_m]^{s_m, \omega}_{p_m}(\bR^n,\,E_m) \longrightarrow X^{s, \omega}_p(\bR^n,\,E)
	\end{equation*}
	in a unique way.
\end{theorem}
\vspace*{-0.75em}
A complete proof of \Thmref{Multiplication-Anisotropic} is given in \Secref{Multiplication-Proof}.
Here, however, some remarks are in order.
\begin{remark}
\remlabel{Multiplication-Anisotropic}
(a) \Remref{Algebra-Anisotropic}~(a) applies also here.  
Concerning parabolic problems one is in particular interested in the situation $n = (1,\,n^\prime)$ with $n^\prime \in \bN$
and $D = J \times \Omega$, where $J = (0,\,a)$ with $0 < a \leq \infty$ is an interval
and $\Omega \subseteq \bR^{n^\prime}$ is a (sufficiently smooth) domain.
Combining Propositions~\ref{prop:Characterization-Anisotropic} and \ref{prop:Characterization-Anisotropic-Besov}
with standard extension techniques, one can obtain a bounded linear extension operator
$\textrm{ext}: X^{s, \omega}_p(J \times \Omega,\,E) \longrightarrow X^{s, \omega}_p(\bR^n,\,E)$ simultaneously for all $s \geq 0$ and $1 < p < \infty$.
Similarly, \Thmref{Multiplication-Anisotropic} may be transferred to anisotropic function spaces on manifolds as e.\,g.\ $J \times \partial \Omega$,
provided these manifolds are sufficiently smooth and allow for a suitable localization.

(b) An application of \eqnref{HW-Anisotropic} and \eqnref{BW-Anisotropic} shows that we may choose $X_j = W$ for some/all $j \in \{\,1,\,\dots,\,m\,\}$ and/or $X = W$.
Note, however, that such a choice requires $0 < s_j < \infty$ and/or $0 < s < \infty$ to be admissible for the identifications \eqnref{HW-Anisotropic} and \eqnref{BW-Anisotropic} to hold.
Concerning parabolic problems one is in particular interested in the situation described in Remark~(a).
In these cases one has $\omega = (2 \sigma,\,1)$ with $\sigma \in \bN$ and one has to deal
with anisotropic Bessel potential spaces $H^{2 \sigma, \omega}_p(J \times \Omega,\,E)$,
while anisotropic Sobolev-Slobodeckij spaces $W^{2 \sigma - \kappa - 1 / p, \omega}_p(J \times \partial \Omega,\,E)$
with $\kappa \in \bN_0$, $\kappa < \sigma$ usually appear due to boundary conditions.
These spaces, however, coincide with Besov spaces thanks to \eqnref{BW-Anisotropic}.
Therefore, \Thmref{Multiplication-Anisotropic} may be safely applied in
these cases with Sobolev-Slobodeckij spaces replaced by Besov spaces.
This useful fact will be employed for all examples presented in \Secref{Applications}.

(c) The index condition (iii) is readily seen to be equivalent to
\begin{equation*}
	\textrm{(iii)} \quad \textrm{ind} \leq \sum_{j \in M} \textrm{ind}_j, \qquad \varnothing \neq M \subseteq \{\,1,\,\dots,\,m\,\}.
\end{equation*}
This fact will be employed several times in the proof of \Thmref{Multiplication-Anisotropic}.

(d) As remarked in (a) the assertions of \Thmref{Multiplication-Anisotropic} remain valid for certain domains $D \subseteq \bR^n$,
e.\,g.\ for the unit ball $D = B_1(0)$.
In that case the constant functions belong to each of the anisotropic function spaces considered in \Thmref{Multiplication-Anisotropic}.
Thus, if we choose $E_\ell = E$ for some $\ell \in \{\,1,\,\dots,\,m\,\}$ and set $E_j = \bF \in \{\,\bR,\,\bC\,\}$ for all $j \in \{\,1,\,\dots,\,m\,\}$ with $j \neq \ell$,
then there trivially exists a multiplication of type \eqnref{Multiplication}.
Therefore, if the parameters $0 \leq s_1,\,\dots,\,s_m,\,s < \infty$ and $1 < p_1,\,\dots,\,p_m,\,p < \infty$ satisfy the requirements of \Thmref{Multiplication-Anisotropic},
then we obtain
\begin{equation*}
	[X_j]^{s_j, \omega}_{p_j}(D,\,E) \hookrightarrow X^{s, \omega}_p(D,\,E).
\end{equation*}
However, such an embedding for $X_j \neq X$ can only be valid, if $s_j > s$.
This shows that the constraint (a) in \Thmref{Multiplication-Anisotropic} is necessary.

(e) \Thmref{Multiplication-Anisotropic} contains H{\"o}lder's inequality as a borderline case.
Indeed, for the special choice $s_1 = \ldots = s_m = s = 0$ it is required that $X = H$ and (ii) and (iii) reduce to $\sum^m_{j = 1} \frac{1}{p_j} = \frac{1}{p}$.
This also shows that the restriction $s > 0$ for $X = B$ in \Thmref{Multiplication-Anisotropic}~(b) is necessary, since it is well-known that
\begin{equation*}
	B^0_{p_1}(\bR^n) \cdot \ldots \cdot B^0_{p_m}(\bR^n) \not\hookrightarrow B^0_p(\bR^n)
\end{equation*}
for $\sum^m_{j = 1} \frac{1}{p_j} = \frac{1}{p}$ even in the scalar-valued isotropic case, cf.~\cite[Theorem~4.3.2]{Sickel-Triebel:Hoelder-Inequalities}.

(f) If $s_j = s$, then the index condition (iii) implies $p_j \geq p$.
Therefore, an equivalent formulation of constraint (b) of \Thmref{Multiplication-Anisotropic} reads
\begin{itemize}
	\item[(b)] in case $X = B$: $s > 0$ and $p \geq p_j$ for all $j \in \{\,1,\,\dots,\,m\,\}$ for which $s_j = s$.
\end{itemize}
Moreover, there is the hidden constraint
\begin{itemize}
	\item[(g)] if $\textrm{ind} = \sum_{j \in M} \textrm{ind}_j$ for some $\varnothing \neq M \subseteq \{\,1,\,\dots,\,m\,\}$, then $\frac{1}{p} \leq \sum_{j \in M} \frac{1}{p_j}$.
\end{itemize}
This stems from the fact that $s \leq \sum_{j \in M} s_j$ for all $\varnothing \neq M \subseteq \{\,1,\,\dots,\,m\,\}$, cf.~also Remark~(g).

(g) A careful examination of the proof of \Thmref{Multiplication-Anisotropic} shows that the case $X_1 = \ldots = X_m = X = B$
may be generalized to spaces $B^{s_1, \omega}_{p_1, q_1}(\bR^n,\,E_1),\,\dots,\,B^{s_m, \omega}_{p_m, q_m}(\bR^n,\,E_m),\,B^{s, \omega}_{p, q}(\bR^n,\,E)$
with parameters $0 < s_1,\,\dots,\,s_m,\,s < \infty$, $1 \leq p_1,\,\dots,\,p_m,\,p < \infty$, and $1 \leq q_1,\,\dots,\,q_m,\,q < \infty$.
The inequalities (i), (ii), and (iii) as well as the constraint (f) remain the same for this setting. 
Constraints (a)~and~(e) are trivially satisfied for $X_1 = \ldots = X_m = X$.
However, the constraint (c) and the constraints (b) and (g) of Remark~(f) have to be replaced by
\begin{itemize}
	\item[(b')] $s > 0$ and $q \geq q_j$ for all $j \in \{\,1,\,\dots,\,m\,\}$ for which $s_j = s$; \\[-0.75\baselineskip]
	\item[(c')] the inequality (iii) is strict or \\
		$\max\,\{\,q_1,\,\dots,\,q_m\,\} \leq q$, and $1 \leq q_j \leq p_j$ for all $j \in \{\,1,\,\dots,\,m\,\}$ with $\ind_j < 0$; \\[-0.75\baselineskip]
	\item[(g')] if $\textrm{ind} = \sum_{j \in M} \textrm{ind}_j$ for some $\varnothing \neq M \subseteq \{\,1,\,\dots,\,m\,\}$, then $\frac{1}{q} \leq \sum_{j \in M} \frac{1}{q_j}$.
\end{itemize}
Note that these constraints are revealed as necessary in the scalar-valued case for $m = 2$ by \mbox{\cite[Theorem~4.2]{Johnsen:Multiplication}}.
Moreover, this version of \Thmref{Multiplication-Anisotropic} would include some borderline cases
that are excluded in the result \cite[Theorem~4.1]{Amann:Multiplication} for the isotropic case.
The reason is that the proof of \cite[Theorem~4.1]{Amann:Multiplication} does not make use of the sharp Sobolev type embedding
\eqnref{Besov-Bessel-Potential-Embedding}, which is derived in Appendix~A, and which is necessary for certain borderline cases,
but which seems to be unavailable, if $E_1,\,\dots,\,E_m,\,E$ are not required to be UMD-spaces as in \cite{Amann:Multiplication},
where arbitrary Banach spaces are considered.

(h) The constraint (d) in \Thmref{Multiplication-Anisotropic} is introduced to simplify the strategy of the proof of \Thmref{Multiplication-Anisotropic}.
However, we conjecture that the theorem remains valid without the constraint (d),
since the excluded cases should be treatable e.\,g.\ by arguments that employ intrinsic norms for anisotropic vector-valued Bessel potential spaces,
which, however, seem not to by available in the literature at this time.

(i) \Thmref{Multiplication-Anisotropic} contains higher order H{\"o}lder inequalities for anisotropic function spaces as special cases.
Indeed, if $X_1 = \ldots = X_m = X \in \{\,B,\,H\,\}$, $s_1 = \ldots = s_m = s > 0$, and $1 < p_1,\,\dots,\,p_m < \infty$
with $\frac{1}{p} = \sum^m_{j = 1} \frac{1}{p_j}$, then the requirements of \Thmref{Multiplication-Anisotropic} are satisfied and we obtain
\begin{equation*}
	X^{s, \omega}_{p_1}(\bR^n,\,E_1) \bullet \dots \bullet X^{s, \omega}_{p_m}(\bR^n,\,E_m) \hookrightarrow X^{s, \omega}_p(\bR^n,\,E).
\end{equation*}

(j) \Thmref{Multiplication-Anisotropic} may be used to derive results concerning multiplication of vector-valued anisotropic function spaces of negative order.
Indeed, following the strategy of the proofs of \cite[Theorems~2.3~\&~4.3]{Amann:Multiplication} one may employ \eqnref{Multiplication} to define
a multiplication
\begin{equation*}
	\circ: E_1 \times \ldots \times E_{\ell - 1} \times E^\prime \times E_{\ell + 1} \times \ldots \times E_m \longrightarrow E^\prime_\ell
\end{equation*}
for some fixed $\ell \in \{\,1,\,\dots,\,m\,\}$ via
\begin{equation*}
	\langle e_\ell,\,e_1 \circ \dots \circ e_{\ell - 1} \circ e^\prime \circ e_{\ell + 1} \circ e_m \rangle := \langle e_1 \bullet \dots \bullet e_m,\,e^\prime \rangle,
		\qquad e_j \in E_j,\ e^\prime \in E^\prime.
\end{equation*}
Now, if we choose a set of parameters $0 < s_1,\,\dots,\,s_m,\,s < \infty$, $1 < p_1,\,\dots,\,p_m,\,p < \infty$ such that the parameters
$s_1,\,\dots,\,s_{\ell - 1},\,s,\,s_{\ell + 1},\,\dots,\,s_m,\,s_\ell$ and $p_1,\,\dots,\,p_{\ell - 1},\,p^\prime,\,p_{\ell + 1},\,\dots,\,p_m,\,p^\prime_\ell$
satisfy the requirements of \Thmref{Multiplication-Anisotropic}, then we obtain a bounded $m$-linear operator
\begin{equation*}
	\circ: \prod^{\ell - 1}_{j = 1} [X_j]^{s_j, \omega}_{p_j}(\bR^n,\,E_j) \times X^{s, \omega}_{p^\prime}(\bR^n,\,E^\prime) \times \prod^{m}_{j = \ell + 1} [X_j]^{s_j, \omega}_{p_j}(\bR^n,\,E_j) \longrightarrow [X_\ell]^{s_\ell, \omega}_{p^\prime_\ell}(\bR^n,\,E^\prime_\ell).
\end{equation*}
Thus, the same duality argument as above, this time on the level of the involved function spaces, implies the multiplication
\begin{equation*}
	\bullet: \prod^{\ell - 1}_{j = 1} [X_j]^{s_j, \omega}_{p_j}(\bR^n,\,E_j) \times [X_\ell]^{- s_\ell, \omega}_{p_\ell}(\bR^n,\,E_\ell) \times \prod^{m}_{j = \ell + 1} [X_j]^{s_j, \omega}_{p_j}(\bR^n,\,E_j) \longrightarrow X^{- s, \omega}_p(\bR^n,\,E)
\end{equation*}
to be continuous.
Note that the UMD-spaces $E_\ell$ and $E$ are reflexive,
i.\,e.\ $E^{\prime \prime}_\ell \doteq E_\ell$ and $E^{\prime \prime} \doteq E$
and we have $[X_\ell]^{s_\ell, \omega}_{p^\prime_\ell}(\bR^n,\,E^\prime_\ell)^\prime \doteq [X_\ell]^{- s_\ell, \omega}_{p_\ell}(\bR^n,\,E_\ell)$
and $X^{s, \omega}_{p^\prime}(\bR^n,\,E^\prime)^\prime \doteq X^{- s, \omega}_p(\bR^n,\,E)$. \pagebreak

(k) Some cases for which a continuous multiplication is available are excluded in \Thmref{Multiplication-Anisotropic}.
For example, we obtain continuity of the multiplications
\begin{equation*}
	H^{5 / 2}(\bR^n) \cdot H^{3 / 2}(\bR^n) \hookrightarrow H^{3 / 2}(\bR^n), \qquad
	H^{3 / 2}(\bR^n) \cdot H^{1 / 2}(\bR^n) \hookrightarrow H^{1 / 2}(\bR^n)
\end{equation*}
for $\nu = 2$, $n = \omega = (1,\,1)$.
Thus, by bilinear complex interpolation w.\,r.\,t.\ $[\,\cdot\,,\,\cdot\,]_{1/2}$ we obtain continuity of the multiplication
\begin{equation*}
	H^2(\bR^n) \cdot H^1(\bR^n) \hookrightarrow H^1(\bR^n).
\end{equation*}
This case, however, is not included in \Thmref{Multiplication-Anisotropic},
since $\textrm{ind}(H^1(\bR^n)) = 0$ and the inequality (iii) is required to be strict by constraint (f) of \Thmref{Multiplication-Anisotropic}.
The underlying problem is the following:
The family of complex interpolation functors $[\,\cdot\,,\,\cdot\,]_\theta$ as introduced e.\,g.\ in \cite[Sec.~1.9]{Triebel:Interpolation}
may be used to interpolate multilinear mappings, cf.~\cite[Sec.~1.19.5]{Triebel:Interpolation}.
Since the anisotropic vector-valued Besov and Bessel potential scales form complex interpolation scales,
cf.~\eqnref{Besov-Complex-Interpolation} and \eqnref{Bessel-Potential-Complex-Interpolation},
the set of differentiability/integrability parameters, for which the induced multiplication is continuous, is convex.
However, the set of parameters that satisfy the assumptions of \Thmref{Multiplication-Anisotropic} is not convex due to the constraint (f).
Nevertheless, most of these cases that are excluded in \Thmref{Multiplication-Anisotropic} may be obtained by multilinear complex interpolation as shown above;
see also \Remref{Multiplication-Anisotropic-Simple}~(b).

(l) If we assume $E_\ell$ to be a unital Banach algebra with unit element $1_\ell$ for some $\ell \in \{\,1,\,\dots,\,m\,\}$,
then the multiplication \eqnref{Multiplication} induces a multiplication
\begin{equation}
	\eqnlabel{Multiplication-Reduced}
	\begin{array}{c}
		\circ: E_1 \times \dots \times E_{\ell - 1} \times E_{\ell + 1} \times \dots \times E_m \longrightarrow E, \\[0.5em]
		(e_1,\,\dots,\,e_{\ell -1},\,e_{\ell + 1},\,\dots,\,e_m) \mapsto e_1 \bullet \dots \bullet e_{\ell - 1} \bullet 1_\ell \bullet e_{\ell + 1} \bullet \dots \bullet e_m.
	\end{array}
\end{equation}
Now, if $X_1,\,\dots,\,X_m,\,X$, $s_1,\,\dots,\,s_m,\,s$, and $p_1,\,\dots,\,p_m,\,p$ satisfy the assumptions of \Thmref{Multiplication-Anisotropic}
with $\max\,\{\,p_1,\,\dots,\,p_m\,\} \leq p$, then this choice of parameters without $X_\ell$, $s_\ell$, and $p_\ell$ also satisfies the assumptions of \Thmref{Multiplication-Anisotropic}
for the reduced multiplication \eqnref{Multiplication-Reduced} and we obtain a continuous, $(m - 1)$-linear operator
\begin{equation*}
	\circ: \prod^{\ell - 1}_{j = 1} [X_j]^{s_j, \omega}_{p_j}(\bR^n,\,E_j) \, \times \!\! \prod^m_{j = \ell + 1} \!\! [X_j]^{s_j, \omega}_{p_j}(\bR^n,\,E_j) \longrightarrow X^{s, \omega}_p(\bR^n,\,E).
\end{equation*}
Note that for $u_j \in [X_j]^{s_j, \omega}_{p_j}(\bR^n,\,E_j)$ for $j \in \{\,1,\,\dots,\,m\,\} \setminus \{\,\ell\,\}$ we then have
\begin{equation*}
	u_1 \circ \dots \circ u_{\ell - 1} \circ u_{\ell + 1} \circ \dots \circ u_m
		= u_1 \bullet \dots \bullet u_{\ell - 1} \bullet \boldsymbol{1}_\ell \bullet u_{\ell + 1} \bullet \dots \bullet u_m \in X^{s, \omega}_p(\bR^n,\,E),
\end{equation*}
where the constant function $\boldsymbol{1}_\ell \equiv 1_\ell$ satisfies $\boldsymbol{1}_\ell \notin [X_\ell]^{s_\ell, \omega}_{p_\ell}(\bR^n,\,E_\ell) \subseteq L_{p_\ell}(\bR^n,\,E_\ell)$.
However, for $u_\ell \in [X_\ell]^{s_\ell, \omega}_{p_\ell}(\bR^n,\,E_\ell)$ the product
\begin{equation*}
	u_1 \bullet \dots \bullet u_{\ell - 1} \bullet u^0_\ell \bullet u_{\ell + 1} \bullet \dots \bullet u_m
		= u_1 \circ \dots \circ u_{\ell - 1} \circ u_{\ell + 1} \circ \dots \circ u_m \in X^{s, \omega}_p(\bR^n,\,E)
\end{equation*}
is well-defined, even if $u^0_\ell = \boldsymbol{1}_\ell \notin [X_\ell]^{s_\ell, \omega}_{p_\ell}(\bR^n,\,E_\ell)$.
More generally, if $E_1,\,\dots,\,E_m$ are unital Banach algebras, we obtain in the same way as above for every $\varnothing \neq N \subsetneq \{\,1,\,\dots,\,m\,\}$
a reduced multiplication
\begin{equation}
	\eqnlabel{Induced-Multiplication-Reduced}
	\circ_N: \prod^m_{\substack{j = 1 \\ j \notin N}} [X_j]^{s_j, \omega}_{p_j}(\bR^n,\,E_j) \longrightarrow X^{s, \omega}_p(\bR^n,\,E).
\end{equation}
Thus, for $u_j \in [X_j]^{s_j, \omega}_{p_j}(\bR^n,\,E_j)$ for $j \in \{\,1,\,\dots,\,m\,\}$ the functions
$u^{\alpha_1}_1 \bullet \dots \bullet u^{\alpha_m}_m \in X^{s, \omega}_p(\bR^n,\,E)$ are well-defined for $\alpha_1,\,\dots,\,\alpha_m \in \{\,0,\,1\,\}$
in the above sense based on $\circ_N$ for $N = \{\,j\,:\,\alpha_j = 0\,\}$, even if $u^0_j = \boldsymbol{1}_j \notin [X_j]^{s_j, \omega}_{p_j}(\bR^n,\,E_j)$ for $j \in N$.
This observation will be useful for the proof of \Thmref{Nemytskij-Anisotropic} below,
which involves power-series expansions within the function spaces $[X_j]^{s_j, \omega}_{p_j}(\bR^n,\,E_j)$ that involve multiindices $\alpha \in \bN^m_0$
that satisfy $\alpha \neq 0$, but for which $\alpha_j = 0$ can occur for some (but not all) $j \in \{\,1,\,\dots,\,m\,\}$.
\end{remark}
In applications we are often faced with the situation that one of the factors on the left-hand side equals the target space.
This is the case for many examples considered in \Secref{Applications} or for multiplication algebras, for instance.
In such a situation the index inequality (iii) of
\Thmref{Multiplication-Anisotropic} simplifies as the following
consequence of \Thmref{Multiplication-Anisotropic} shows.
\begin{theorem}
	\thmlabel{Multiplier-Anisotropic}
	Let $m,\,\nu \in \bN$ and $n,\,\omega \in \bN^\nu$.
	Let $E_1,\,\dots,\,E_m$ and $E$ be UMD-spaces that allow for a multiplication \eqnref{Multiplication},
	and that have property~$(\alpha)$ if $\omega \neq \dot{\omega} \cdot (1,\,\dots,\,1)$.
	Moreover, let $X_1,\,\dots,\,X_m,\,X \in \{\,B,\,H\,\}$,
	let $0 \leq s \leq s_1,\,\dots,\,s_m < \infty$ and let $1 < p_1,\,\dots,\,p_m,\,p < \infty$
	such that $X_\ell = X$, $s_\ell = s$ and $p_\ell = p$ for some $\ell \in \{\,1,\,\dots,\,m\,\}$.
	Furthermore, let
	\begin{equation*}
		\textrm{\upshape ind}_j := \textrm{\upshape ind}([X_j]^{s_j, \omega}_{p_j}(\bR^n,\,E_j)), \quad j = 1,\,\dots,\,m \qquad \textrm{and} \qquad
		\textrm{\upshape ind} := \textrm{\upshape ind}(X^{s, \omega}_p(\bR^n,\,E))
	\end{equation*}
	and assume that
	\begin{equation*}
		\ind_j > 0 \qquad \textrm{and} \qquad \ind_j \geq \ind \qquad \textrm{for all} \ j \in \{\,1,\,\dots,\,m\,\} \setminus \{\,\ell\,\}.
	\end{equation*}
	Finally, assume that
	\begin{itemize}
		\item[{\itshape (a)}] $s_j > s$ for all $j \in \{\,1,\,\dots,\,m\,\}$ for which $X_j \neq X$;
		\item[{\itshape (b)}] in case $X = B$: $s > 0$ and $p_1,\,\dots,\,p_m \leq p$;
		\item[{\itshape (c)}] in case $X = H$: $s \in \dot{\omega} \cdot \bN_0$.
	\end{itemize}
	Then the multiplication \eqnref{Multiplication-S} extends to a continuous, $m$-linear operator
	\begin{equation*}
		\bullet: [X_1]^{s_1, \omega}_{p_1}(\bR^n,\,E_1) \times \dots \times [X_m]^{s_m, \omega}_{p_m}(\bR^n,\,E_m) \longrightarrow X^{s, \omega}_p(\bR^n,\,E)
	\end{equation*}
	in a unique way.
\end{theorem}
A complete proof of \Thmref{Multiplier-Anisotropic} is given in \Secref{Multiplier-Proof}.
Here we only note the following comments.
\begin{remark}
\remlabel{Multiplication-Anisotropic-Simple}
(a) \Thmref{Multiplier-Anisotropic} provides sufficient conditions for a space $X^{s, \omega}_p$ with $X \in \{\,B,\,H\,\}$
to be a space of multipliers for another space $Y^{t, \omega}_q$ with $Y \in \{\,B,\,H\,\}$.
However, the example given in \Remref{Multiplication-Anisotropic}~(k) shows that these conditions are not necessary.
Nevertheless, most of these cases that are excluded in \Thmref{Multiplier-Anisotropic} may be obtained by multilinear complex interpolation
as demonstrated in \Remref{Multiplication-Anisotropic}~(k).
In fact \Thmref{Multiplier-Anisotropic} covers some of these cases,
where the continuity of the induced multiplication is obtained by multilinear complex interpolation for the case $\ind_\ell = \ind = 0$;
cf.~\Subsecref{Multiplier-Proof-Step-2}; see also Remark (b) below.

(b) \Thmref{Multiplier-Anisotropic} provides sufficient conditions for a space $X^{s, \omega}_p$ with $X \in \{\,B,\,H\,\}$
to be a multiplication algebra, i.\,e.\ \Thmref{Multiplier-Anisotropic} implies \Thmref{Algebra-Anisotropic}~(b).
In this case the condition $\textrm{ind}(H^{s, \omega}_p) > 0$ is known to be necessary for isotropic scalar-valued Bessel potential spaces,
cf.~\cite[Theorem~4.6.4/1]{Runst-Sickel:Function-Spaces}, while $\textrm{ind}(B^{s, \omega}_{p, q}) = 0$ is allowed for isotropic vector-valued Besov spaces,
provided that e.\,g.\ $q = 1$, cf.~\cite[Remark~4.2~(a)]{Amann:Multiplication}.
Such a borderline case may not be recovered by multilinear complex interpolation from \Thmref{Multiplier-Anisotropic}.

(c) Remarks~\ref{rem:Multiplication-Anisotropic}~(a), (b), (g), and (h),
which concern possible generalizations of \Thmref{Multiplication-Anisotropic}, also apply to \Thmref{Multiplier-Anisotropic}.
\end{remark}

\subsection{Nemytskij Operators in Anisotropic Function Spaces}\subseclabel{Anisotropic-Nemytskij}
As a consequence of \Thmref{Multiplication-Anisotropic} we obtain our second main result.
It is on analytic Nemytskij operators in vector-valued anisotropic Besov and Bessel potential spaces.
This is in particular useful to treat Nemytskij operators that arise from spatial transformations
of free boundary problems, cf.~\Secref{Applications} for several demonstrations of its application. \pagebreak
\begin{theorem}
	\thmlabel{Nemytskij-Anisotropic}
	Let $m,\,\nu \in \bN$ and $n,\,\omega \in \bN^\nu$.
	Let $E_1,\,\dots,\,E_m,\,E$ be UMD-spaces that are unital Banach algebras, that allow for a multiplication \eqnref{Multiplication},
	and that have property~$(\alpha)$ if $\omega \neq \dot{\omega} \cdot (1,\,\dots,\,1)$.
	Let $X_1,\,\dots,\,X_m,\,X \in \{\,B,\,H\,\}$, let $0 < s \leq s_1,\,\dots,\,s_m < \infty$, and let $1 < p_1,\,\dots,\,p_m \leq p < \infty$ such that
	\begin{equation*}
		0 < \textrm{\upshape ind}(X^{s, \omega}_p(\bR^n,\,E)) \leq \textrm{\upshape ind}([X_1]^{s_1, \omega}_{p_1}(\bR^n,\,E_1)),\,\dots,\,\textrm{\upshape ind}([X_m]^{s_m, \omega}_{p_m}(\bR^n,\,E_m))
	\end{equation*}
	as well as
	\begin{itemize}
		\item[{\itshape (a)}] $s_j > s$ for all $j \in \{\,1,\,\dots,\,m\,\}$ for which $X_j \neq X$;
		\item[{\itshape (b)}] in case $X = H$: $s \in \dot{\omega} \cdot \bN$ or $s < \min\,\{s_1,\,\dots,\,s_m\,\}$.
	\end{itemize}
	Furthermore, let $r > 0$ and let
	\begin{equation*}
		\phi: (-r,\,r)^m \longrightarrow \bR, \qquad \qquad
		\phi(x) = \sum_{\alpha \in \bN^m_0} a_\alpha x^\alpha, \qquad x \in \bR^m,\ |x_j| < r,
	\end{equation*}
	be an analytic function with $\phi(0) = 0$.
	Then, there exists $\rho > 0$ such that the Nemytskij operator
	\begin{equation*}
		\Phi: U \longrightarrow X^{s, \omega}_p(\bR^n),
		\qquad \qquad \Phi(u) = \phi \circ u, \qquad u \in U,
	\end{equation*}
	is well-defined and analytic (represented in $X^{s, \omega}_p(\bR^n, E)$ as a power series) on
	\begin{equation*}
		U := \Big\{\,(u_1,\,\dots,\,u_m) \in [X_1]^{s_1,
		\omega}_{p_1}(\bR^n,\,E_1) \times \ldots \times
		[X_m]^{s_m,
		\omega}_{p_m}(\bR^n,\,E_m)\,:\,\|u_j\|_{[X_j]^{s_j,
		\omega}_{p_j}(\bR^n, E_j)} < \rho\,\Big\},
		\end{equation*}
	and satisfies the estimate
	\begin{equation*}
		\|\Phi(u)\|_{X^{s, \omega}_p(\bR^n, E)} \leq L \!\! \max_{j = 1, \dots, m} \! \|u_j\|_{[X_j]^{s_j, \omega}_{p_j}(\bR^n, E_j)}, \qquad u = (u_1,\,\dots,\,u_m) \in U,
	\end{equation*}
	for some constant $L > 0$.
	Moreover, the power series that represents $\Phi$ converges in $C_0(\bR^n,\,E)$ and point-wise in $\bR^n$.
\end{theorem}
A complete proof of \Thmref{Nemytskij-Anisotropic} is given in \Secref{Nemytskij-Proof}.
\begin{remark}
\remlabel{Nemytskij-Anisotropic}
(a) Due to the assumptions on the Sobolev indices Theorems~\ref{thm:Algebra-Anisotropic} and \ref{thm:Multiplication-Anisotropic} imply
\begin{equation*}
	\begin{array}{rcrlrl}
		\|u\|_{C_0(\bR^n, E)}   & \leq & C   & \!\! \|u\|_{X^{s, \omega}_p(\bR^n, E)},             & \qquad u & \in X^{s, \omega}_p(\bR^n, E),             \\[0.5em]
		\|u\|_{C_0(\bR^n, E_j)} & \leq & C_j & \!\! \|u\|_{[X_j]^{s_j, \omega}_{p_j}(\bR^n, E_j)}, & \qquad u & \in [X_j]^{s_j, \omega}_{p_j}(\bR^n, E_j),
	\end{array}
\end{equation*}
as well as
\begin{equation*}
	\begin{array}{l}
		\|u_1 \cdot \ldots \cdot u_m\|_{X^{s, \omega}_p(\bR^n, E)} \leq M \|u_1\|_{[X_1]^{s_1, \omega}_{p_1}(\bR^n, E_1)} \cdot \ldots \cdot \|u_m\|_{[X_m]^{s_m, \omega}_{p_m}(\bR^n, E_m)}, \\[0.5em]
			\qquad \qquad \qquad \qquad \qquad \qquad \qquad \qquad u_1 \in [X_1]^{s_1, \omega}_{p_1}(\bR^n, E_1), \ldots, u_m \in [X_m]^{s_m, \omega}_{p_m}(\bR^n, E_m),
	\end{array}
\end{equation*}
as well as
\begin{equation*}
	\|u v\|_{[X_j]^{s_j, \omega}_{p_j}(\bR^n, E_j)} \leq M_j \|u\|_{[X_j]^{s_j, \omega}_{p_j}(\bR^n, E_j)}\,\|v\|_{[X_j]^{s_j, \omega}_{p_j}(\bR^n, E_j)}, \qquad \ u,\,v \in [X_j]^{s_j, \omega}_{p_j}(\bR^n, E_j),
\end{equation*}
for some constants $C_1,\,\dots,\,C_m,\,C,\,M_1,\,\dots,\,M_m,\,M > 0$.
Moreover, $M$ can be chosen such that the operator norms of all reduced multiplications \eqnref{Induced-Multiplication-Reduced}
for $\varnothing \neq N \subsetneq \{\,1,\,\dots,\,m\,\}$ are bounded by $M$.
Now, an inspection of the proof of \Thmref{Nemytskij-Anisotropic} shows that $\rho$ has to be chosen such that
\begin{equation*}
	\rho < \min\,\{\,C^{-1}_j,\,M^{-1}_j\,\}\,r
\end{equation*}
and we have $L = L(M,\,\{\,a_\alpha\,:\,\alpha \in \bN^m_0,\ |\alpha| = 1\,\})$.

(b) Clearly, Remarks~\ref{rem:Multiplication-Anisotropic}~(a), (b), (g), and (h) concerning possible generalizations of \Thmref{Multiplication-Anisotropic}
also apply to \Thmref{Nemytskij-Anisotropic}.
\end{remark}

\section{Proof of \Thmref{Multiplication-Anisotropic}: Multiplication in Anisotropic Function Spaces}
\seclabel{Multiplication-Proof}
This section is devoted to the proof of \Thmref{Multiplication-Anisotropic}.
In the following subsections we always assume the assumptions of \Thmref{Multiplication-Anisotropic} to be valid.
We have to distinguish several cases, where any additional assumptions are stated at the beginning of the corresponding subsection.
Moreover, we employ the notation of \Thmref{Multiplication-Anisotropic}; in particular
\begin{equation*}
	\textrm{\upshape ind}_j := \textrm{\upshape ind}([X_j]^{s_j, \omega}_{p_j}(\bR^n,\,E_j)), \quad j = 1,\,\dots,\,m \qquad \textrm{and} \qquad
	\textrm{\upshape ind} := \textrm{\upshape ind}(X^{s, \omega}_p(\bR^n,\,E)).
\end{equation*}

\subsection{Preliminaries}
\subseclabel{Multiplication-Proof-Preliminaries}
Before we give a proof of \Thmref{Multiplication-Anisotropic},
we show that it is sufficient to prove a slightly simpler version of the theorem.
So assume that $X_j \neq X$ for some fixed $j \in \{\,1,\,\dots,\,m\,\}$.
Then $s_j > s$ and at least one of the inequalities (ii) and (iii) is strict due to constraints (a) and (e).
So assume inequality (ii) to be strict.
Then we choose $\epsilon > 0$ such that $t := s_j - \epsilon > s$ and
\begin{equation*}
	\sum^m_{\substack{k = 1 \\ k \neq j}} \frac{1}{p_k} + \frac{1}{q}
		= \sum^m_{k = 1} \frac{1}{p_k} - \frac{\epsilon}{\omega \cdot n} > \frac{1}{p} \qquad \qquad \textrm{for} \quad
	\frac{\omega \cdot n}{q} := \frac{\omega \cdot n}{p_j} - \epsilon > 0.
\end{equation*}
Then we have $\textrm{ind}(X^{t, \omega}_q(\bR^n,\,E_j)) = \textrm{ind}([X_j]^{s_j, \omega}_{p_j}(\bR^n,\,E_j))$,
which shows that the set of parameters $s_1,\,\dots,\,s_{j - 1},\,t,\,s_{j + 1},\,\dots,\,s_m,\,s$ together with
$p_1,\,\dots,\,p_{j - 1},\,q,\,p_{j + 1},\,\dots,\,p_m,\,p$ satisfies the assumptions of \Thmref{Multiplication-Anisotropic},
where $t > s$ and the inequality (ii) is strict.
Thus, once it is proved that the induced multiplication
\begin{equation*}
	\bullet: \prod^{j - 1}_{k = 1} [X_k]^{s_k, \omega}_{p_k}(\bR^n,\,E_k) \times X^{t, \omega}_q(\bR^n,\,E_j) \times \prod^{m}_{k = j + 1} [X_k]^{s_k, \omega}_{p_k}(\bR^n,\,E_k) \longrightarrow \prod^{j - 1}_{k = 1} X^{s, \omega}_p(\bR^n,\,E)
\end{equation*}
is continuous, it follows that the induced multiplication
\begin{equation*}
	\bullet: [X_1]^{s_1, \omega}_{p_1}(\bR^n,\,E_1) \times \dots \times [X_m]^{s_m, \omega}_{p_m}(\bR^n,\,E_m) \longrightarrow X^{s, \omega}_p(\bR^n,\,E)
\end{equation*}
is also continuous, since \eqnref{Besov-Bessel-Potential-Embedding} and \eqnref{Bessel-Potential-Besov-Embedding}
imply $[X_j]^{s_j, \omega}_{p_j}(\bR^n,\,E_j) \hookrightarrow X^{t, \omega}_q(\bR^n,\,E_j)$.
However, if there is equality in (ii) and the inequality (iii) is strict, we need a different approach.
In this case we choose $\epsilon > 0$ such that $t := s_j - \epsilon > s$ and
\begin{equation*}
	\frac{\epsilon}{\dot{\omega}} < \min\,\left\{\,\sum_{k \in M} \textrm{ind}_k\,:\,\varnothing \neq M \subseteq \{\,1,\,\dots,\,m\,\}\,\right\} - \textrm{ind}.
\end{equation*}
Then we have $\textrm{ind}(X^{t, \omega}_{p_j}(\bR^n,\,E_j)) = \textrm{ind}([X_j]^{s_j, \omega}_{p_j}(\bR^n,\,E_j)) - \epsilon / \dot{\omega}$,
which shows that the set of parameters $s_1,\,\dots,\,s_{j - 1},\,t,\,s_{j + 1},\,\dots,\,s_m,\,s$ together with
$p_1,\,\dots,\,p_m,\,p$ satisfies the assumptions of \Thmref{Multiplication-Anisotropic}, where $t > s$ and the inequality (iii) is strict.
Thus, once it is proved that the induced multiplication
\begin{equation*}
	\bullet: \prod^{j - 1}_{k = 1} [X_k]^{s_k, \omega}_{p_k}(\bR^n,\,E_k) \times X^{t, \omega}_{p_j}(\bR^n,\,E_j) \times \prod^{m}_{k = j + 1} [X_k]^{s_k, \omega}_{p_k}(\bR^n,\,E_k) \longrightarrow \prod^{j - 1}_{k = 1} X^{s, \omega}_p(\bR^n,\,E)
\end{equation*}
is continuous, it follows that the induced multiplication
\begin{equation*}
	\bullet: [X_1]^{s_1, \omega}_{p_1}(\bR^n,\,E_1) \times \dots \times [X_m]^{s_m, \omega}_{p_m}(\bR^n,\,E_m) \longrightarrow X^{s, \omega}_p(\bR^n,\,E)
\end{equation*}
is also continuous, since \eqnref{Besov-Bessel-Potential-Embedding} and \eqnref{Bessel-Potential-Besov-Embedding}
imply $[X_j]^{s_j, \omega}_{p_j}(\bR^n,\,E_j) \hookrightarrow X^{t, \omega}_{p_j}(\bR^n,\,E_j)$.
Since these arguments may be repeatedly applied for all $j \in \{\,1,\,\dots,\,m\,\}$ for which $X_j \neq X$,
we may for the proof of \Thmref{Multiplication-Anisotropic} w.\,l.\,o.\,g.\ assume $X_1 = \dots = X_m = X$
whenever this seems to be convenient.

\subsection
	[The Case $s = 0$]
	{The Case $\boldsymbol{s = 0}$}
	\subseclabel{Multiplication-Proof-Step-1}
In this case we necessarily have $X = H$ due to constraint (b), i.\,e.
\begin{equation*}
	X^{s, \omega}_p(\bR^n,\,E) = H^{s, \omega}_p(\bR^n,\,E) = W^{s, \omega}_p(\bR^n,\,E) = L_p(\bR^n,\,E),
\end{equation*}
and based on \Subsecref{Multiplication-Proof-Preliminaries} we may assume $X_1 = \dots = X_m = X = H$.
Note that $s = 0$ is the only case, where we may have $s_j = 0$ for some $j \in \{\,1,\,\dots,\,m\,\}$.
In order to prove the assertion in this case,
it is sufficient to determine values $1 \leq q_1,\,\dots,\,q_m \leq \infty$ such that
\begin{equation}
	\eqnlabel{Proof-MA-1}
	H^{s_j, \omega}_{p_j}(\bR^n,\,E_j) \hookrightarrow L_{q_j}(\bR^n,\,E_j), \qquad j = 1,\,\dots,\,m, \qquad \qquad
	\sum^m_{j = 1} \frac{1}{q_j} = \frac{1}{p},
\end{equation}
since then the desired embedding follows from an application of H{\"o}lder's inequality.
Now, for the embeddings in \eqnref{Proof-MA-1} to hold due to \eqnref{Bessel-Potential-Lebesgue-Embedding}
it is sufficient that $1 / p_j \geq 1 / q_j$ as well as
\begin{equation*}
	\frac{1}{\dot{\omega}} \left( s_j - \frac{\omega \cdot n}{p_j} \right)
		   = \textrm{ind}_j
		\geq \textrm{ind}_\omega (L_{q_j}(\bR^n,\,E_j))
		   = - \frac{\omega \cdot n}{\dot{\omega}} \frac{1}{q_j}, \qquad j = 1,\,\dots,\,m,
\end{equation*}
with $q_j < \infty$, if $\mbox{ind}_j = 0$.
Thus, we need to ensure that
\begin{equation*}
	- \frac{\omega \cdot n}{\dot{\omega}} \sum^m_{j = 1} \frac{1}{p_j}
		\leq - \frac{\omega \cdot n}{\dot{\omega}} \sum^m_{j = 1} \frac{1}{q_j}
		\leq \sum^m_{j = 1} \left[\frac{1}{ \dot{\omega}} \left( s_j - \frac{\omega \cdot n}{p_j} \right) \right]_\ominus
\end{equation*}
with $[\rho]_\ominus := \min\,\{\,0,\,\rho\,\}$ for $\rho \in \bR$.
However, by assumption we have
\begin{equation*}
	- \frac{\omega \cdot n}{\dot{\omega}} \sum^m_{j = 1} \frac{1}{p_j}
		\leq - \frac{\omega \cdot n}{\dot{\omega}} \frac{1}{p} = \mbox{ind}
		\leq \!\! \sum^m_{\substack{j = 1 \\ \textrm{ind}_j < 0}} \!\! \mbox{ind}_j
		= \sum^m_{j = 1} \left[\frac{1}{ \dot{\omega}} \left( s_j - \frac{\omega \cdot n}{p_j} \right) \right]_\ominus,
\end{equation*}
where the second inequality is strict, if $\mbox{ind}_j = 0$ for some $j \in \{\,1,\,\dots,\,m\,\}$.
Therefore, we may apply \Lemref{Appendix-Realization} (with $\sigma_j \sim s_j / (\omega \cdot n)$, $\pi_j \sim 1 / p_j$, $\rho \sim 1 / p$, and $\rho_j \sim 1 / q_j$)
to obtain $p_j \leq q_j \leq \infty$ for $j = 1,\,\dots,\,m$ with $q_j < \infty$ for all $j = 1,\,\dots,\,m$, if $\mbox{ind}_j = 0$ for some $j \in \{\,1,\,\dots,\,m\,\}$.
By construction, \eqnref{Proof-MA-1} holds for these values of the $q_j$,
which completes the proof for this case.

\subsection
	[The Case $X = H$, Part (a)]
	{The Case $\boldsymbol{X = H}$, $\boldsymbol{s \in \dot{\omega}\,\cdot\,} \pmb{\bN}$}
	\subseclabel{Multiplication-Proof-Step-2}
In this case we have
\begin{equation*}
	X^{s, \omega}_p(\bR^n,\,E) = H^{s, \omega}_p(\bR^n,\,E) = W^{s, \omega}_p(\bR^n,\,E),
\end{equation*}
and based on \Subsecref{Multiplication-Proof-Preliminaries} we may assume $X_1 = \dots = X_m = X = H$.
Hence, in order to prove \Thmref{Multiplication-Anisotropic} in this special case,
it is necessary and sufficient to show that
\begin{equation*}
	\|\partial^\alpha_k (u_1 \bullet \ldots \bullet u_m)\|_{L_p(\bR^n,\,E)}
		\leq C \|u_1\|_{H^{s_1, \omega}_{p_1}(\bR^n,\,E_1)} \cdot \ldots \cdot \|u_m\|_{H^{s_m, \omega}_{p_m}(\bR^n,\,E_m)}
\end{equation*}
for all multiindices $\alpha \in \bN^{n_k}_0$ with $|\alpha| \leq s / \omega_k$ and all $k = 1,\,\dots,\,\nu$
with suitable constants \mbox{$C = C(k,\,\alpha) > 0$}.
Here $\partial^\alpha_k = \partial^{|\alpha|} / \partial x^\alpha_k$ again denotes the partial derivative of order $\alpha$
w.\,r.\,t.\ the $k$-th component $x_k$ of $x = (x_1,\,\dots,\,x_\nu) \in \bR^n = \bR^{n_1} \times \dots \times \bR^{n_\nu}$.
Now, we have
\begin{equation*}
	\partial^\alpha_k (u_1 \bullet \ldots \bullet u_m) =
		\!\!\!\!\!\!\!\!\!\!\!\! \sum_{\substack{\beta_1,\,\dots,\,\beta_m \in \bN^{n_k}_0 \\ \beta_1 + \ldots + \beta_m = \alpha}}
		\!\!\!\!\!\!\!\!\!\!\!\! {\textstyle \frac{\alpha!}{\beta_1! \cdot \ldots \cdot \beta_m!}}\,\partial^{\beta_1}_k u_1 \bullet \ldots \bullet \partial^{\beta_m}_k u_m
\end{equation*}
and, thus, it is necessary and sufficient to obtain an estimate
\begin{equation*}
	\|\partial^{\beta_1}_k u_1 \bullet \ldots \bullet \partial^{\beta_m}_k u_m\|_{L_p(\bR^n,\,E)}
		\leq C \|u_1\|_{H^{s_1, \omega}_{p_1}(\bR^n,\,E_1)} \cdot \ldots \cdot \|u_m\|_{H^{s_m, \omega}_{p_m}(\bR^n,\,E_m)}
\end{equation*}
for all possible decompositions $\beta_1,\,\dots,\,\beta_m \in \bN^{n_k}_0$ of a given multiindex $\alpha \in \bN^{n_k}_0$ with $|\alpha| \leq s / \omega_k$
for suitable constants $C = C(k,\,\alpha,\,\beta_1,\,\dots,\,\beta_m) > 0$.
We have
\begin{equation*}
	\partial^{\beta_j}_k u_j \in H^{s_j - |\beta_j| \omega_k, \omega}_{p_j}(\bR^n,\,E_j), \qquad j = 1,\,\dots,\,m
\end{equation*}
and the desired estimate follows by an application of H{\"o}lder's inequality,
provided that there exist $1 \leq q_1,\,\dots,\,q_m \leq \infty$, such that
\begin{equation}
	\eqnlabel{Proof-MA-2}
	H^{s_j - |\beta_j| \omega_k, \omega}_{p_j}(\bR^n,\,E_j) \hookrightarrow L_{q_j}(\bR^n,\,E_j), \qquad j = 1,\,\dots,\,m, \qquad \qquad \sum^m_{j = 1} \frac{1}{q_j} = \frac{1}{p}.
\end{equation}
Due to \eqnref{Bessel-Potential-Lebesgue-Embedding} for the embeddings in \eqnref{Proof-MA-2} to hold 
it is sufficient that $1 / p_j \geq 1 / q_j$ and
\begin{equation*}
	\frac{1}{\dot{\omega}} \left( s_j - |\beta_j| \omega_k - \frac{\omega \cdot n}{p_j} \right)
		   = \mbox{ind}\left(H^{s_j - |\beta_j| \omega_k, \omega}_{p_j}(\bR^n,\,A_j) \right)
		\geq - \frac{\omega \cdot n}{\dot{\omega}} \frac{1}{q_j}, \qquad j = 1,\,\dots,\,m,
\end{equation*}
with $q_j < \infty$, if $s_j - |\beta_j| \omega_k = (\omega \cdot n) / p_j$.
Hence, we necessarily have
\begin{equation*}
	- \frac{\omega \cdot n}{\dot{\omega}} \sum^m_{j = 1} \frac{1}{p_j} \leq - \frac{\omega \cdot n}{\dot{\omega}} \sum^m_{j = 1} \frac{1}{q_j} \leq \sum^m_{j = 1} \left[ \frac{1}{\dot{\omega}} \left( s_j - |\beta_j| \omega_k - \frac{\omega \cdot n}{p_j} \right) \right]_\ominus.
\end{equation*}
Therefore, due to the second necessary condition in \eqnref{Proof-MA-2} we need to ensure that
\begin{equation}
	\eqnlabel{Proof-MA-2-NC}
	- \frac{\omega \cdot n}{\dot{\omega}} \sum^m_{j = 1} \frac{1}{p_j} \leq - \frac{\omega \cdot n}{\dot{\omega}} \frac{1}{p} \leq \sum^m_{j = 1} \left[ \frac{1}{\dot{\omega}} \left( s_j - |\beta_j| \omega_k - \frac{\omega \cdot n}{p_j} \right) \right]_\ominus,
\end{equation}
where the first inequality is satisfied by assumption.
Based on \Lemref{Appendix-Minimization} (with \mbox{$n \sim |\alpha| \omega_k / (\omega \cdot n)$}, $\sigma_j \sim s_j / (\omega \cdot n)$, and $\pi_j \sim 1 / p_j$)
the minimal value $\lambda$ of the right-hand side of \eqnref{Proof-MA-2-NC} w.\,r.\,t.\ the decomposition $\beta_1 + \dots + \beta_m = \alpha$ is either
\begin{equation*}
	\lambda = 0 > - \frac{\omega \cdot n}{\dot{\omega}} \frac{1}{p},
\end{equation*}
if $\mbox{ind}_1,\,\dots,\,\mbox{ind}_m \geq |\alpha| \omega_k / \dot{\omega}$, or is given by
\begin{equation*}
	\lambda = \left\{
	\begin{array}{rl}
		{\displaystyle{\min^m_{j = 1}}}\,\mbox{ind}_j - |\alpha| \omega_k / \dot{\omega},                                      & \quad \mbox{if} \ \mbox{ind}_1,\,\dots,\,\mbox{ind}_m \geq 0, \\[1.0em]
		{\displaystyle{\sum^m_{\substack{j = 1 \\ \textrm{ind}_j < 0}}}} \!\! \mbox{ind}_j - |\alpha| \omega_k / \dot{\omega}, & \quad \mbox{otherwise}
	\end{array} \right\} \geq
	\mbox{ind} - s / \dot{\omega} = - \frac{\omega \cdot n}{\dot{\omega}} \frac{1}{p}
\end{equation*}
with strict inequality, if $\mbox{ind}_j = 0$ for some $j \in {\,1,\,\dots,\,m\,}$, since $\sum^m_{j = 1} |\beta_j| = |\alpha|$.
Moreover, \Lemref{Appendix-Minimization} shows that a minimal value $\lambda < 0$ is only realized by decompositions
$\beta_1 + \dots + \beta_m = \alpha$ that satisfy $\beta_j = 0$, if $\textrm{\ind}_j > 0$.
This ensures that the second inequality in \eqnref{Proof-MA-2-NC} is strict,
if $s_j - |\beta_j| \omega_k = (\omega \cdot n) / p_j$ for some $j \in \{\,1,\,\dots,\,m\,\}$.
Thus, we may apply \Lemref{Appendix-Realization} (with $\sigma_j \sim (s_j - |\beta_j| \omega_k) / (\omega \cdot n)$, $\pi_j \sim 1 / p_j$, $\rho \sim 1 / p$, and $\rho_j \sim 1 / q_j$)
to obtain $p_j \leq q_j \leq \infty$ for $j = 1,\,\dots,\,m$ with $q_j < \infty$ for all $j = 1,\,\dots,\,m$, if $s_j - |\beta_j| \omega_k = (\omega \cdot n) / p_j$
for some $j \in \{\,1,\,\dots,\,m\,\}$.
By construction, \eqnref{Proof-MA-2} holds for these values of the $q_j$,
which completes the proof for this case.

\subsection
	[The Case $X = H$, Part (b)]
	{The Case $\boldsymbol{X = H}$, $\boldsymbol{\sum^m_{j = 1} \frac{1}{p_j} = \frac{1}{p}}$}
	\subseclabel{Multiplication-Proof-Step-3}
Based on the considerations in \Subsecref{Multiplication-Proof-Preliminaries} we may assume $X_1 = \dots = X_m = X = H$.
Moreover, we may now additionally assume $s \in [0,\,\infty) \setminus \dot{\omega} \cdot \bN_0$,
since $s \in \dot{\omega} \cdot \bN_0$ is covered by \Subsecsref{Multiplication-Proof-Step-1}{Multiplication-Proof-Step-2}.
Therefore, we have $s < s^\prime := [s]^+_{\dot{\omega}} \in \dot{\omega} \cdot \bN$, where we denote by
\begin{equation*}
	[\sigma]^-_{\dot{\omega}} := \max\,\Big\{\,\kappa \dot{\omega}\,:\,\kappa \in \bN_0,\ \kappa \dot{\omega} \leq \sigma\,\Big\}, \quad
	[\sigma]^+_{\dot{\omega}} := \min\,\Big\{\,\kappa \dot{\omega}\,:\,\kappa \in \bN_0,\ \kappa \dot{\omega} \geq \sigma\,\Big\}, \quad \sigma \in [0,\,\infty),
\end{equation*}
the rounding functions w.\,r.\,t.\ multiples of $\dot{\omega}$.
Now, $s^\prime = s / \theta$ for some $0 < \theta < 1$.
If we set $s^\prime_j := s_j / \theta = s_j + \frac{1 - \theta}{\theta} s_j$ for $j = 1,\,\dots,\,m$, then we have
\begin{equation*}
	\mbox{ind}^\prime_j := \mbox{ind}\left( H^{s^\prime_j, \omega}_{p_j}(\bR^n,\,E_j) \right)
		= \frac{1}{\dot{\omega}} \left( s^\prime_j - \frac{\omega \cdot n}{p_j} \right)
		= \frac{1}{\dot{\omega}} \left( s_j - \frac{\omega \cdot n}{p_j} \right) + {\textstyle \frac{1 - \theta}{\theta \dot{\omega}}} s_j
		= \mbox{ind}_j + {\textstyle \frac{1 - \theta}{\theta \dot{\omega}}} s_j
\end{equation*}
and, analogously, $\mbox{ind}^\prime := \mbox{ind}(H^{s^\prime, \omega}_p(\bR^n,\,E)) = \mbox{ind} + \frac{1 - \theta}{\theta \dot{\omega}} s$.
Therefore, we have
\begin{equation}
	\eqnlabel{Proof-MA-3}
	\ind^\prime = \ind + {\textstyle \frac{1 - \theta}{\theta \dot{\omega}}} s
		\leq \sum_{j \in M} \mbox{ind}_j + \sum_{j \in M} {\textstyle \frac{1 - \theta}{\theta \dot{\omega}}} s_j
		= \sum_{j \in M} \mbox{ind}^\prime_j, \qquad \varnothing \neq M \subseteq \{\,1,\,\dots,\,m\,\},
\end{equation}
with strict inequality for $|M| \geq 2$.
For $M = \{\,j\,\}$ with $j \in \{\,1,\,\dots,\,m\,\}$ we may only have equality in \eqnref{Proof-MA-3},
if $s_j = s$ and $\mbox{ind}_j = \mbox{ind}$, which would imply $p_j = p$ in contradiction to the assumptions in this case.
Thus, \eqnref{Proof-MA-3} is always a strict inequality.
Hence, H{\"o}lder's inequality implies the induced multiplication
\begin{subequations}
\eqnlabel{Proof-MA-4}
\begin{equation}
	\bullet: L_{p_1}(\bR^n,\,E_1) \times \dots \times L_{p_m}(\bR^n,\,E_m) \longrightarrow L_p(\bR^n,\,E)
\end{equation}
to be continuous and by what has been proved in \Subsecref{Multiplication-Proof-Step-2} the induced multiplication
\begin{equation}
	\bullet: H^{s^\prime_1, \omega}_{p_1}(\bR^n,\,E_1) \times \dots \times H^{s^\prime_m, \omega}_{p_m}(\bR^n,\,E_m) \longrightarrow H^{s^\prime, \omega}_p(\bR^n,\,E)
\end{equation}
\end{subequations}
is continuous, too.
By construction we have
\begin{equation*}
	\begin{array}{c}
		\big[ L_{p_j}(\bR^n,\,E_j),\ H^{s^\prime_j, \omega}_{p_j}(\bR^n,\,E_j) \big]_\theta \doteq H^{s_j, \omega}_{p_j}(\bR^n,\,E_j), \qquad j = 1,\,\dots,\,m, \\[1.0em]
		\big[ L_p(\bR^n,\,E),\ H^{s^\prime, \omega}_p(\bR^n,\,E) \big]_\theta \doteq H^{s, \omega}_p(\bR^n,\,E)
	\end{array}
\end{equation*}
thanks to \eqnref{Bessel-Potential-Complex-Interpolation}
and we obtain the continuity of the induced multiplication
\begin{equation*}
	\bullet: H^{s_1, \omega}_{p_1}(\bR^n,\,E_1) \times \dots \times H^{s_m, \omega}_{p_m}(\bR^n,\,E_m) \longrightarrow H^{s, \omega}_p(\bR^n,\,E)
\end{equation*}
by multilinear complex interpolation \cite{Triebel:Interpolation} w.\,r.\,t.\ $[\,\cdot\,,\,\cdot\,]_\theta$ of \eqnref{Proof-MA-4},
which completes the proof for this case.

\subsection
	[The Case $X = B$, Part (a)]
	{The Case $\boldsymbol{X = B}$, $\boldsymbol{s \in [0,\,\infty)\,\setminus\,} \pmb{\bQ}$}
	\subseclabel{Multiplication-Proof-Step-4}
In this case we necessarily have $s > 0$ due to constraint~(b) as well as \mbox{$(s - |\alpha| \omega_k) / \omega_\ell \notin \bN$}
for all $k,\,\ell \in \{\,1,\,\dots,\,\nu\,\}$ and all $\alpha \in \bN^{n_k}_0$ with $|\alpha| \leq s / \omega_k$.
In particular, due to \eqnref{BW-Anisotropic} we have
\begin{equation*}
	X^{s, \omega}_p(\bR^n,\,E) = B^{s, \omega}_p(\bR^n,\,E) = W^{s, \omega}_p(\bR^n,\,E)
\end{equation*}
and based on the considerations in \Subsecref{Multiplication-Proof-Preliminaries} we may assume $X_1 = \dots = X_m = X = B$.
Now, thanks to \Propref{Characterization-Anisotropic-Besov},
\begin{equation*}
	\|u\|_{L_p(\bR^n)} + \sum^\nu_{k = 1} \!\!\!\! \sum_{\substack{\alpha \in \bN^{n_k}_0 \\ |\alpha| = [s / \omega_k]}} \!\!\!\!
		\left\| |h|^{|\alpha| - s / \omega_k} \|\Delta^h_k \partial^\alpha_k u\|_{L_p(\bR^n, E)} \right\|_{L_p(\bR^{n_k}, |h|^{- n_k} \textrm{d}h)}, \qquad
		u \in B^{s, \omega}_p(\bR^n,\,E),
\end{equation*}
defines an equivalent norm on $B^{s, \omega}_p(\bR^n,\,E)$.
By \Subsecsref{Multiplication-Proof-Preliminaries}{Multiplication-Proof-Step-1}, we have that
\begin{equation*}
	\|u_1 \bullet \ldots \bullet u_m\|_{L_p(\bR^n, E)} \leq C \|u_1\|_{B^{s_1, \omega}_{p_1}(\bR^n, E_1)} \cdot \ldots \cdot \|u_m\|_{B^{s_m, \omega}_{p_m}(\bR^n, E_m)}
\end{equation*}
for some constant $C > 0$.
Thus, it remains to estimate the terms
\begin{equation*}
	\left\| |h|^{|\alpha| - s / \omega_k} \|\Delta^h_k \partial^\alpha_k (u_1 \bullet \ldots \bullet u_m)\|_{L_p(\bR^n, E)} \right\|_{L_p(\bR^{n_k}, |h|^{- n_k} \textrm{d}h)}, \qquad
		\begin{array}{c} u_j \in B^{s_j, \omega}_{p_j}(\bR^n,\,E_j), \\[0.5em] j = 1,\,\dots,\,m, \end{array}
\end{equation*}
for all $k = 1,\,\dots,\,\nu$ and all multiindices $\alpha \in \bN^{n_k}_0$ with $|\alpha| = [s / \omega_k]$.
Now, we have
\begin{equation*}
	\begin{array}{l}
		\Delta^h_k \partial^\alpha_k (u_1 \bullet \ldots \bullet u_m) \\[0.5em] \qquad \qquad =
		\!\!\!\!\!\!\!\!\!\!\!\! {\displaystyle \sum_{\substack{\beta_1,\,\dots,\,\beta_m \in \bN^{n_k}_0 \\ \beta_1 + \ldots + \beta_m = \alpha}}}
		\!\!\!\!\!\!\!\!\!\!\!\! {\textstyle \frac{\alpha!}{\beta_1! \cdot \ldots \cdot \beta_m!}}\,{\displaystyle \sum^m_{\ell = 1} \tau^h_k \partial^{\beta_1}_k u_1 \bullet \ldots \bullet \tau^h_k \partial^{\beta_{\ell - 1}}_k u_{\ell - 1} \bullet \Delta^h_k \partial^{\beta_\ell}_k u_\ell \bullet \partial^{\beta_{\ell + 1}}_k u_{\ell + 1} \bullet \ldots \bullet \partial^{\beta_m}_k u_m}
	\end{array}
\end{equation*}
and, thus, it is necessary and sufficient to obtain an estimate
\begin{equation}
	\eqnlabel{Proof-MA-5}
	\begin{array}{l}
		\left\| \tau^h_k \partial^{\beta_1}_k u_1 \bullet \ldots \bullet \tau^h_k \partial^{\beta_{\ell - 1}}_k u_{\ell - 1} \bullet \Delta^h_k \partial^{\beta_\ell}_k u_\ell \bullet \partial^{\beta_{\ell + 1}}_k u_{\ell + 1} \bullet \ldots \bullet \partial^{\beta_m}_k u_m \right\|_{L_p(\bR^n, E)} \\[1.5em] \qquad \qquad \qquad \qquad \qquad \qquad \qquad \leq
			C\,{\displaystyle \prod^m_{\substack{j = 1 \\ j \neq \ell}} \|u_j\|_{B^{s_j, \omega}_{p_j}(\bR^n, E_j)} \cdot \|\Delta^h_k \partial^{\beta_\ell}_k u_\ell\|_{L_{q_\ell}(\bR^n, E_\ell)}}
	\end{array}
\end{equation}
for $\ell = 1,\,\dots,\,m$ and all possible decompositions $\beta_1,\,\dots,\,\beta_m \in \bN^{n_k}_0$
of a given multiindex $\alpha \in \bN^{n_k}_0$ with $|\alpha| = [s / \omega_k]$
for suitable constants $C = C(k,\,\ell,\,\alpha,\,\beta_1,\,\dots,\,\beta_m) > 0$ and some suitably chosen $1 \leq q_\ell \leq \infty$ such that
\begin{subequations}\eqnlabel{Proof-MA-6}
\begin{equation}
	\eqnlabel{Proof-MA-6a}
	B^{s_\ell, \omega}_{p_\ell}(\bR^n,\,E_\ell) \hookrightarrow B^{s - (|\alpha| - |\beta_\ell|) \omega_k, \omega}_{q_\ell, p}(\bR^n,\,E_\ell).
\end{equation}
Indeed, in this case we may use \eqnref{Besov-Embedding-Partial} to obtain
\begin{equation*}
	\begin{array}{l}
		\left\| |h|^{|\alpha| - s / \omega_k} \|\Delta^h_k \partial^{\beta_\ell}_k u_\ell \|_{L_{q_\ell}(\bR^n, E_\ell)} \right\|_{L_p(\bR^{n_k}, |h|^{- n_k} \textrm{d}h)}
			\leq C \|u_\ell\|_{B^{s / \omega_k - (|\alpha| - |\beta_\ell|)}_{q_\ell, p}(\bR^{n_k}, L_{q_\ell}(\bR^{n^\prime_k}, E_\ell))} \\[1.5em] \qquad \qquad
			\leq C \|u_\ell\|_{B^{s - (|\alpha| - |\beta_\ell|) \omega_k, \omega}_{q_\ell, p}(\bR^n, E_\ell)}
			\leq C \|u_\ell\|_{B^{s_\ell, \omega}_{p_\ell}(\bR^n, E_\ell)}.
	\end{array}
\end{equation*}
Now, according to \eqnref{Besov-Embedding}, the embedding \eqnref{Proof-MA-6a} holds, provided that $1 / p_\ell \geq 1 / q_\ell$ and
\begin{equation*}\tag{$\ast$}
	\frac{1}{\dot{\omega}} \left( s_\ell - \frac{\omega \cdot n}{p_\ell} \right) \geq \frac{1}{\dot{\omega}} \left( s - (|\alpha| - |\beta_\ell|) \omega_k - \frac{\omega \cdot n}{q_\ell} \right)
\end{equation*}
with strict inequality, if $p_\ell > p$.
Note that $s_\ell = s - (|\alpha| - |\beta_\ell|) \omega_k$ precisely, if $s_\ell = s$ and $|\beta_\ell| = |\alpha|$.
However, if $s_\ell = s$, then we have $p_\ell = p$ and the above constraints imply $q_\ell = p_\ell$, if $s_\ell = s - (|\alpha| - |\beta_\ell|) \omega_k$.
Moreover, we have
\begin{equation*}
	\partial^{\beta_j}_k u_j \in B^{s_j - |\beta_j| \omega_k, \omega}_{p_j}(\bR^n,\,E_j), \qquad j = 1,\,\dots,\,m,\ j \neq \ell,
\end{equation*}
and the estimate \eqnref{Proof-MA-5} may be obtained via H{\"o}lders inequality and the translation invariance of the Lebesgue measure,
provided that there exist $1 \leq q_1,\,\dots,\,q_{\ell - 1},\,q_{\ell + 1},\,\dots,\,q_m \leq \infty$ such that
\begin{equation}
	\eqnlabel{Proof-MA-6b}
	B^{s_j - |\beta_j| \omega_k, \omega}_{p_j}(\bR^n,\,E_j) \hookrightarrow L_{q_j}(\bR^n,\,E_j), \qquad j = 1,\,\dots,\,m,\ j \neq \ell, \qquad \qquad \sum^m_{j = 1} \frac{1}{q_j} = \frac{1}{p}.
\end{equation}
\end{subequations}
Due to \eqnref{Besov-Lebesgue-Embedding} for the embeddings in \eqnref{Proof-MA-6b} to hold it is sufficient that $1 / p_j \geq 1 / q_j$ and
\begin{equation*}
	\frac{1}{\dot{\omega}} \left( s_j - |\beta_j| \omega_k - \frac{\omega \cdot n}{p_j} \right)
		   = \mbox{ind}\left(B^{s_j - |\beta_j| \omega_k, \omega}_{p_j}(\bR^n,\,A_j) \right)
		\geq - \frac{\omega \cdot n}{\dot{\omega}} \frac{1}{q_j}, \qquad j = 1,\,\dots,\,m,\ j \neq \ell
\end{equation*}
with $q_j < \infty$, if $s_j - |\beta_j| \omega_k = (\omega \cdot n) / p_j$.
Hence, we necessarily have
\begin{equation*}
	- \frac{\omega \cdot n}{\dot{\omega}} \sum^m_{j = 1} \frac{1}{p_j} \leq - \frac{\omega \cdot n}{\dot{\omega}} \sum^m_{j = 1} \frac{1}{q_j} \leq \sum^m_{j = 1} \left[ \frac{1}{\dot{\omega}} \left( s^\prime_j - |\beta_j| \omega_k - \frac{\omega \cdot n}{p_j} \right) \right]_\ominus.
\end{equation*}
with $s^\prime_j := s_j$ for $j = 1,\,\dots,\,m,\ j \neq \ell$ and with $s^\prime_\ell := s_\ell - (s - |\alpha| \omega_k)$.
Therefore, due to the second necessary condition in \eqnref{Proof-MA-6b} we need to ensure that
\begin{equation}
	\eqnlabel{Proof-MA-6-NC}
	- \frac{\omega \cdot n}{\dot{\omega}} \sum^m_{j = 1} \frac{1}{p_j} \leq - \frac{\omega \cdot n}{\dot{\omega}} \frac{1}{p} \leq \sum^m_{j = 1} \left[ \frac{1}{\dot{\omega}} \left( s^\prime_j - |\beta_j| \omega_k - \frac{\omega \cdot n}{p_j} \right) \right]_\ominus,
\end{equation}
where the first inequality is satisfied by assumption.
Now, if we set
\begin{equation*}
	\textrm{ind}^\prime_j := \frac{1}{\dot{\omega}} \left( s^\prime_j - \frac{\omega \cdot n}{p_j} \right) = \left\{
	\begin{array}{ll}
		\textrm{ind}_j,                                          & \quad \textrm{if}\ j \neq \ell, \\[0.5em]
		\textrm{ind}_j - (s - |\alpha| \omega_k) / \dot{\omega}, & \quad \textrm{if}\ j = \ell,
	\end{array}
	\right. \qquad j = 1,\,\dots,\,m,
\end{equation*}
then we have
\begin{equation*}
	\textrm{ind}^\prime_\ell - |\alpha| \omega_k / \dot{\omega} = \textrm{ind}_\ell - s / \dot{\omega} \qquad \textrm{and} \qquad
	\textrm{ind}^\prime_j    - |\alpha| \omega_k / \dot{\omega} > \textrm{ind}_j    - s / \dot{\omega}, \quad j = 1,\,\dots,\,m,\ j \neq \ell,
\end{equation*}
which implies
\begin{equation*}
	\min^m_{j = 1}\,\mbox{ind}^\prime_j - |\alpha| \omega_k / \dot{\omega}
		\ >\ \min^m_{j = 1}\,\mbox{ind}_j - s / \dot{\omega}, \qquad \textrm{if}\ \ind^\prime_j < \ind^\prime_\ell\ \textrm{for some}\ j \neq \ell,
\end{equation*}
while
\begin{equation*}
	\min^m_{j = 1}\,\mbox{ind}^\prime_j - |\alpha| \omega_k / \dot{\omega}
		\ =\ \min^m_{j = 1}\,\mbox{ind}_j - s / \dot{\omega}, \qquad \textrm{if}\ \ind^\prime_\ell \leq \ind^\prime_j\ \textrm{for all}\ j \neq \ell.
\end{equation*}
Moreover, we have
\begin{equation*}
	\sum^m_{\substack{j = 1 \\ \textrm{ind}^\prime_j < 0}} \!\! \mbox{ind}^\prime_j - |\alpha| \omega_k / \dot{\omega}
		> \!\! \sum^m_{\substack{j = 1 \\ \textrm{ind}_j < 0}} \!\! \mbox{ind}_j - s / \dot{\omega}, \qquad \textrm{if}\ \textrm{ind}^\prime_\ell \geq 0,
\end{equation*}
while
\begin{equation*}
	\sum^m_{\substack{j = 1 \\ \textrm{ind}^\prime_j < 0}} \!\! \mbox{ind}^\prime_j - |\alpha| \omega_k / \dot{\omega}
		\, = \!\!\!\!\!\!\!\! \sum^m_{\substack{j = 1 \\ \textrm{ind}_j < 0,\,j \neq \ell}} \!\!\!\!\!\!\! \mbox{ind}_j + \textrm{ind}_\ell - s / \dot{\omega}
		\, = \!\! \sum^m_{\substack{j = 1 \\ \textrm{ind}_j < 0}} \!\! \mbox{ind}_j - s / \dot{\omega}, \qquad \textrm{if}\ \textrm{ind}^\prime_\ell < 0.
\end{equation*}
According to \Lemref{Appendix-Minimization} (with $n \sim |\alpha| \omega_k / (\omega \cdot n)$, $\sigma_j \sim s_j / (\omega \cdot n)$, and $\pi_j \sim 1 / p_j$)
the minimal value $\lambda$ of the right-hand side of \eqnref{Proof-MA-6-NC} w.\,r.\,t.\ the decomposition $\beta_1 + \dots + \beta_m = \alpha$ is either
\begin{equation*}
	\lambda = 0 > - \frac{\omega \cdot n}{\dot{\omega}} \frac{1}{p},
\end{equation*}
if $\mbox{ind}^\prime_1,\,\dots,\,\mbox{ind}^\prime_m \geq |\alpha| \omega_k / \dot{\omega}$, or is given by
\begin{equation*}
	\lambda = \left\{
	\begin{array}{rl}
		{\displaystyle{\min^m_{j = 1}}}\,\mbox{ind}^\prime_j - |\alpha| \omega_k / \dot{\omega},                                             & \quad \mbox{if} \ \mbox{ind}^\prime_1,\,\dots,\,\mbox{ind}^\prime_m \geq 0, \\[1.0em]
		{\displaystyle{\sum^m_{\substack{j = 1 \\ \textrm{ind}^\prime_j < 0}}}} \!\! \mbox{ind}^\prime_j - |\alpha| \omega_k / \dot{\omega}, & \quad \mbox{otherwise}
	\end{array} \right\} \geq
	\mbox{ind} - s / \dot{\omega} = - \frac{\omega \cdot n}{\dot{\omega}} \frac{1}{p}
\end{equation*}
with strict inequality, if $\mbox{ind}_j = 0$ for some $j \in {\,1,\,\dots,\,m\,}$, since $\sum^m_{j = 1} |\beta_j| = |\alpha|$.
Moreover, \Lemref{Appendix-Minimization} shows that a minimal value $\lambda < 0$ is only realized by decompositions
$\beta_1 + \dots + \beta_m = \alpha$ that satisfy $\beta_j = 0$, if $\textrm{\ind}^\prime_j > 0$.
This ensures that the second inequality in \eqnref{Proof-MA-6-NC} is strict,
if $s_j - |\beta_j| \omega_k = (\omega \cdot n) / p_j$ for some $j \in \{\,1,\,\dots,\,m\,\}$ with $j \neq \ell$.
Thus, we may apply \Lemref{Appendix-Realization} (with $\sigma_j \sim (s^\prime_j - |\beta_j| \omega_k) / (\omega \cdot n)$, $\pi_j \sim 1 / p_j$, $\rho \sim 1 / p$, and $\rho_j \sim 1 / q_j$)
to obtain $p_j \leq q_j \leq \infty$ for $j = 1,\,\dots,\,m$ with $q_j < \infty$ for all $j = 1,\,\dots,\,m$,
if $s_j - |\beta_j| \omega_k = (\omega \cdot n) / p_j$ for some $j \in \{\,1,\,\dots,\,m\,\}$ with $j \neq \ell$.
Also note that equality in ($\ast$) may only occur, either if $s^\prime_\ell - |\beta_\ell| \omega_k = 0$, which implies $p_\ell = q_\ell = p$ as shown above,
or if the second inequality in \eqnref{Proof-MA-6-NC} is an equality and $\ind^\prime_\ell \leq 0$.
The latter case can only occur, if $\beta_1,\,\dots,\,\beta_m$ minimize the right-hand side of \eqnref{Proof-MA-6-NC} to a value $\lambda < 0$,
and if $\ind^\prime_j \geq \ind^\prime_\ell \geq 0$ for all $j = 1,\,\dots,\,m$ or $\ind^\prime_\ell < 0$.
In the first of the latter cases we have
\begin{equation*}
	- \frac{\omega \cdot n}{\dot{\omega}} \frac{1}{p}
		= \sum^m_{j = 1} \left[ \frac{1}{\dot{\omega}} \left( s^\prime_j - |\beta_j| \omega_k - \frac{\omega \cdot n}{p_j} \right) \right]_\ominus
		= \ \lambda \ = \ \min^m_{j = 1}\,\mbox{ind}_j - s / \dot{\omega} \ = \ \ind - s / \dot{\omega}
\end{equation*}
while in the second of the latter cases it holds that
\begin{equation*}
	- \frac{\omega \cdot n}{\dot{\omega}} \frac{1}{p}
		= \sum^m_{j = 1} \left[ \frac{1}{\dot{\omega}} \left( s^\prime_j - |\beta_j| \omega_k - \frac{\omega \cdot n}{p_j} \right) \right]_\ominus
		= \ \lambda \ = \sum^m_{\substack{j = 1 \\ \textrm{ind}_j < 0}} \!\! \mbox{ind}_j - s / \dot{\omega} \ = \ \ind - s / \dot{\omega}
\end{equation*}
In both of the latter cases we have equality in (iii) of \Thmref{Multiplication-Anisotropic}
and constraint (c) of \Thmref{Multiplication-Anisotropic} implies $p_\ell \leq p$.
Thus, by construction, \eqnref{Proof-MA-6} holds for these values of the $q_j$,
which completes the proof for this case. \vspace*{-0.5\baselineskip}

\subsection
	[The Case $X = B$, Part (b)]
	{The Case $\boldsymbol{X = B}$}
	\subseclabel{Multiplication-Proof-Step-5}
In this case we necessarily have $s > 0$ and based on the considerations in \Subsecref{Multiplication-Proof-Preliminaries} we may assume $X_1 = \dots = X_m = X = B$.
If $s \notin \bQ$, then the assertion of \Thmref{Multiplication-Anisotropic} has already been proved in \Subsecref{Multiplication-Proof-Step-4}.
Thus, we could assume $s \in \bQ$ here, which, however, is not necessary.
We simply choose $0 < \lambda < 1 < \mu < \min\,\{\,p_1,\,\dots,\,p_m,\,p\,\}$ such that $\alpha s,\,\mu s \notin \bQ$.
Then there exists $0 < \theta < 1$ such that $(1 - \theta) \lambda + \theta \mu = 1$
and the considerations in \Subsecref{Multiplication-Proof-Step-4} imply the induced multiplications
\begin{equation}
	\eqnlabel{Proof-MA-7}
	\begin{array}{clcl}
		\bullet: & B^{\lambda s_1, \omega}_{p_1 / \lambda}(\bR^n,\,E_1) \times \dots \times B^{\lambda s_m, \omega}_{p_m / \lambda}(\bR^n,\,E_m) & \longrightarrow & B^{\lambda s, \omega}_{p / \lambda}(\bR^n,\,E), \\[0.5em]
		\bullet: & B^{\mu s_1, \omega}_{p_1 / \mu}(\bR^n,\,E_1)         \times \dots \times B^{\mu s_m, \omega}_{p_m / \mu}(\bR^n,\,E_m)         & \longrightarrow & B^{\mu s, \omega}_{p / \mu}(\bR^n,\,E)
	\end{array}
\end{equation}
to be continuous.
By construction we have
\begin{equation*}
	\begin{array}{c}
		\big[ B^{\lambda s_j, \omega}_{p_j / \lambda}(\bR^n,\,E_j),\ B^{\mu s_j, \omega}_{p_j / \mu}(\bR^n,\,E_j) \big]_\theta \doteq B^{s_j, \omega}_{p_j}(\bR^n,\,E_j), \qquad j = 1,\,\dots,\,m, \\[1.0em]
		\big[ B^{\lambda s, \omega}_{p / \lambda}(\bR^n,\,E),\ B^{\mu s, \omega}_{p / \mu}(\bR^n,\,E) \big]_\theta \doteq B^{s, \omega}_p(\bR^n,\,E)
	\end{array}
\end{equation*}
thanks to \eqnref{Besov-Complex-Interpolation}
and we obtain the continuity of the induced multiplication
\begin{equation*}
	\bullet: B^{s_1, \omega}_{p_1}(\bR^n,\,E_1) \times \dots \times B^{s_m, \omega}_{p_m}(\bR^n,\,E_m) \longrightarrow B^{s, \omega}_p(\bR^n,\,E)
\end{equation*}
by multilinear complex interpolation w.\,r.\,t.\ $[\,\cdot\,,\,\cdot\,]_\theta$ of \eqnref{Proof-MA-7},
which completes the proof for this case.

\subsection
	[The Case $X = H$, Part (c)]
	{The Case $\boldsymbol{X = H}$, $\boldsymbol{s_1,\,\dots,\,s_m > s}$, $\boldsymbol{\sum^m_{j = 1} \frac{1}{p_j} > \frac{1}{p}}$}
	\subseclabel{Multiplication-Proof-Step-6}
Based on the considerations in \Subsecref{Multiplication-Proof-Preliminaries} we may assume $X_1 = \dots = X_m = X = H$.
Moreover, we may choose an $0 < \epsilon < 1 - \frac{1}{p}$ such that $s^\prime_j := s_j - \epsilon > s + \epsilon =: s^\prime$ for $j = 1,\,\dots,\,m$ as well as
\begin{equation*}
	\sum^m_{j = 1} \frac{1}{p^\prime_j} = \sum^m_{j = 1} \frac{1}{p_j} - \frac{m \epsilon}{\omega \cdot n} > \frac{1}{p} + \frac{\epsilon}{\omega \cdot n} = \frac{1}{p^\prime} \qquad
	\textrm{for} \quad \frac{\omega \cdot n}{p^\prime_j} := \frac{\omega \cdot n}{p_j} - \epsilon > 0, \quad \frac{\omega \cdot n}{p^\prime} := \frac{\omega \cdot n}{p} + \epsilon.
\end{equation*}
Then we have $s_j > s^\prime_j$ and $1 / p_j > 1 / p^\prime_j$ as well as $\textrm{ind}(H^{s_j, \omega}_{p_j}(\bR^n,\,E_j)) = \textrm{ind}(B^{s^\prime_j, \omega}_{p^\prime_j}(\bR^n,\,E_j))$
by construction for all $j = 1,\,\dots,\,m$.
Moreover, we have $s^\prime > s$ and $1 / p^\prime > 1 / p$ as well as $\textrm{ind}(B^{s^\prime, \omega}_{p^\prime}(\bR^n,\,E)) = \textrm{ind}(H^{s, \omega}_p(\bR^n,\,E))$.
Using \eqnref{Bessel-Potential-Besov-Embedding}, the result from \Subsecref{Multiplication-Proof-Step-5}, and \eqnref{Besov-Bessel-Potential-Embedding} we obtain
\begin{equation*}
	\begin{array}{l}
		H^{s_1, \omega}_{p_1}(\bR^n,\,E_1) \bullet \dots \bullet H^{s_m, \omega}_{p_m}(\bR^n,\,E_m) \\[0.5em]
			\qquad \qquad \hookrightarrow B^{s^\prime_1, \omega}_{p^\prime_1}(\bR^n,\,E_1) \bullet \dots \bullet B^{s^\prime_m, \omega}_{p^\prime_m}(\bR^n,\,E_m)
				\hookrightarrow B^{s^\prime, \omega}_{p^\prime}(\bR^n,\,E) \hookrightarrow H^{s, \omega}_p(\bR^n,\,E).
	\end{array}
\end{equation*}
This completes the proof for this case.

Since all possible cases are covered now, the proof of \Thmref{Multiplication-Anisotropic} is complete. \qedbox

\section{Proof of \Thmref{Multiplier-Anisotropic}: Multipliers for Anisotropic Function Spaces}
\seclabel{Multiplier-Proof}
In this section we deduce \Thmref{Multiplier-Anisotropic} from \Thmref{Multiplication-Anisotropic}.
We assume that $m,\,\nu \in \bN$ and $n,\,\omega \in \bN^\nu$, where w.\,o.\,l.\,g.\ $m \geq 2$.
Moreover, we assume $E_1,\,\dots,\,E_m$ and $E$ to be UMD-spaces that allow for a multiplication \eqnref{Multiplication},
and that have property~$(\alpha)$ if $\omega \neq \dot{\omega} \cdot (1,\,\dots,\,1)$.
Furthermore, we assume that $X_1,\,\dots,\,X_m,\,X \in \{\,B,\,H\,\}$,
that $0 \leq s \leq s_1,\,\dots,\,s_m < \infty$, and that $1 < p_1,\,\dots,\,p_m,\,p < \infty$
such that $X_\ell = X$, $s_\ell = s$ and $p_\ell = p$ for some $\ell \in \{\,1,\,\dots,\,m\,\}$.
Moreover, we employ the notation of \Thmref{Multiplier-Anisotropic} and set
\begin{equation*}
	\textrm{\upshape ind}_j := \textrm{\upshape ind}([X_j]^{s_j, \omega}_{p_j}(\bR^n,\,E_j)), \quad j = 1,\,\dots,\,m \qquad \textrm{and} \qquad
	\textrm{\upshape ind} := \textrm{\upshape ind}(X^{s, \omega}_p(\bR^n,\,E)).
\end{equation*}
Now, we have to distinguish two cases.

\subsection
	[Part (a)]
	{The Case $\boldsymbol{\ind_\ell = \ind \neq 0}$}
	\subseclabel{Multiplier-Proof-Step-1}
Due to the above assumptions it is obvious that inequality (i) of \Thmref{Multiplication-Anisotropic} is satisfied as an equality,
and that inequality (ii) of \Thmref{Multiplication-Anisotropic} is satisfied as a strict inequality.
Note that this also implies constraint (e) of \Thmref{Multiplication-Anisotropic} to be satisfied.
Taking into account the equivalent formulation of (iii) of \Thmref{Multiplication-Anisotropic} given in \Remref{Multiplication-Anisotropic}~(c)
for $M = \{\,j\,\}$ and $M = \{\,\ell,\,j\,\}$ with $j \in \{\,1,\,\dots,\,m\,\} \setminus \{\,\ell\,\}$,
we deduce that (iii) of \Thmref{Multiplication-Anisotropic} is satisfied if and only if $\ind_j \geq \min\,\{\,\ind,\,0\,\}$
for all $j \in \{\,1,\,\dots,\,m\,\} \setminus \{\,\ell\,\}$.
Therefore, under the assumptions of \Thmref{Multiplier-Anisotropic} the inequality (iii) of \Thmref{Multiplication-Anisotropic} is satisfied as an equality.
Since we assume $\ind_j > 0$ for all $j \in \{\,1,\,\dots,\,m\,\} \setminus \{\,\ell\,\}$
as well as $\ind_\ell = \ind \neq 0$, we infer that constraint (f) of \Thmref{Multiplication-Anisotropic} is satisfied.
Hence, it remains to check the constraints (a)--(d) of \Thmref{Multiplication-Anisotropic}.
However, constraint (a) of \Thmref{Multiplication-Anisotropic} is equivalent to constraint (a) of \Thmref{Multiplier-Anisotropic},
and, based on the above observations concerning the inequalities (i)--(iii),
constraints (b) and (c) of \Thmref{Multiplication-Anisotropic} are equivalent to constraint (b) of \Thmref{Multiplier-Anisotropic},
and constraint (d) of \Thmref{Multiplication-Anisotropic} is equivalent to constraint (c) of \Thmref{Multiplier-Anisotropic}.
Thus, \Thmref{Multiplication-Anisotropic} implies the induced multiplication
\begin{equation*}
	\bullet: [X_1]^{s_1, \omega}_{p_1}(\bR^n,\,E_1) \times \dots \times [X_m]^{s_m, \omega}_{p_m}(\bR^n,\,E_m) \longrightarrow X^{s, \omega}_p(\bR^n,\,E)
\end{equation*}
to be continuous.

\subsection
	[Part (b)]
	{The Case $\boldsymbol{\ind_\ell = \ind = 0}$}
	\subseclabel{Multiplier-Proof-Step-2}
In this case we have $\ind_j > 0 = \ind$ for all $j \in \{\,1,\,\dots,\,m\,\} \setminus \{\,\ell\,\}$.
Therefore, we may choose $0 < \lambda < 1 < \mu < \min\,\{\,p_1,\,\dots,\,p_m,\,p\,\}$ such that
\begin{equation*}
	\ind_j > (\mu - 1) \frac{\omega \cdot n}{\dot{\omega}} \frac{1}{p_j}, \qquad \qquad j \in \{\,1,\,\dots,\,m\,\} \setminus \{\,\ell\,\},
\end{equation*}
as well as
\begin{equation*}
	\ind_j - \ind \geq \max \left\{ (1 - \lambda) \frac{\omega \cdot n}{\dot{\omega}} \left( \,\frac{1}{p} - \frac{1}{p_j} \right),\ (\mu - 1) \frac{\omega \cdot n}{\dot{\omega}} \left( \frac{1}{p_j} - \frac{1}{p}\, \right) \right\}, \quad j = 1,\,\dots,\,m, \ j \neq \ell.
\end{equation*}
Then there exists $0 < \theta < 1$ such that $(1 - \theta) \lambda + \theta \mu = 1$ and we have
\begin{equation*}
	\ind([X_\ell]^{s_\ell, \omega}_{p_\ell / \lambda}(\bR^n,\,E_\ell))
		= \ind(X^{s, \omega}_{p / \lambda}(\bR^n,\,E))
		= \frac{1}{\dot{\omega}} \left( s - \frac{\omega \cdot n}{p / \lambda} \right)
		= \ind + (1 - \lambda) \frac{\omega \cdot n}{\dot{\omega}} \frac{1}{p}
		> 0 \phantom{.}
\end{equation*}
as well as
\begin{equation*}
	\ind([X_\ell]^{s_\ell, \omega}_{p_\ell / \mu}(\bR^n,\,E_\ell))
		= \ind(X^{s, \omega}_{p / \mu}(\bR^n,\,E))
		= \frac{1}{\dot{\omega}} \left( s - \frac{\omega \cdot n}{p / \mu} \right)
		= \ind - (\mu - 1) \frac{\omega \cdot n}{\dot{\omega}} \frac{1}{p}
		< 0.
\end{equation*}
We also obtain
\begin{equation*}
	\ind([X_j]^{s_j, \omega}_{p_j / \lambda}(\bR^n,\,E_j)) = \ind_j + (1 - \lambda) \frac{\omega \cdot n}{\dot{\omega}} \frac{1}{p_j}
		\geq \ind(X^{s, \omega}_{p / \lambda}(\bR^n,\,E)) > 0
\end{equation*}
for all $j \in \{\,1,\,\dots,\,m\,\} \setminus \{\,\ell\,\}$ as well as
\begin{equation*}
	\ind([X_j]^{s_j, \omega}_{p_j / \mu}(\bR^n,\,E_j)) = \ind_j - (\mu - 1) \frac{\omega \cdot n}{\dot{\omega}} \frac{1}{p_j}
		\geq \ind(X^{s, \omega}_{p / \mu}(\bR^n,\,E)),
\end{equation*}
which also implies that $\ind([X_j]^{s_j, \omega}_{p_j / \mu}(\bR^n,\,E_j)) > 0$ for all $j \in \{\,1,\,\dots,\,m\,\} \setminus \{\,\ell\,\}$
due to the choice of $\mu$.
Therefore, the considerations in \Subsecref{Multiplier-Proof-Step-1} imply the induced multiplications
\begin{equation}
	\eqnlabel{Proof-Multiplier}
	\begin{array}{clcl}
		\bullet: & [X_1]^{s_1, \omega}_{p_1 / \lambda}(\bR^n,\,E_1) \times \dots \times [X_m]^{s_m, \omega}_{p_m / \lambda}(\bR^n,\,E_m) & \longrightarrow & X^{s, \omega}_{p / \lambda}(\bR^n,\,E), \\[0.5em]
		\bullet: & [X_1]^{s_1, \omega}_{p_1 / \mu}(\bR^n,\,E_1)     \times \dots \times [X_m]^{s_m, \omega}_{p_m / \mu}(\bR^n,\,E_m)     & \longrightarrow & X^{s, \omega}_{p / \mu}(\bR^n,\,E)
	\end{array}
\end{equation}
to be continuous.
By construction we have
\begin{equation*}
	\begin{array}{c}
		\big[ [X_j]^{s_j, \omega}_{p_j / \lambda}(\bR^n,\,E_j),\ [X_j]^{s_j, \omega}_{p_j / \mu}(\bR^n,\,E_j) \big]_\theta \doteq [X_j]^{s_j, \omega}_{p_j}(\bR^n,\,E_j), \qquad j = 1,\,\dots,\,m, \\[1.0em]
		\big[ X^{s, \omega}_{p / \lambda}(\bR^n,\,E),\ X^{s, \omega}_{p / \mu}(\bR^n,\,E) \big]_\theta \doteq X^{s, \omega}_p(\bR^n,\,E)
	\end{array}
\end{equation*}
thanks to \eqnref{Besov-Complex-Interpolation} and \eqnref{Bessel-Potential-Complex-Interpolation}
and we obtain the continuity of the induced multiplication
\begin{equation*}
	\bullet: [X_1]^{s_1, \omega}_{p_1}(\bR^n,\,E_1) \times \dots \times [X_m]^{s_m, \omega}_{p_m}(\bR^n,\,E_m) \longrightarrow X^{s, \omega}_p(\bR^n,\,E)
\end{equation*}
by multilinear complex interpolation w.\,r.\,t.\ $[\,\cdot\,,\,\cdot\,]_\theta$ of \eqnref{Proof-Multiplier},
which completes the proof for this case.

Since all possible cases are covered now, the proof of \Thmref{Multiplier-Anisotropic} is complete. \qedbox

\section{Proof of \Thmref{Nemytskij-Anisotropic}: Nemytskij Operators in Anisotropic Functions Spaces}\seclabel{Nemytskij-Proof}\vspace*{-0.75\baselineskip}
The proof of \Thmref{Nemytskij-Anisotropic} requires several steps.
In the following we assume the requirements of \Thmref{Nemytskij-Anisotropic} to be valid and we employ the notation used in the formulation of the theorem.
Moreover, we will make use of the constants introduced in \Remref{Nemytskij-Anisotropic}~(a).
We fix $u = (u_1,\,\dots,\,u_m) \in U$, where $U \subseteq [X_1]^{s_1, \omega}_{p_1}(\bR^n,\,E_1) \times \ldots \times [X_m]^{s_m, \omega}_{p_m}(\bR^n,\,E_m)$
is defined as stated in the theorem with a constant $\rho > 0$ to be chosen below.
Note that $\phi(0) = 0$ implies $a_0 = 0$ in the power series expansion of $\phi$.

{\bfseries Step 1.}
First note that $\|u_j\|_{C_0(\bR^n,\,E_j)} < r$ implies $\Phi(u_1,\,\dots,\,u_m) \in C_0(\bR^n,\,E)$,
i.\,e.\ the embeddings $[X_j]^{s_j, \omega}_{p_j}(\bR^n,\,E_j) \hookrightarrow C_0(\bR^n,\,E_j)$
ensure $\Phi: U \longrightarrow C_0(\bR^n,\,E)$ to be well-defined, provided we choose $\rho < \min\,C^{-1}_j r$.
In this case we have
\begin{equation*}
	\Phi(u_1,\,\dots,\,u_m)(x) = \phi(u_1(x),\,\dots,\,u_m(x)) = \sum_{\alpha \in \bN^m_0} a_\alpha (u_1(x))^{\alpha_1} \bullet \dots \bullet (u_m(x))^{\alpha_m}
\end{equation*}
for all $x \in \bR^n$, i.\,e.\ $\Phi(u_1,\,\dots,\,u_m)$ is defined pointwise.

{\bfseries Step 2.}
On the other hand, by \Thmref{Multiplication-Anisotropic} we have
\begin{equation*}
	\begin{array}{l}
		\|u^\alpha\|_{X^{s, \omega}_p(\bR^n, E)} = \|u^{\alpha_1}_1 \bullet \dots \bullet u^{\alpha_m}_m\|_{X^{s, \omega}_p(\bR^n, E)} \\[0.5em]
		\qquad \leq M M^{\alpha_1}_1 \|u_1\|^{\alpha_1}_{[X_1]^{s_1, \omega}_{p_1}(\bR^n, E_1)} \cdot \ldots \cdot M^{\alpha_m}_m \|u_m\|^{\alpha_m}_{[X_m]^{s_m, \omega}_{p_m}(\bR^n, E_m)},
		\qquad \alpha \in \bN^m_0,\ \alpha \neq 0,
	\end{array}
\end{equation*}
where $u^{\alpha_1}_1 \bullet \dots \bullet u^{\alpha_m}_m$ has to be understood in the sense of \Remref{Multiplication-Anisotropic}~(l),
cf.~also the choice of $M > 0$ in \Remref{Nemytskij-Anisotropic}~(a).
Hence,
\begin{equation*}
	\Phi_\mu(u_1,\,\dots,\,u_m) := \!\!\!\!\!\!\!\! \sum_{\alpha \in \bN^m_0, |\alpha| \leq \mu} \!\!\!\!\!\!\!\! a_\alpha u^{\alpha_1}_1 \bullet \dots \bullet u^{\alpha_m}_m, \qquad \mu \in \bN
\end{equation*}
is a Cauchy sequence in $X^{s, \omega}_p(\bR^n, E)$, provided $\rho < \min\,M^{-1}_j r$.
Indeed,
\begin{equation*}
	\begin{array}{l}
		\|\Phi_\mu(u_1,\,\dots,\,u_m) - \Phi_{\mu^\prime}(u_1,\,\dots,\,u_m)\|_{X^{s, \omega}_p(\bR^n, E)}
			\leq \left\| {\displaystyle{\sum_{\substack{\alpha \in \bN^m_0 \\ \mu^\prime < |\alpha| \leq \mu}}}} a_\alpha u^{\alpha_1}_1 \bullet \dots \bullet u^{\alpha_m}_m \right\|_{X^{s, \omega}_p(\bR^n, E)} \\[3.5em]
			\qquad \qquad \leq M \!\!\!\! {\displaystyle{\sum_{\substack{\alpha \in \bN^m_0 \\ \mu^\prime < |\alpha| \leq \mu}}}} \!\!\!\! |a_\alpha|\,M^{\alpha_1}_1 \|u_1\|^{\alpha_1}_{[X_1]^{s_1, \omega}_{p_1}(\bR^n, E_1)} \cdot \ldots \cdot M^{\alpha_m}_m \|u_m\|^{\alpha_m}_{[X_m]^{s_m, \omega}_{p_m}(\bR^n, E_m)}
	\end{array}
\end{equation*}
for $\mu,\,\mu^\prime \in \bN$ with $\mu^\prime < \mu$ and the function
\begin{equation*}
	\begin{array}{c}
		\psi: (-M^{-1}_1 r,\,M^{-1}_1 r) \times \ldots \times (-M^{-1}_m r,\,M^{-1}_m r) \longrightarrow \bR \\[0.5em]
		\psi(\xi) := M {\displaystyle{\sum_{\alpha \in \bN^m_0}}} |a_\alpha|\,(M_1 \xi_1)^{\alpha_1} \cdot \ldots \cdot (M_m \xi_m)^{\alpha_m}, \qquad \xi \in \bR^m,\ |\xi_j| < M^{-1}_j r
	\end{array}
\end{equation*}
is a well-defined analytic function, since the power series expansion of
$\phi$ is absolutely convergent on $(-r,\,r)^m$.

{\bfseries Step 3.}
The embeddings $[X_j]^{s_j, \omega}_{p_j}(\bR^n,\,E_j) \hookrightarrow C_0(\bR^n,\,E_j)$ imply
\begin{equation*}
	\begin{array}{l}
		\|u^\alpha\|_{C_0(\bR^n, E)} = \|u^{\alpha_1}_1 \bullet \dots \bullet u^{\alpha_m}_m\|_{C_0(\bR^n, E)}
			\leq \|u_1\|^{\alpha_1}_{C_0(\bR^n, E_1)} \cdot \ldots \cdot \|u_m\|^{\alpha_m}_{C_0(\bR^n, E_m)} \\[0.5em]
			\qquad \leq C^{\alpha_1}_1 \|u_1\|^{\alpha_1}_{[X_1]^{s_1, \omega}_{p_1}(\bR^n, E_1)} \cdot \ldots \cdot C^{\alpha_m}_m \|u_m\|^{\alpha_m}_{[X_m]^{s_m, \omega}_{p_m}(\bR^n, E_m)},
			\qquad \alpha \in \bN^m_0,\ \alpha \neq 0.
	\end{array}
\end{equation*}
Hence, the same argumentation as in the second step implies $(\Phi_\mu(u_1,\,\dots,\,u_m))_{\mu \in \bN}$ to be a Cauchy sequence in $C_0(\bR^n,\,E)$,
provided $\rho < \min\,C^{-1}_j r$.

{\bfseries Step 4.}
Thus, if $\rho < \min\,\{\,C^{-1}_j,\,M^{-1}_j\,\}\,r$, then the limits
\begin{equation*}
	v := X^{s, \omega}_p(\bR^n,\,E) \! - \!\!\! \lim_{\mu \rightarrow \infty} \Phi_\mu(u_1,\,\dots,\,u_m) = C_0(\bR^n,\,E) \! - \!\!\! \lim_{\mu \rightarrow \infty} \Phi_\mu(u_1,\,\dots,\,u_m)
\end{equation*}
both exist and coincide due to the embedding $X^{s, \omega}_p(\bR^n,\,E) \hookrightarrow C_0(\bR^n,\,E)$.
Moreover, we have the pointwise representation
\begin{equation*}
	v(x) = \Phi(u)(x) = \Phi(u_1,\,\dots,\,u_m)(x) = \sum_{\alpha \in \bN^m_0} a_\alpha u^{\alpha_1}_1(x) \cdot \ldots \cdot u^{\alpha_m}_m(x), \qquad x \in \bR^n,
\end{equation*}
due to the first step and the fact that $v$ may be represented as a power series in $C_0(\bR^n,\,E)$.
Furthermore, $\Phi(u)$ may be represented as a power series in $X^{s, \omega}_p(\bR^n,\,E)$ and, therefore,
we infer that $\Phi: U \longrightarrow X^{s, \omega}_p(\bR^n,\,E)$ is analytic.

{\bfseries Step 5.}
Finally, if $\rho < \mbox{min}\,\{\,C^{-1}_j,\,M^{-1}_j\,\}\,r$, then the estimate
\begin{equation*}
	\begin{array}{l}
		\|\Phi(u_1,\,\dots,\,u_m)\|_{X^{s, \omega}_p(\bR^n, E)}
			= \left\| {\displaystyle{\sum_{\alpha \in \bN^m_0}}} u^{\alpha_1}_1 \bullet \dots \bullet u^{\alpha_m}_m \right\|_{X^{s, \omega}_p(\bR^n, E)} \\[2.5em]
			\qquad \leq M {\displaystyle{\sum_{\alpha \in \bN^m_0}}} |a_\alpha|\,M^{\alpha_1}_1 \|u_1\|^{\alpha_1}_{[X_1]^{s_1, \omega}_{p_1}(\bR^n, E_1)} \cdot \ldots \cdot M^{\alpha_m}_m \|u_m\|^{\alpha_m}_{[X_m]^{s_m,\omega}_{p_m}(\bR^n, E_m)} \\[2.0em]
			\qquad = \psi \Big( \|u_1\|_{[X_1]^{s_1, \omega}_{p_1}(\bR^n, E_1)},\,\dots,\,\|u_m\|_{[X_m]^{s_m, \omega}_{p_m}(\bR^n, E_m)} \Big)
	\end{array}
\end{equation*}
is valid.
Now, the function $\psi$ as defined in the second step is analytic and, hence,
Lipschitz continuous on the compact set $[-\rho,\,\rho]^m$ with $\psi(0) = 0$.
Thus, we infer
\begin{equation*}
	\|\Phi(u_1,\,\dots,\,u_m)\|_{X^{s, \omega}_p(\bR^n, E)}
		\leq L\,\mbox{max}\,\Big\{\,\|u_1\|_{[X_1]^{s_1, \omega}_{p_1}(\bR^n, E_1)},\,\dots,\,\|u_m\|_{[X_m]^{s_m, \omega}_{p_m}(\bR^n, E_m)}\,\Big\}
\end{equation*}
for some constant $L = L(M,\,\{\,a_\alpha\,:\,\alpha \in \bN^m_0,\ |\alpha| = 1\,\}) > 0$.
This completes the proof of \Thmref{Nemytskij-Anisotropic} and the additional assertions of \Remref{Nemytskij-Anisotropic}~(a). \qedbox

\section{Applications}\seclabel{Applications}
In this section we demonstrate how \Thmref{Multiplication-Anisotropic},
\Thmref{Multiplier-Anisotropic}, and \Thmref{Nemytskij-Anisotropic} may be employed to treat two prominent quasilinear parabolic problems.
These are the Stefan problem with Gibbs-Thomson correction, cf.~\Subsecref{Example-Stefan},
and the two-phase Navier-Stokes equations, cf.~\Subsecref{Example-NVS}.
Both problems constitute free boundary problems and in both cases we apply the above results
in a setting described in \Remref{Multiplication-Anisotropic}~(a), (b),
\Remref{Multiplication-Anisotropic-Simple}~(c), and \Remref{Nemytskij-Anisotropic}~(b).

\subsection{The Stefan Problem with Gibbs-Thomson Correction}\subseclabel{Example-Stefan}
As an example, we consider the Stefan problem with Gibbs-Thomson correction
given as
\begin{equation}
	\eqnlabel{stefan}
	\begin{array}{rcll}
		          \kappa \partial_t u - \mu \Delta u & = & 0,                                 & \quad t > 0,\    x \in \Omega(t), \\[0.5em]
		[u]_\Gamma = \sigma H_\Gamma, \quad V_\Gamma & = & [\![\mu \partial_\nu u]\!]_\Gamma, & \quad t \geq 0,\ x \in \Gamma(t), \\[0.5em]
		            \Gamma(0) = \Gamma_0, \quad u(0) & = & u_0,                               & \quad x \in \Omega(0).
	\end{array}
\end{equation}
This is a model for phase transitions in solid-liquid systems,
where the solid is assumed to occupy at a time $t \geq 0$ a region $\Omega_-(t) \subseteq \bR^n$,
while the liquid occupies a region $\Omega_+(t) \subseteq \bR^n$, which is separated from the solid phase by a sharp interface $\Gamma(t)$.
Here we assume $\bR^n = \Omega(t)\ \dot{\cup}\ \Gamma(t)$ with $\Omega(t) := \Omega_-(t)\ \dot{\cup}\ \Omega_+(t)$ for all $t \geq 0$.
The unknowns of the model are the temperature $u$ and the interface $\Gamma$,
whose evolution is determined by its normal velocity $V_\Gamma$.
For the classical Stefan problem one assumes that the temperatures $u_\pm := {u|}_{\Omega_\pm}$ coincide on the interface $\Gamma$
and are equal to the melting temperature, i.\,e.
\begin{equation*}
	[u_-]_\Gamma = [u_+]_\Gamma = 0, \qquad t \geq 0,\ x \in \Gamma(t),
\end{equation*}
where we denote by $[\,\cdot\,]_\Gamma$ the trace of a quantity defined in $\Omega_-$ and/or $\Omega_+$ on the interface $\Gamma$.
However, to account for effects like supercooling or dentritic growth of crystals
this condition has to be replaced by the so-called Gibbs-Thomson correction as above,
where $\sigma > 0$ denotes the constant surface tension and $H_\Gamma$ denotes the mean curvature of the interface $\Gamma$,
i.\,e.\ $H_\Gamma = - \mbox{div}_\Gamma\,\nu_\Gamma$, where we assume the unit normal field $\nu_\Gamma$ on $\Gamma$
to point from $\Omega_-$ into $\Omega_+$.
The evolution of the interface $\Gamma$ is determined by its normal velocity $V_\Gamma$,
which is subject to the above kinematic condition,
where $[\![\,\cdot\,]\!]_\Gamma$ denotes the jump of a quantity defined in $\Omega \setminus \Gamma$ across the interface $\Gamma$, i.\,e.
\begin{equation*}
	[\![\phi]\!]_\Gamma\,(t,\,x) := \lim_{\epsilon \rightarrow 0+} \big( \phi_+(t,\,x + \epsilon \nu_\Gamma(x)) - \phi_-(t,\,x - \epsilon \nu_\Gamma(x)) \big),
	\qquad t \geq 0,\ x \in \Gamma(t).
\end{equation*}
Note that the Gibbs-Thomson correction $[u]_\Gamma = \sigma H_\Gamma$ implicitly contains the continuity constraint $[\![u]\!]_\Gamma = 0$.
Finally, the heat capacities $\kappa = \kappa_\pm > 0$ and diffusion coefficients $\mu = \mu_\pm > 0$ are assumed to be constant, but may be different for the different phases,
and the initial configuration of the interface and the initial temperature distribution are given by $\Gamma_0$ and $u_0$, respectively.

Now, if we assume that the interface is sufficiently flat, such that it may be described as a graph
\begin{equation*}
	\Gamma(t) = \Big\{\,x + h(t,\,x) \nu_\Sigma(x)\,:\,x \in \Sigma\,\Big\}, \qquad t \geq 0,
\end{equation*}
where $\Sigma := \bR^n \setminus \dot{\bR}^n$ denotes the interface for the prototype geometry $\dot{\bR}^n := \bR^n_-\ \dot{\cup}\ \bR^n_+$
of two halfspaces $\bR^n_\pm := \left\{\,x \in \bR^n\,:\,\pm x_n > 0\,\right\}$
and $h(t,\,\cdot\,): \Sigma \longrightarrow \bR^n$ denotes a parametrization for $\Gamma(t)$,
then the Stefan problem \eqnref{stefan} may be studied with the aid of a transformation of the time dependent geometry
$\Omega(\,\cdot\,) \times \Gamma(\,\cdot\,)$ to the fixed geometry $\dot{\bR}^n \times \Sigma$.
Indeed, based on the parametrizations $h$ we may set
\begin{equation*}
	\bar{u}(t,\,x) := u(t,\,x + h(t,\,\Pi_\Sigma x) \nu_\Sigma(\Pi_\Sigma x)), \qquad t \geq 0,\ x \in \dot{\bR}^n,
\end{equation*}
where we denote by $\Pi_\Sigma: \bR^n \longrightarrow \Sigma$ the orthogonal projection onto $\Sigma$.
This way, we arrive at the quasilinear problem
\begin{equation}
	\eqnlabel{stefan-qlin}
	\begin{array}{rclll}
		                                                                  \kappa \partial_t u - \mu \Delta u & = & F(u,\,h)   & \quad \mbox{in} & J \times \dot{\bR}^n, \\[0.5em]
		[u]_\Sigma - \sigma \Delta_\Sigma h = G_u(h), \quad \partial_t h - [\![\mu \partial_\nu u]\!]_\Sigma & = & G_h(u,\,h) & \quad \mbox{on} & J \times \Sigma,      \\[0.5em]
		                                               u(0) = u_0 \quad \mbox{in}\ \ \dot{\bR}^n, \quad h(0) & = & h_0        & \quad \mbox{on} & \Sigma
	\end{array}
\end{equation}
on a given time interval $J = (0,\,a)$ with $0 < a \leq \infty$, where we dropped the bars again for convenience.
The non-linear right-hand sides are given as
\begin{equation*}
	\begin{array}{rcl}
		  F(u,\,h) & = & (\kappa \partial_t h - \mu \Delta_\Sigma h) \partial_n u - 2 \mu (\nabla_\Sigma h \cdot \nabla_\Sigma) \partial_n u + \mu |\nabla_\Sigma h|^2 \partial^2_n u \\[1.0em]
		    G_u(h) & = & - \sigma {\displaystyle{\frac{|\nabla_\Sigma h|^2}{\sqrt{1 + |\nabla_\Sigma h|^2}\,\big(1 + \sqrt{1 + |\nabla_\Sigma h|^2}\big)}}} \Delta_\Sigma h - \sigma {\displaystyle{\frac{\nabla_\Sigma h \otimes \nabla_\Sigma h}{\sqrt{1 + |\nabla_\Sigma h|^2}^{\,3}}}} : \nabla^2_\Sigma h \\[2.0em]
		G_h(u,\,h) & = & |\nabla_\Sigma h|^2 [\![\mu \partial_\nu u]\!]_\Sigma - \nabla_\Sigma h \cdot [\![\mu \nabla_\Sigma u]\!]_\Sigma.
	\end{array}
\end{equation*}
Thus, one may obtain a maximal regular solution
\begin{equation*}
	\begin{array}{rclcl}
		u & \in & \bX_u(a) & := & H^1_p(J,\,L_p(\dot{\bR}^n)) \cap L_p(J,\,H^2_p(\dot{\bR}^n)) = H^{2, (2, 1)}_p(J\times\dot{\bR}^n), \\[0.5em]
		h & \in & \bX_h(a) & := & W^{3/2 - 1/2p}_p(J,\,L_p(\Sigma)) \cap W^{1 - 1/2p}_p(J,\,H^2_p(\Sigma)) \cap L_p(J,\,W^{4 - 1/p}_p(\Sigma))
	\end{array}
\end{equation*}
to \eqnref{stefan-qlin} via the fixed point equation
\begin{equation*}
	L(u,\,h) = N(u,\,h), \quad (u(0),\,h(0)) = (u_0,\,h_0), \qquad
	(u,\,h) \in \bX(a) := \bX_u(a) \times \bX_h(a).
\end{equation*}
Here the bounded linear operator $L: \bX(a) \longrightarrow \bY(a)$ is given by the left-hand side (without initial conditions) of the linear problem
\begin{equation}
	\eqnlabel{Stefan-Linear}
	\begin{array}{rclll}
		                                                               \kappa \partial_t u - \mu \Delta u & = & f   & \quad \mbox{in} & J \times \dot{\bR}^n, \\[0.5em]
		[u]_\Sigma - \sigma \Delta_\Sigma h = g_u, \quad \partial_t h - [\![\mu \partial_\nu u]\!]_\Sigma & = & g_h & \quad \mbox{on} & J \times \Sigma,      \\[0.5em]
		                                            u(0) = u_0 \quad \mbox{in}\ \ \dot{\bR}^n, \quad h(0) & = & h_0 & \quad \mbox{on} & \Sigma
	\end{array}
\end{equation}
and the non-linear operator $N: \bX(a) \longrightarrow \bY(a)$ is given as
\begin{equation*}
	N(u,\,h) = (F(u,\,h),\,G_u(h),\,G_h(u,\,h)), \qquad (u,\,h) \in \bX(a).
\end{equation*}
The data has to be chosen according to the regularity class of the solution as
\begin{equation*}
	\begin{array}{rclrl}
		  f & \in & \bY_f(a)      & := & L_p(J \times \bR^n), \\[0.5em]
		g_u & \in & \bY_{g, u}(a) & := & W^{1 - 1/2p}_p(J,\,L_p(\Sigma)) \cap L_p(J,\,W^{2 - 1/p}_p(\Sigma)) \\[0.5em]
		    &     &               &  = & W^{2 - 1/p, (2,1)}_p(J \times \Sigma), \\[0.5em]
		g_h & \in & \bY_{g, h}(a) & := & W^{1/2 - 1/2p}_p(J,\,L_p(\Sigma)) \cap L_p(J,\,W^{1 - 1/p}_p(\Sigma)) \\[0.5em]
		    &     &               &  = & W^{1 - 1/p, (2,1)}_p(J \times \Sigma),
	\end{array}
\end{equation*}
and the initial conditions have to satisfy
\begin{equation*}
	u_0 \in \bZ_u := W^{2 - 2/p}_p(\dot{\bR}^n), \qquad \qquad
	h_0 \in \bZ_h := \left\{ \begin{array}{ll} W^{6 - 6/p}_p(\Sigma), & \quad 1 < p < 3/2, \\[0.5em] W^{4 - 3/p}_p(\Sigma), & \quad 3/2 < p < \infty \end{array} \right.
\end{equation*}
as well as the compatibility conditions
\begin{equation}
	\eqnlabel{Stefan-Compatibility}
	[u_0]_\Sigma - \sigma \Delta_\Sigma h_0 = g_u(0), \ \mbox{if} \ p > 3/2, \quad [\![\mu \partial_\nu u_0]\!]_\Sigma + g_h(0) \in W^{2 - 6/p}_p(\Sigma), \ \mbox{if} \ p > 3.
\end{equation}
Furthermore, we set $\bY(a) := \bY_f(a) \times \bY_{g, u}(a) \times \bY_{g, h}(a)$ and $\bZ := \bZ_u \times \bZ_h$.

In order to solve the above fixed point problem, first of all the linear operator $L$ has to be invertible, which is given by the following result;
see also \cite{Denk-Kaip:Mixed-Order-Systems, Denk-Saal-Seiler:Newton-Polygon} for a systematic approach to maximal regularity for general linear mixed order systems
based on the Newton-Polygon method.
\begin{proposition}[{\cite[Theorem~6.1]{Escher-Pruess-Simonett:Stefan-Analytic}}]
	\proplabel{Stefan-Linear}
	Let $a > 0$ and $1 < p < \infty$ with $p \neq \frac{3}{2},\,3$.
	Let $\kappa = \kappa_\pm > 0$, $\mu = \mu_\pm > 0$ and $\sigma > 0$.
	Then there exists a unique solution $(u,\,h) \in \bX(a)$ to the linear problem \eqnref{Stefan-Linear}, if and only if the data satisfies
	\begin{equation*}
		(f,\,g_u,\,g_h) \in \bY(a), \qquad (u_0,\,h_0) \in \bZ
	\end{equation*}
	and the compatibility conditions \eqnref{Stefan-Compatibility}.
	This solution then satisfies
	\begin{equation*}
		\|(u,\,h)\|_{\bX(a)} \leq C(a,\,p) \Big( \|(f,\,g_u,\,g_h)\|_{\bY(a)} + \|(u_0,\,h_0)\|_{\bZ} \Big)
	\end{equation*}
	with some constant $C(a,\,p) > 0$, which is independent of the data. \qedbox
\end{proposition}

In a second step we have to show that the non-linear operator $N$ enjoys suitable mapping properties to solve the above fixed-point problem.
In order to economize the notation, we denote by $\bB_{h, r}(a)$ the open ball of radius $r > 0$ in $\bX_h(a)$.
\begin{proposition}
	\proplabel{Stefan-Non-Linear-Optimal}
	Let $a > 0$ and $\frac{n + 2}{2} \leq p < \infty$.
	Let $\kappa = \kappa_\pm > 0$, $\mu = \mu_\pm > 0$ and $\sigma > 0$.
	Then $N \in C^\omega(\bX_r(a),\,\bY(a))$ for some $r > 0$,
	where $\bX_r(a) := \bX_u(a) \times \bB_{h, r}(a)$.
	Moreover, $N(0) = 0$ and the Fr{\'e}chet derivative of $N$ satisfies $DN(0) = 0$.
\end{proposition}
\begin{proof}
First, we have the following embeddings,
which are consequences of corresponding {\itshape mixed derivative
theorems}; see e.g.~\cite{Denk-Saal-Seiler:Newton-Polygon}.
We have
\begin{equation*}
	W^{3/2 - 1/2p}_p(J,\,L_p(\Sigma)) \cap W^{1 - 1/2p}_p(J,\,H^2_p(\Sigma)) \hookrightarrow H^1_p(J,\,W^{2 - 2/p}_p(\Sigma)),
\end{equation*}
which implies
\begin{subequations}
\begin{equation}
	\label{tderh}
	\begin{array}{l}
		\partial_t h \in W^{1/2 - 1/2p}_p(J,\,L_p(\Sigma)) \cap L_p(J,\,W^{2 - 2/p}_p(\Sigma)) \\[0.5em]
			\qquad \qquad = W^{2 - 2/p, (4,1)}_p(J \times \Sigma) \hookrightarrow W^{1 - 1/p, (2,1)}_p(J \times \Sigma)
	\end{array}
\end{equation}
for $h \in \bX_h(a)$. Next,
\begin{equation*}
	W^{3/2 - 1/2p}_p(J,\,L_p(\Sigma)) \cap W^{1 - 1/2p}_p(J,\,H^2_p(\Sigma)) \hookrightarrow W^{5/4 - 1/2p}_p(J,\,H^1_p(\Sigma)),
\end{equation*}
yields
\begin{equation}
	\label{derh}
	\begin{array}{l}
		\partial_j h \in W^{5/4 - 1/2p}_p(J,\,L_p(\Sigma)) \cap L_p(J,\,W^{3 - 1/p}_p(\Sigma)), \\[0.5em]
			\qquad \qquad \hookrightarrow W^{5/4 - 1/2p}_p(J,\,L_p(\Sigma)) \cap L_p(J,\,W^{5/2 - 1/p}_p(\Sigma)) = W^{5/2 - 1/p, (2,1)}_p(J \times \Sigma)
	\end{array}
\end{equation}
for $h \in \bX_h(a)$ and $j = 1,\,\dots,\,n - 1$.
Finally,
\begin{equation}
	\label{2derh}
	\partial_j \partial_k h \in W^{1 - 1/2p}_p(J,\,L_p(\Sigma)) \cap L_p(J,\,W^{2 - 1/p}_p(\Sigma)) = W^{2 - 1/p, (2,1)}_p(J \times \Sigma)
\end{equation}
\end{subequations}
for $h \in \bX_h(a)$ and $j,\,k = 1,\,\dots,\,n - 1$.
Similarly, we obtain
\begin{subequations}
\begin{align}
	           \partial_j u & \in H^{1,(2,1)}_p(J \times \dot{\bR}^n) \hookrightarrow H^{1, (2,1)}_p(J \times \Sigma,\,L_p(\dot{\bR})), \label{deru} \\
	\partial_j \partial_k u & \in L_p(J \times \dot{\bR}^n) = L_p(J \times \Sigma,\,L_p(\dot{\bR})), \label{2deru}
\end{align}
for $j,\,k = 1,\,\dots,\,n$ as well as 
\begin{equation}
	\label{tracederu}
	[\partial_j u]_\Sigma \in W^{1/2 - 1/2p}_p(J,\,L_p(\Sigma)) \cap L_p(J,\,W^{1 - 1/p}_p(\Sigma)) = W^{1 - 1/p, (2,1)}_p(J \times \Sigma).
\end{equation}
\end{subequations}

{\itshape Mapping properties of $F$.}
The non-linearity $F$ is a sum of simple multilinear operators given as products of components of the solution and its derivatives.
In order to obtain the desired mapping properties of $F$ it is hence sufficient to establish corresponding estimates.
According to (\ref{tderh}), (\ref{2derh}), and (\ref{deru}) we can handle the term
\begin{equation*}
	(\kappa \partial_t h - \mu \Delta_\Sigma h) \partial_n u
\end{equation*}
provided the vector-valued embedding
\begin{equation*}
	\underbrace{W^{1 - 1/p, (2,1)}_p(J \times \Sigma)}_{\textrm{ind}_1 = \frac{1}{2} - \frac{n + 2}{2p}} \cdot \underbrace{H^{1, (2,1)}_p(J \times \Sigma,\,L^p(\dot{\bR}))}_{\textrm{ind}_2 = \frac{1}{2} - \frac{n + 1}{2p}}
		\hookrightarrow \underbrace{H^{0, (2,1)}_p(J \times \Sigma,\,L^p(\dot{\bR}))}_{\textrm{ind} = - \frac{n + 1}{2p}}
\end{equation*}
is valid.
Based on Theorem~\ref{thm:Multiplication-Anisotropic} we infer that it is valid, if $\max\,\{\,\mbox{ind}_1,\,\mbox{ind}_2\,\} \geq 0$.
However, for small values of $p$ both indices on the left-hand-side become negative and Theorem~\ref{thm:Multiplication-Anisotropic} requires
the stronger condition that $\mbox{ind}_1 + \mbox{ind}_2 \geq \mbox{ind}$, which is easily seen to be equivalent to
\begin{equation}
	\label{cond1p}
	p \geq \frac{n + 2}{2}.
\end{equation}
In order to estimate the terms
\begin{equation*}
	2 \mu (\nabla_\Sigma h \cdot \nabla_\Sigma) \partial_n u,\ \mu |\nabla_\Sigma h|^2 \partial^2_n u
\end{equation*}
we employ (\ref{derh}), (\ref{2deru}) and the vector-valued embeddings
\begin{equation*}
	\big[ \underbrace{W^{5/2 - 1/p, (2,1)}_p(J \times \Sigma)}_{\textrm{ind}_1 = \frac{5}{4} - \frac{n + 2}{2p}} \big]^m \cdot \underbrace{H^{0, (2,1)}_p(J \times \Sigma,\,L^p(\dot{\bR}))}_{\textrm{ind}_2 = - \frac{n + 1}{2p}}
		\hookrightarrow \underbrace{H^{0, (2,1)}_p(J \times \Sigma,\,L^p(\dot{\bR}))}_{\textrm{ind} = - \frac{n + 1}{2p}}
\end{equation*}
for $m = 1,\,2$.
Now, \Thmref{Multiplier-Anisotropic} implies these embeddings to be valid (for all $m \in \bN$),
provided we have $\mbox{ind}_1 > 0$ or, equivalently,
\begin{equation}
	\label{cond2p}
	p > \frac{2(n + 2)}{5},
\end{equation}
which is satisfied, if (\ref{cond1p}) holds.
In summary, these considerations show that $F$ has the desired mapping properties,
provided that the constraint (\ref{cond1p}) is satisfied.

{\itshape Mapping properties of $G_u$.}
The structure of the non-linearity $G_u$ is a little bit more complex.
Here, we show that there exists an $r > 0$ such that
\begin{equation*}
	G_u: \bB_{h, r}(a) \longrightarrow W^{2 - 1/p, (2, 1)}_p(J \times \Sigma)
\end{equation*}
is analytic.
In order to do so, we employ the functions
\begin{equation*}
	\phi: \bR^{n - 1} \longrightarrow \bR, \qquad
	\phi(\xi) := \frac{|\xi|^2}{\sqrt{1 + |\xi|^2} \big( 1 + \sqrt{1 + |\xi|^2} \big)}, \quad \xi \in \bR^{n - 1},
\end{equation*}
and
\begin{equation*}
	\psi_{jk}: \bR^{n - 1} \longrightarrow \bR, \qquad
	\psi_{jk}(\xi) := \frac{\xi_j \xi_k}{{\sqrt{1 + |\xi|^2}}^3}, \quad \xi \in \bR^{n - 1}, \quad j,\,k = 1,\,\dots,\,n - 1,
\end{equation*}
which are obviously analytic in a neighbourhood of the origin with $\phi(0) = \psi_{jk}(0) = 0$.
Thus, if
\begin{equation*}
	\frac{5}{4} - \frac{n + 2}{2p} = \mbox{ind} \Big( W^{5/2 - 1/p, (2, 1)}_p(J \times \Sigma) \Big) > 0,
\end{equation*}
which is equivalent to (\ref{cond2p}),
then (\ref{derh}) and Theorem~\ref{thm:Nemytskij-Anisotropic} yield the analyticity of the mappings
\begin{equation}
	\label{phipsi}
	h \mapsto \phi(\nabla_\Sigma h),\ h \mapsto \psi_{j k}(\nabla_\Sigma h): \bB_{h, r}(a) \longrightarrow W^{5/2 - 1/p, (2, 1)}_p(J \times \Sigma)
\end{equation}
for $j,\,k = 1,\,\dots,\,n - 1$ and some $r > 0$.
Since
\begin{equation*}
	G_u(h) = - \sigma \phi(\nabla_\Sigma h) \Delta_\Sigma h - \sigma \sum^{n - 1}_{j, k = 1} \psi_{j k}(\nabla_\Sigma h) \partial_j \partial_k h, \qquad h \in \bB_{h, r}(a),
\end{equation*}
in order to obtain the desired mapping properties of $G_u$
we only need to use (\ref{2derh}), the mapping properties (\ref{phipsi}), and the embedding
\begin{equation*}
	\underbrace{W^{5/2 - 1/p, (2,1)}_p(J \times \Sigma)}_{\textrm{ind}_1 = \frac{5}{4} - \frac{n + 2}{2p}} \cdot \underbrace{W^{2 - 1/p, (2,1)}_p(J \times \Sigma)}_{\textrm{ind}_2 = 1 - \frac{n + 2}{2p}}
		\hookrightarrow \underbrace{W^{2 - 1/p, (2,1)}_p(J \times \Sigma)}_{\textrm{ind} = 1 - \frac{n + 2}{2p}},
\end{equation*}
which is provided by \Thmref{Multiplier-Anisotropic},
if $\mbox{ind}_1 > 0$.
Hence, the non-linearity $G_u$ has the desired mapping properties, provided that $\mbox{ind}_1 > 0$
or, equivalently, if (\ref{cond2p}) is satisfied, which is true, if (\ref{cond1p}) is satisfied.

{\itshape Mapping properties of $G_h$.}
The non-linearity $G_h$ is again a sum of simple multilinear operators given as products of components of the solution and its derivatives.
In order to obtain the desired mapping properties of $G_h$ it is hence sufficient to establish corresponding estimates.
According to (\ref{derh}) and (\ref{tracederu}) we may estimate the terms
\begin{equation*}
	\nabla_\Sigma h \cdot [\![\mu \nabla_\Sigma u]\!]_\Sigma,\ |\nabla_\Sigma h|^2 [\![\mu \partial_\nu u]\!]_\Sigma
\end{equation*}
provided we have the embeddings
\begin{equation*}
	\big[ \underbrace{W^{5/2 - 1/p, (2,1)}_p(J \times \Sigma)}_{\textrm{ind}_1 = \frac{5}{4} - \frac{n + 2}{2p}} \big]^m \cdot \underbrace{W^{1 - 1/p, (2,1)}_p(J \times \Sigma))}_{\textrm{ind}_2 = \frac{1}{2} - \frac{n + 2}{2p}}
		\hookrightarrow \underbrace{W^{1 - 1/p, (2,1)}_p(J \times \Sigma))}_{\textrm{ind} = \frac{1}{2} - \frac{n + 2}{2p}}
\end{equation*}
for $m = 1,\,2$.
Based on \Thmref{Multiplier-Anisotropic} we infer that these are valid (for all $m \in \bN$),
if $\mbox{ind}_1 > 0$.
Hence, the non-linearity $G_h$ has the desired mapping properties, provided that $\mbox{ind}_1 > 0$
or, equivalently, if (\ref{cond2p}) is satisfied, which is true, if (\ref{cond1p}) is satisfied.

The facts that $N(0) = 0$ and $DN(0) = 0$ follow by a straight forward calculation.	
Collecting the outcome of the single steps, the assertion is proved.
\end{proof}

Based on Propositions~\ref{prop:Stefan-Linear}, \ref{prop:Stefan-Non-Linear-Optimal}
and the contraction mapping principle in the same way as in 
\cite{Escher-Pruess-Simonett:Stefan-Analytic} we obtain the following 
improvement of \cite[Theorem~7.5]{Escher-Pruess-Simonett:Stefan-Analytic}.
\begin{theorem}
	\thmlabel{Stefan-Full}
	Let $a > 0$ and $\frac{n + 2}{2} \leq p < \infty$ with $p \neq \frac{3}{2},\,3$,
	let $\kappa = \kappa_\pm > 0$, $\mu = \mu_\pm > 0$, and $\sigma > 0$.
	Then there exists $\delta > 0$ such that the quasilinear problem \eqnref{stefan-qlin} admits a unique solution $(u,\,h) \in \bX(a)$,
	provided the initial conditions $(u_0,\,h_0) \in \bZ$ satisfy
	\begin{equation*}
		\begin{array}{c}
			\|(u_0,\,h_0)\|_{\bZ} < \delta, \\[1.0em]
			[u_0]_\Sigma - \sigma \Delta_\Sigma h_0 = G_u(h_0), \qquad \|[\![\mu \partial_\nu u_0]\!]_\Sigma + G_h(u_0,\,h_0)\|_{W^{2 - 6/p}_p(\Sigma)} < \delta.
		\end{array}
	\end{equation*}
	The solutions depend continuously on the data.
\end{theorem}

\begin{remark}
\remlabel{Example-Stefan}
(a) Theorem~\ref{thm:Stefan-Full} and its proof given here improves 
\cite{Escher-Pruess-Simonett:Stefan-Analytic} into two directions.
Combined with the Newton-Polygon approach to maximal regularity
for linearized mixed order systems developed in \cite{Denk-Saal-Seiler:Newton-Polygon},
it provides a rather systematic way to handle quasilinear problems as 
the Stefan problem \eqnref{stefan-qlin}.
Secondly, by the sharp embedding result Theorem~\ref{thm:Multiplication-Anisotropic}
it includes a considerable improvement of the range for admissible $p$ from $p > n + 2$
as given in \cite[Theorem~7.5]{Escher-Pruess-Simonett:Stefan-Analytic}
to $p \geq (n + 2) / 2$.

(b) It is interesting to mention that by our approach we can also include
the scaling invariant and sharp case $p = (n + 2) / 2$.
Moreover, the smallness condition in \Propref{Stefan-Non-Linear-Optimal},
i.\,e.\ analyticity for $h \in \bB_{h, r}(a)$, stem from the application of \Thmref{Nemytskij-Anisotropic}
for the analysis of the non-linearity $G_u$.
However, the Nemytskij operators within $G_u$ are only applied to $\nabla_\Sigma h \in \partial \bX_h(a)$,
where $\partial \bX_h(a) := W^{5/2 - 1/p, (2,1)}_p(J \times \Sigma)$.
Thus, the smallness condition on $h$ in $\bX_h(a)$ that is required in \Propref{Stefan-Non-Linear-Optimal}
can be substituted by a smallness condition on $\nabla_\Sigma h$ in $\partial \bX_h(a)$.

(c) We note that the statements of 
Propositions~\ref{prop:Stefan-Linear}, \ref{prop:Stefan-Non-Linear-Optimal} and Theorem~\ref{thm:Stefan-Full} slightly differ
from those of Theorem~6.1, Lemma~7.4 and Theorem~7.5 in \cite{Escher-Pruess-Simonett:Stefan-Analytic} in the following way:
\begin{itemize}
	\item First of all, in \cite{Escher-Pruess-Simonett:Stefan-Analytic} the whole problem is considered in $\bR^{n^\prime + 1}$ and not in $\bR^n$.
		For this reason, \cite[Lemma~7.4 \& Theorem~7.5]{Escher-Pruess-Simonett:Stefan-Analytic} require the condition $p > n^\prime + 3$,
		which may cause confusion unless one notes that $n = n^\prime + 1$.
	\item Secondly, the considerations there are restricted to the case $\kappa = \kappa_\pm = 1$, $\mu = \mu_\pm = 1$ and $\sigma = 1$.
		However, it may be readily checked that the proof of \cite[Theorem~6.1]{Escher-Pruess-Simonett:Stefan-Analytic} carries over to the
		case $\kappa = \kappa_\pm > 0$, $\mu = \mu_\pm > 0$ and $\sigma > 0$, where the heat capacities and diffusion coefficients may be different
		for the different phases.
		The proofs of \cite[Lemma~7.4 \& Theorem~7.5]{Escher-Pruess-Simonett:Stefan-Analytic} are not at all affected by this modification.
	\item For the general theory of parabolic free boundary problems including the Stefan problem
		we refer to the pertinent monograph \cite{Pruess-Simonett:Moving-Interfaces}.
\end{itemize}
\end{remark}

\subsection{The Two-Phase Navier-Stokes Equations with Surface Tension}\subseclabel{Example-NVS}
As a second example, we consider the two-phase Navier-Stokes equations with surface tension given as
\begin{equation}
	\eqnlabel{two-phase-nvs}
	\begin{array}{rcll}
		\rho (\partial_t u + (u \cdot \nabla) u)) - \mbox{div}\,S(u,\,q) & = & 0,                           & \quad t > 0,\    x \in \Omega(t), \\[0.5em]
		                                                   \mbox{div}\,u & = & 0,                           & \quad t > 0,\    x \in \Omega(t), \\[0.5em]
		                            [\![u]\!]_\Gamma = 0, \quad V_\Gamma & = & [u]_\Gamma \cdot \nu_\Gamma, & \quad t \geq 0,\ x \in \Gamma(t), \\[0.5em]
		                           - [\![S(u,\,q)]\!]_\Gamma\,\nu_\Gamma & = & \sigma H_\Gamma \nu_\Gamma,  & \quad t \geq 0,\ x \in \Gamma(t), \\[0.5em]
		                                \Gamma(0) = \Gamma_0, \quad u(0) & = & u_0,                         & \quad x \in \Omega(0).
	\end{array}
\end{equation}
This is a model for the flow of two immiscible fluids.
The first fluid is assumed to occupy at a time $t \geq 0$ a region $\Omega_-(t) \subseteq \bR^n$,
while the second fluid occupies a region $\Omega_+(t) \subseteq \bR^n$ and both regions are assumed to be separated by a sharp interface $\Gamma(t)$.
Here we assume $\bR^n = \Omega(t)\ \dot{\cup}\ \Gamma(t)$ with $\Omega(t) := \Omega_-(t)\ \dot{\cup}\ \Omega_+(t)$ for all $t \geq 0$.
Moreover, we assume both fluids to be incompressible, and we assume the stresses to be given as
\begin{equation*}
	S(u,\,q) = 2 \mu D(u) - q, \qquad D(u) = {\textstyle \frac{1}{2}} (\nabla u + \nabla u^{\sfT}), \qquad t \geq 0,\ x \in \Omega(t).
\end{equation*}
The unknowns of the model are the velocity $u$, the pressure $q$ and the interface $\Gamma$,
whose evolution is determined by its normal velocity $V_\Gamma$.
As in \Subsecref{Example-Stefan} we denote by $[\,\cdot\,]_\Gamma$ the trace of a quantity defined in $\Omega_-$ and/or $\Omega_+$ on the interface $\Gamma$,
by $[\![\,\cdot\,]\!]_\Gamma$ the jump of a quantity defined in $\Omega \setminus \Gamma$ across the interface $\Gamma$,
by $\nu_\Gamma$ the unit normal field on $\Gamma$, which is assumed to point from $\Omega_-$ into $\Omega_+$,
and by $H_\Gamma$ the mean curvature of the interface $\Gamma$,
i.\,e.\ $H_\Gamma = - \mbox{div}_\Gamma\,\nu_\Gamma$.
Finally, the densities $\rho = \rho_\pm > 0$ and viscosities $\mu = \mu_\pm > 0$ are assumed to be constant, but may be different for the different phases.
The surface tension $\sigma > 0$ is also assumed to be constant
and the initial configuration of the interface resp.\ the initial velocity are given by $\Gamma_0$ resp.\ $u_0$.

As in \Subsecref{Example-Stefan} we assume that the interface is sufficiently flat, such that it may be described as a graph
\begin{equation*}
	\Gamma(t) = \Big\{\,x + h(t,\,x) \nu_\Sigma(x)\,:\,x \in \Sigma\,\Big\}, \quad t \geq 0,
\end{equation*}
where $\Sigma = \bR^n \setminus \dot{\bR}^n$ denotes the interface for the prototype geometry
$\dot{\bR}^n = \bR^n_-\ \dot{\cup}\ \bR^n_+$ of two halfspaces $\bR^n_\pm$
and $h(t,\,\cdot\,): \Sigma \longrightarrow \bR^n$ denotes a parametrization for $\Gamma(t)$.
Then the two-phase Navier-Stokes equations \eqnref{two-phase-nvs} may be studied with the aid of a transformation of the time dependent geometry
$\Omega(\,\cdot\,) \times \Gamma(\,\cdot\,)$ to the fixed geometry $\dot{\bR}^n \times \Sigma$.
Indeed, based on the parametrizations $h$ we may set
\begin{equation*}
	\begin{array}{rcll}
		\bar{v}(t,\,x) & := & \left[ \begin{array}{c} u_1(t,\,x + h(t,\,\Pi_\Sigma x) \nu_\Sigma(\Pi_\Sigma x)) \\[0.5em] \vdots \\[0.5em] u_{n - 1}(t,\,x + h(t,\,\Pi_\Sigma x) \nu_\Sigma(\Pi_\Sigma x)) \end{array} \right], & \quad t \geq 0,\ x \in \dot{\bR}^n, \\[3.0em]
		\bar{w}(t,\,x) & := & u_n(t,\,x + h(t,\,\Pi_\Sigma x) \nu_\Sigma(\Pi_\Sigma x)), & \quad t \geq 0,\ x \in \dot{\bR}^n, \\[0.5em]
		\bar{q}(t,\,x) & := &   q(t,\,x + h(t,\,\Pi_\Sigma x) \nu_\Sigma(\Pi_\Sigma x)), & \quad t \geq 0,\ x \in \dot{\bR}^n,
	\end{array}
\end{equation*}
where we again denote by $\Pi_\Sigma: \bR^n \longrightarrow \Sigma$ the orthogonal projection onto $\Sigma$.
This way, we arrive at the quasilinear problem
\begin{equation}
	\eqnlabel{two-phase-nvs-qlin}
	\begin{array}{rclll}
		                               \rho \partial_t v - \mu \Delta v + \nabla^\prime q & = & F_v(v,\,w,\,q,\,h)                & \quad \mbox{in} & J \times \dot{\bR}^n, \\[0.5em]
		                                  \rho \partial_t w - \mu \Delta w + \partial_n q & = & F_w(v,\,w,\,h)                    & \quad \mbox{in} & J \times \dot{\bR}^n, \\[0.5em]
		                                               \mbox{div}^\prime v + \partial_n w & = & G_q(v,\,h)                        & \quad \mbox{in} & J \times \dot{\bR}^n, \\[0.5em]
		[\![v]\!]_\Sigma = 0, \quad [\![w]\!]_\Sigma = 0, \quad \partial_t h - [w]_\Sigma & = & F_h(v,\,h)                        & \quad \mbox{on} & J \times \Sigma,      \\[0.5em]
		           - [\![\mu \partial_n v]\!]_\Sigma - [\![\mu \nabla^\prime w]\!]_\Sigma & = & G_v(v,\,w,\,[\![q]\!]_\Sigma,\,h) & \quad \mbox{on} & J \times \Sigma,      \\[0.5em]
		  - [\![2 \mu \partial_n w]\!]_\Sigma + [\![q]\!]_\Sigma - \sigma \Delta_\Sigma h & = & G_w(v,\,w,\,h)                    & \quad \mbox{on} & J \times \Sigma,      \\[0.5em]
		          v(0) = v_0, \quad w(0) = w_0 \quad \mbox{in}\ \ \dot{\bR}^n, \quad h(0) & = & h_0                               & \quad \mbox{on} & \Sigma
	\end{array}
\end{equation}
on a given time interval $J = (0,\,a)$ with $0 < a \leq \infty$, where we dropped the bars again for convenience.
Moreover, we abbreviated $\nabla^\prime := (\partial_1,\,\dots,\,\partial_{n - 1})^{\sfT}$ and $\mbox{div}^\prime := \nabla^\prime\,\cdot\,$.
The non-linear right-hand sides are given as
\begin{equation*}
	\begin{array}{rcl}
		                 F_v(v,\,w,\,q,\,h) & = & (\rho \partial_t h - \mu \Delta_\Sigma h) \partial_n v - 2 \mu (\nabla_\Sigma h \cdot \nabla_\Sigma) \partial_n v + \mu |\nabla_\Sigma h|^2 \partial^2_n v \\[0.5em]
			                                  &   & \quad - \rho (v \cdot \nabla^\prime) v + \rho (v \cdot \nabla_\Sigma h) \partial_n v - \rho w\,\partial_n v + \nabla_\Sigma h\,\partial_n q,               \\[0.5em]
		                     F_w(v,\,w,\,h) & = & (\rho \partial_t h - \mu \Delta_\Sigma h) \partial_n w - 2 \mu (\nabla_\Sigma h \cdot \nabla_\Sigma) \partial_n w + \mu |\nabla_\Sigma h|^2 \partial^2_n w \\[0.5em]
			                                  &   & \quad -\ \rho (v \cdot \nabla^\prime) w + \rho (v \cdot \nabla_\Sigma h) \partial_n w - \rho w\,\partial_n w,                                              \\[0.5em]
		                         G_q(v,\,h) & = & \nabla_\Sigma h \cdot \partial_n v,                                                                                                                        \\[0.5em]
		                         F_h(v,\,h) & = & - \nabla_\Sigma h \cdot [v]_\Sigma,                                                                                                                        \\[0.5em]
		  G_v(v,\,w,\,[\![q]\!]_\Sigma,\,h) & = & - [\![\mu(\nabla^\prime v + \nabla^\prime v^{\sfT})]\!]_\Sigma\,\nabla_\Sigma h + (\nabla_\Sigma h \cdot [\![\mu \partial_n v]\!]_\Sigma) \nabla_\Sigma h  \\[0.5em]
			                                  &   & \quad +\ |\nabla_\Sigma h|^2 [\![\mu \partial_n v]\!]_\Sigma - [\![\mu \partial_n w]\!]_\Sigma\,\nabla_\Sigma h + [\![q]\!]_\Sigma\,\nabla_\Sigma h        \\[0.5em]
			                                  &   & \quad -\ (\sigma \Delta_\Sigma h + G_\sigma(h)) \nabla_\Sigma h,                                                                                           \\[0.5em]
		                     G_w(v,\,w,\,h) & = & - \nabla_\Sigma h \cdot [\![\mu \partial_n v]\!]_\Sigma - \nabla_\Sigma h \cdot [\![\nabla^\prime w]\!]_\Sigma                                             \\[0.5em]
			                                  &   & \quad +\ |\nabla_\Sigma h|^2 [\![\mu \partial_n w]\!]_\Sigma + G_\sigma(h)
	\end{array}
\end{equation*}
with
\begin{equation*}
	G_\sigma(h) = - \sigma {\displaystyle{\frac{|\nabla_\Sigma h|^2}{\sqrt{1 + |\nabla_\Sigma h|^2}\,\big(1 + \sqrt{1 + |\nabla_\Sigma h|^2}\big)}}} \Delta_\Sigma h - \sigma {\displaystyle{\frac{\nabla_\Sigma h \otimes \nabla_\Sigma h}{\sqrt{1 + |\nabla_\Sigma h|^2}^{\,3}}}} : \nabla^2_\Sigma h
\end{equation*}
and it is convenient to set $u = (v,\,w)$.
Therefore, as in \Subsecref{Example-Stefan} one may obtain a maximal regular solution
\begin{equation*}
	\begin{array}{rclcl}
		u & \in & \bX_u(a)      & := & H^1_p(J,\,L_p(\dot{\bR}^n,\,\bR^n)) \cap L_p(J,\,H^2_p(\dot{\bR}^n,\,\bR^n)) = H^{2, (2, 1)}_p(J\times\dot{\bR}^n,\,\bR^n), \\[1.0em]
		q & \in & \bX_q(a)      & := & \Big\{\,\pi \in L_p(J,\,\dot{H}^1_p(\dot{\bR}^n))\,:\,[\![\pi]\!]_\Sigma \in \bX_\gamma(a)\,\Big\},                         \\[1.0em]
		  &     & \bX_\gamma(a) & := & W^{1/2 - 1/2p}_p(J,\,L_p(\Sigma)) \cap L_p(J,\,W^{1 - 1/p}_p(\Sigma)) = W^{1 - 1/p, (2, 1)}_p(J \times \Sigma),             \\[0.5em]
		h & \in & \bX_h(a)      & := & W^{2 - 1/2p}_p(J,\,L_p(\Sigma)) \cap H^1_p(J,\,W^{2 - 1/p}_p(\Sigma)) \cap L_p(J,\,W^{3 - 1/p}_p(\Sigma)),
	\end{array}
\end{equation*}
to \eqnref{two-phase-nvs-qlin} via the fixed point equation
\begin{equation*}
	L(u,\,q,\,h) = N(u,\,q,\,h), \quad (u(0),\,h(0)) = (u_0,\,h_0), \qquad
	(u,\,h) \in \bX(a)
\end{equation*}
for $\bX(a) := \bX_u(a) \times \bX_q(a) \times \bX_h(a)$.
The bounded linear operator $L: \bX(a) \longrightarrow \bY(a)$ is given by the left-hand side (without initial conditions) of the linear problem
\begin{equation}
	\eqnlabel{two-phase-nvs-lin}
	\begin{array}{rclll}
		                               \rho \partial_t v - \mu \Delta v + \nabla^\prime q & = & f_v & \quad \mbox{in} & J \times \dot{\bR}^n, \\[0.5em]
		                                  \rho \partial_t w - \mu \Delta w + \partial_n q & = & f_w & \quad \mbox{in} & J \times \dot{\bR}^n, \\[0.5em]
		                                               \mbox{div}^\prime v + \partial_n w & = & g_q & \quad \mbox{in} & J \times \dot{\bR}^n, \\[0.5em]
		[\![v]\!]_\Sigma = 0, \quad [\![w]\!]_\Sigma = 0, \quad \partial_t h - [w]_\Sigma & = & f_h & \quad \mbox{on} & J \times \Sigma,      \\[0.5em]
		           - [\![\mu \partial_n v]\!]_\Sigma - [\![\mu \nabla^\prime w]\!]_\Sigma & = & g_v & \quad \mbox{on} & J \times \Sigma,      \\[0.5em]
		  - [\![2 \mu \partial_n w]\!]_\Sigma + [\![q]\!]_\Sigma - \sigma \Delta_\Sigma h & = & g_w & \quad \mbox{on} & J \times \Sigma,      \\[0.5em]
		          v(0) = v_0, \quad w(0) = w_0 \quad \mbox{in}\ \ \dot{\bR}^n, \quad h(0) & = & h_0 & \quad \mbox{on} & \Sigma
	\end{array}
\end{equation}
and the non-linear operator $N: \bX(a) \longrightarrow \bY(a)$ is given as
\begin{equation*}
	N(u,\,q,\,h) = (F_u(u,\,q,\,h),\,G_q(u,\,h),\,F_h(u,\,h),\,G_u(u,\,[\![q]\!]_\Sigma,\,h)), \ (u,\,q,\,h) \in \bX(a)
\end{equation*}
with $F_u = (F_v,\,F_w)$ and $G_u = (G_v,\,G_w)$.
The data has to be chosen according to the regularity class of the solution as
\begin{equation*}
	\begin{array}{rclrl}
		f_u & \in & \bY_{f, u}(a) & := & L_p(J \times \bR^n,\,\bR^n),                                                          \\[0.5em]
		g_q & \in & \bY_{g, q}(a) & := & H^1_p(J,\,\dot{H}^{-1}_p(\bR^n)) \cap L_p(J,\,H^1_p(\dot{\bR}^n)),                    \\[0.5em]
		f_h & \in & \bY_{f, h}(a) & := & W^{1 - 1/2p}_p(J,\,L_p(\Sigma)) \cap L_p(J,\,W^{2 - 1/p}_p(\Sigma))                   \\[0.5em]
		    &     &               &  = & W^{2-1/p, (2,1)}_p(J \times \Sigma),                                                  \\[0.5em]
		g_u & \in & \bY_{g, u}(a) & := & W^{1/2 - 1/2p}_p(J,\,L_p(\Sigma,\,\bR^n)) \cap L_p(J,\,W^{1 - 1/p}_p(\Sigma,\,\bR^n)) \\[0.5em]
		    &     &               &  = & W^{1 - 1/p, (2,1)}_p(J \times \Sigma,\,\bR^n),
	\end{array}
\end{equation*}
where we abbreviated $f_u = (f_v,\,f_w)$ and $g_u = (g_v,\,g_w)$,
and the initial conditions have to satisfy
\begin{equation*}
		u_0 \in \bZ_u := W^{2 - 2/p}_p(\dot{\bR}^n), \qquad
		h_0 \in \bZ_h := W^{3 - 2/p}_p(\Sigma)
\end{equation*}
as well as the compatibility conditions
\begin{equation}
	\eqnlabel{two-phase-nvs-compat}
	\begin{array}{c}
		\mbox{div}\,u_0 = g_q(0), \qquad \qquad [\![u_0]\!]_\Sigma = 0,\ \mbox{if}\ p > {\textstyle \frac{3}{2}}, \\[0.5em]
		-[\![\mu \partial_n v_0]\!]_\Sigma - [\![\mu \nabla^\prime w_0]\!]_\Sigma = g_v(0),\ \mbox{if}\ p > 3.
	\end{array}
\end{equation}
Furthermore, we set $\bY(a) := \bY_{f, u}(a) \times \bY_{g, q}(a) \times \bY_{f, h}(a) \times \bY_{g, u}(a)$ and $\bZ := \bZ_u \times \bZ_h$.

In order to solve the above fixed point problem first of all the linear operator 
$L$ has to be invertible, which is given by the following result.
\begin{proposition}[{\cite[Theorem~5.1]{Pruess-Simonett:Two-Phase-Navier-Stokes}}]
	\label{prop:two-phase-nvs-lin}
	Let $a > 0$ and $1 < p < \infty$ with $p \neq {\textstyle \frac{3}{2}},\,3$.
	Let $\rho = \rho_\pm > 0$, $\mu = \mu_\pm > 0$ and $\sigma > 0$.
	Then there exists a unique solution $(u,\,q,\,h) \in \bX(a)$ to the linear problem \eqnref{two-phase-nvs-lin}, if and only if the data satisfies
	\begin{equation*}
		(f_u,\,g_q,\,f_h,\,g_u) \in \bY(a), \qquad (u_0,\,h_0) \in \bZ
	\end{equation*}
	and the compatibility conditions \eqnref{two-phase-nvs-compat}.
	This solution then satisfies
	\begin{equation*}
		\|(u,\,q,\,h)\|_{\bX(a)} \leq C(a,\,p) \Big( \|(f_u,\,g_q,\,f_h,\,g_u)\|_{\bY(a)} + \|(u_0,\,h_0)\|_{\bZ} \Big)
	\end{equation*}
	with some constant $C(a,\,p) > 0$, which is independent of the data. \qedbox
\end{proposition}

In a second step we have to show that the non-linear operator $N$ enjoys suitable mapping properties to solve the above fixed-point problem.
In order to economize the notation, we denote by $\bB_{h, r}(a)$ the open ball of radius $r > 0$ in $\bX_h(a)$.
\begin{proposition}
\label{prop:two-phase-nvs-nonlin-opt}
	Let $a > 0$ and $\frac{n + 2}2 < p < \infty$.
	Let $\rho = \rho_\pm > 0$, $\mu = \mu_\pm > 0$ and $\sigma > 0$.
	Then $N \in C^\omega(\bX_r(a),\,\bY(a))$ for some $r > 0$, where $\bX_r(a) := \bX_u(a) \times \bX_q(a) \times \bB_{h, r}(a)$.
	Moreovoer, $N(0) = 0$ and the Fr{\'e}chet derivative of $N$ satisfies $DN(0) = 0$.
\end{proposition}
\begin{proof}
First observe the following embeddings,
which are consequences of corresponding {\itshape mixed derivative
theorems}; see e.g.~\cite{Pruess-Simonett:Two-Phase-Navier-Stokes}.
We have
\begin{subequations}
\begin{equation}\label{est-ht}
	\partial_t h \in W^{1 - 1/2p}_p(J,\,L_p(\Sigma)) \cap L_p(J,\,W^{2 - 1/p}_p(\Sigma)) = W^{2 - 1/p, (2,1)}_p(J \times \Sigma)
\end{equation}
for $h \in \bX_h(a)$. Furthermore,
\begin{equation*}
	W^{2 - 1/2p}_p(J,\,L_p(\Sigma)) \cap H^1_p(J,\,W^{2 - 1/p}_p(\Sigma)) \hookrightarrow W^{3/2 - 1/2p}_p(J,\,H^1_p(\Sigma)),
\end{equation*}
yields
\begin{equation}\label{est-hx}
	\begin{array}{l}
		\partial_j h \in W^{3/2 - 1/2p}_p(J,\,L_p(\Sigma)) \cap L_p(J,\,W^{2 - 1/p}_p(\Sigma)), \\[0.5em]
			\qquad \qquad \hookrightarrow W^{1 - 1/2p}_p(J,\,L_p(\Sigma)) \cap L_p(J,\,W^{2 - 1/p}_p(\Sigma)) = W^{2 - 1/p, (2,1)}_p(J \times \Sigma)
	\end{array}
\end{equation}
for $h \in \bX_h(a)$ and $j = 1,\,\dots,\,n - 1$. Finally,
\begin{equation}\label{est-hxy}
	\partial_j \partial_k h \in W^{1/2 - 1/2p}_p(J,\,L_p(\Sigma)) \cap L_p(J,\,W^{1 - 1/p}_p(\Sigma)) = W^{1 - 1/p, (2,1)}_p(J \times \Sigma)
\end{equation}
\end{subequations}
for $h \in \bX_h(a)$ and $j,\,k = 1,\,\dots,\,n - 1$.
Similarly, we obtain
\begin{subequations}
\begin{align}
	          \partial_j u & \in H^{1, (2,1)}_p(J \times \dot{\bR}^n) \hookrightarrow H^{1,(2,1)}_p(J \times \Sigma,\,L_p(\dot{\bR})),\label{est-ux} \\
	\partial_j\partial_k u & \in L_p(J \times \dot{\bR}^n) =
	H^{0,(2,1)}_p(J \times \Sigma,\,L_p(\dot{\bR})),\label{est-uxy}
\end{align}
for $j,\,k = 1,\,\dots,\,n$ as well as 
\begin{align}\label{est-ux-tr}
	[\partial_j u]_\Sigma & \in W^{1/2 - 1/2p}_p(J,\,L_p(\Sigma)) \cap L_p(J,\,W^{1 - 1/p}_p(\Sigma)) = W^{1 - 1/p, (2,1)}_p(J \times \Sigma)
\end{align}
\end{subequations}
for $j = 1,\,\dots,\,n$.

{\itshape Mapping properties of $F_u$.}
The non-linearity $F_u$ is a sum of simple multilinear operators given as products of components of the solution and its derivatives.
In order to obtain the desired mapping properties of $F_u$ it is hence sufficient to establish corresponding estimates.
According to (\ref{est-ht}), (\ref{est-hxy}) and (\ref{est-ux}) we can handle the terms
\begin{equation*}
	(\rho \partial_t h - \mu \Delta_\Sigma h)\,\partial_n \{\,v,\,w\,\},
\end{equation*}
provided the vector-valued embedding
\begin{equation}
	\label{emb-fu1}
	\underbrace{W^{1 - 1/p, (2,1)}_p(J \times \Sigma)}_{\textrm{ind}_1 = \frac{1}{2} - \frac{n + 2}{2p}} \cdot \underbrace{H^{1, (2,1)}_p(J \times \Sigma,\,L^p(\dot{\bR}))}_{\textrm{ind}_2 = \frac{1}{2} - \frac{n + 1}{2p}}
		\hookrightarrow \underbrace{H^{0, (2,1)}_p(J \times \Sigma,\,L^p(\dot{\bR}))}_{\textrm{ind} = - \frac{n + 1}{2p}}
\end{equation}
is at our disposal.
Based on Theorem~\ref{thm:Multiplication-Anisotropic} we infer that it is valid, if $\max\,\{\,\mbox{ind}_1,\,\mbox{ind}_2\,\} \geq 0$.
However, for small values of $p$ both indices on the left-hand-side become negative and Theorem~\ref{thm:Multiplication-Anisotropic} requires
the stronger condition that $\mbox{ind}_1 + \mbox{ind}_2 \geq \mbox{ind}$, which is easily seen to be equivalent to
\begin{equation}\label{reqp1weak}
	p \geq \frac{n + 2}{2}.
\end{equation}
In order to estimate the terms
\begin{equation*}
	2 \mu (\nabla_\Sigma h \cdot \nabla_\Sigma)\,\partial_n \{\,v,\,w\,\},
	\ \mu |\nabla_\Sigma h|^2 \partial^2_n \{\,v,\,w\,\},
	\ \nabla_\Sigma h\,\partial_n q
\end{equation*}
we employ (\ref{est-hx}), (\ref{est-uxy}) and the vector-valued embeddings
\begin{equation}
	\label{emb-fu2}
	\big[ \underbrace{W^{2 - 1/p, (2,1)}_p(J \times \Sigma)}_{\textrm{ind}_1 = 1 - \frac{n + 2}{2p}} \big]^m \cdot \underbrace{H^{0, (2,1)}_p(J \times \Sigma,\,L^p(\dot{\bR}))}_{\textrm{ind}_2 = - \frac{n + 1}{2p}}
		\hookrightarrow \underbrace{H^{0, (2,1)}_p(J \times \Sigma,\,L^p(\dot{\bR}))}_{\textrm{ind} = - \frac{n + 1}{2p}}
\end{equation}
for $m = 1,\,2$.
Hence, based on \Thmref{Multiplier-Anisotropic} we infer that the above embeddings are valid (for all $m \in \bN$),
provided that $\mbox{ind}_1 > 0$ or, equivalently,
\begin{equation}\label{reqp1}
	p > \frac{n + 2}{2}.
\end{equation}
Moreover, (\ref{est-hx}) shows that an estimate of the terms
\begin{equation*}
	\rho(v \cdot \nabla^\prime)\,\{\,v,\,w\,\},\ \rho w\,\partial_n \{\,v,\,w\,\}
\end{equation*}
requires the embedding
\begin{equation*}
	\underbrace{H^{2, (2, 1)}_p(J \times \dot{\bR}^n)}_{\textrm{ind}_1 = 1 - \frac{n + 2}{2p}} \cdot \underbrace{H^{1, (2, 1)}_p(J \times \dot{\bR}^n)}_{\textrm{ind}_2 = \frac{1}{2} - \frac{n + 2}{2p}}
		\hookrightarrow \underbrace{H^{0, (2, 1)}_p(J \times \dot{\bR}^n)}_{\textrm{ind} = - \frac{n + 2}{2p}}
\end{equation*}
to be valid.
This is guaranteed by Theorem~\ref{thm:Multiplication-Anisotropic}, if $\max\,\{\,\mbox{ind}_1,\,\mbox{ind}_2\,\} \geq 0$.
On the other hand, for small values of $p$ both indices on the left-hand-side become negative and Theorem~\ref{thm:Multiplication-Anisotropic}
requires the stronger condition $\mbox{ind}_1 + \mbox{ind}_2 \geq \mbox{ind}$ or, equivalently
\begin{equation}\label{reqp2}
	p \geq \frac{n + 2}{3}.
\end{equation}
Finally, according to (\ref{est-hx}) and (\ref{est-ux}) we may estimate the terms
\begin{equation*}
	\rho(v \cdot \nabla_\Sigma h)\,\partial_n \{\,v,\,w\,\}
\end{equation*}
with the aid of the embeddings
\begin{equation*}
	\underbrace{W^{2 - 1/p, (2,1)}_p(J \times \Sigma)}_{\textrm{ind}_1 = 1 - \frac{n + 2}{2p}} \cdot \underbrace{H^{1, (2,1)}_p(J \times \Sigma,\,L^p(\dot{\bR}))}_{\textrm{ind}_2 = \frac{1}{2} - \frac{n + 1}{2p}}
		\hookrightarrow \underbrace{H^{0, (2,1)}_p(J \times \Sigma,\,L^p(\dot{\bR}))}_{\textrm{ind} = - \frac{n + 1}{2p}}.
\end{equation*}
and
\begin{equation*}
	\underbrace{H^{2, (2, 1)}_p(J \times \dot{\bR}^n)}_{\textrm{ind}^\prime_1 = 1 - \frac{n + 2}{2p}} \cdot \underbrace{H^{0, (2, 1)}_p(J \times \dot{\bR}^n)}_{\textrm{ind}^\prime_2 = - \frac{n + 2}{2p}}
		\hookrightarrow \underbrace{H^{0, (2, 1)}_p(J \times \dot{\bR}^n)}_{\textrm{ind}^\prime = - \frac{n + 2}{2p}}.
\end{equation*}
Based on Theorem~\ref{thm:Multiplication-Anisotropic} the first embedding is valid, if $\max\,\{\,\mbox{ind}_1,\,\mbox{ind}_2\,\} \geq 0$;
if both of the indices on the left-hand side are negative, then the stronger condition $\mbox{ind}_1 + \mbox{ind}_2 \geq \mbox{ind}$ is required,
which is equivalent to (\ref{reqp2}).
For the second embedding to be valid \Thmref{Multiplier-Anisotropic} requires $\mbox{ind}^\prime_1 > 0$,
which is equivalent to (\ref{reqp1}).
In summary, these considerations show that $F_u$ has the desired mapping properties, provided that the constraint (\ref{reqp1}) is satisfied,
since this also implies (\ref{reqp1weak}) and (\ref{reqp2}).

{\itshape Mapping properties of $G_q$.}
Due to the structure of $G_q$ it is again sufficient to derive estimates in suitable function spaces.
First we show $G_q(u,\,h) \in H^1_p(J,\,\dot{H}^{-1}_p(\bR^n))$ and the corresponding estimates.
Since $h$ does not depend on $x_n$, we have that $\nabla_\Sigma h \cdot \partial_n v = \partial_n (\nabla_\Sigma h \cdot v)$.
Moreover, $\partial_n \in \cL(L_p(J \times \bR^n),\,L_p(J,\,\dot{H}^{-1}_p(\bR^n)))$
and, thus, we need to estimate the terms
\begin{equation*}
	\partial_t \nabla_\Sigma h \cdot v,\ \nabla_\Sigma h \cdot \partial_t v
\end{equation*}
in $L_p(J \times \bR^n)$.
Based on (\ref{est-ht}) we have $\partial_t \nabla_\Sigma h \in W^{1 - 1/p, (2, 1)}_p(J \times \Sigma)$ and the first term
may be estimated using the vector-valued embedding
\begin{equation*}
	\underbrace{W^{1 - 1/p, (2,1)}_p(J \times \Sigma)}_{\textrm{ind}_1 = \frac{1}{2} - \frac{n + 2}{2p}} \cdot \underbrace{H^{2, (2,1)}_p(J \times \Sigma,\,L^p(\dot{\bR}))}_{\textrm{ind}_2 = 1 - \frac{n + 1}{2p}}
		\hookrightarrow \underbrace{H^{0, (2,1)}_p(J \times \Sigma,\,L^p(\dot{\bR}))}_{\textrm{ind} = - \frac{n + 1}{2p}},
\end{equation*}
which is valid, if $\max\,\{\,\mbox{ind}_1,\,\mbox{ind}_2\,\} \geq 0$ thanks to Theorem~\ref{thm:Multiplication-Anisotropic};
if both of the indices on the left-hand-side are negative, then the stronger condition $\mbox{ind}_1 + \mbox{ind}_2 \geq \mbox{ind}$ is required,
which is equivalent to (\ref{reqp2}).
The second term may be estimated using (\ref{est-hx}), (\ref{est-uxy}), and the vector-valued embedding (\ref{emb-fu2}) for $m = 1$,
which is valid, provided that (\ref{reqp1}) is satisfied.
Finally, to obtain $G_q(u,\,h) \in L_p(J,\,H^1_p(\dot{\bR}^n))$ and the corresponding estimates
we need to estimate the terms
\begin{equation*}
	\partial_j \nabla_\Sigma h \cdot \partial_n v,
	\ \nabla_\Sigma h \cdot \partial_j \partial_n v,
	\ \nabla_\Sigma h \cdot \partial^2_n v,
	\qquad j = 1,\,\dots,\,n - 1,
\end{equation*}
in $L_p(J \times \bR^n)$.
This may be accomplished using (\ref{est-hx}), (\ref{est-hxy}), (\ref{est-ux}), (\ref{est-uxy}) together
with the vector-valued embeddings (\ref{emb-fu1}), and (\ref{emb-fu2}),
which are both valid, provided that (\ref{reqp1}) is satisfied.
In summary, these considerations show that $G_q$ has the desired mapping properties, provided that the constraint (\ref{reqp1}) is satisfied.

{\itshape Mapping properties of $F_h$.}
In order to obtain the desired mapping properties of $F_h$ based on (\ref{est-hx}), and (\ref{est-ux-tr})
it is sufficient to use the embedding
\begin{equation}\label{embfh}
	\underbrace{W^{2 - 1/p, (2,1)}_p(J \times \Sigma)}_{\textrm{ind}_1 = 1 - \frac{n + 2}{2p}} \cdot \underbrace{W^{1 - 1/p, (2,1)}_p(J \times \Sigma)}_{\textrm{ind}_2 = \frac{1}{2} - \frac{n + 2}{2p}}
		\hookrightarrow \underbrace{W^{1 - 1/p, (2,1)}_p(J \times \Sigma)}_{\textrm{ind} = \frac{1}{2} - \frac{n + 2}{2p}},
\end{equation}
which is provided by \Thmref{Multiplier-Anisotropic},
if $\mbox{ind}_1 > 0$.
Hence, the non-linearity $F_h$ has the desired mapping properties, provided that $\mbox{ind}_1 > 0$ or, equivalently, if (\ref{reqp1}) is satisfied.

{\itshape Mapping properties of $G_\sigma$.}
Similar to the treatment of the non-linearity $G_u$ in the proof of Proposition~\ref{prop:Stefan-Non-Linear-Optimal} we show
that there exists an $r > 0$ such that
\begin{equation*}
	G_\sigma: \bB_{h, r}(a) \longrightarrow W^{1 - 1/p, (2, 1)}_p(J \times \Sigma)
\end{equation*}
is analytic.
We again employ the functions
\begin{equation*}
	\phi: \bR^{n - 1} \longrightarrow \bR, \qquad
	\phi(\xi) := \frac{|\xi|^2}{\sqrt{1 + |\xi|^2} \big( 1 + \sqrt{1 + |\xi|^2} \big)}, \quad \xi \in \bR^{n - 1},
\end{equation*}
and
\begin{equation*}
	\psi_{jk}: \bR^{n - 1} \longrightarrow \bR, \qquad
	\psi_{jk}(\xi) := \frac{\xi_j \xi_k}{{\sqrt{1 + |\xi|^2}}^3}, \quad \xi \in \bR^{n - 1},\ j,\,k = 1,\,\dots,\,n - 1,
\end{equation*}
which are obviously analytic in a neighbourhood of the origin with $\phi(0) = \psi_{jk}(0) = 0$.
Thus, if
\begin{equation*}
	1 - \frac{n + 2}{2p} = \mbox{ind} \Big( W^{2 - 1/p, (2, 1)}_p(J \times \Sigma) \Big) > 0,
\end{equation*}
which is equivalent to (\ref{reqp1}),
then (\ref{est-hx}) and Theorem~\ref{thm:Nemytskij-Anisotropic} yield the analyticity of the mappings
\begin{equation}\label{phi-psi-nemytskij}
	h \mapsto \phi(\nabla_\Sigma h),\ h \mapsto \psi_{j k}(\nabla_\Sigma h): \bB_{h, r}(a) \longrightarrow W^{2 - 1/p, (2, 1)}_p(J \times \Sigma)
\end{equation}
for $j,\,k = 1,\,\dots,\,n - 1$ and some $r > 0$.
Since
\begin{equation*}
	G_\sigma(h) = - \sigma \phi(\nabla_\Sigma h) \Delta_\Sigma h - \sigma \sum^{n - 1}_{j, k = 1} \psi_{j k}(\nabla_\Sigma h) \partial_j \partial_k h, \qquad h \in \bB_{h, r}(a),
\end{equation*}
in order to obtain the desired mapping properties of $G_\sigma$
we only need to use (\ref{est-hxy}), the mapping properties (\ref{phi-psi-nemytskij}), and the embedding
(\ref{embfh}), which is available, if (\ref{reqp1}) is satisfied.

{\itshape Mapping properties of $G_u$.}
We may now show that the non-linearity
\vspace*{-0.25em}
\begin{equation*}
	G_u: \bX_u(a) \times \bX_\gamma(a) \times \bB_{h, r}(a) \longrightarrow \bY_{g, u}(a)
\end{equation*}
is analytic with $r > 0$ chosen as above.
First of all, the terms
\vspace*{-0.25em}
\begin{equation*}
	\sigma \Delta_\Sigma h\,\nabla_\Sigma h,\ G_\sigma(h) \nabla_\Sigma h
\end{equation*}
may be estimated based on (\ref{est-hx}), (\ref{est-hxy}), and the mapping properties of $G_\sigma$ together with the embedding
(\ref{embfh}), which is available, if (\ref{reqp1}) is satisfied.
All other terms that appear in the definition of $G_u$ may be estimated
using (\ref{est-hx}), (\ref{est-ux-tr}) and the embeddings
\begin{equation*}
	\big[ \underbrace{W^{2 - 1/p, (2,1)}_p(J \times \Sigma)}_{\textrm{ind}_1 = 1 - \frac{n + 2}{2p}} \big]^m \cdot \underbrace{W^{1 - 1/p, (2,1)}_p(J \times \Sigma)}_{\textrm{ind}_2 = \frac{1}{2} - \frac{n + 2}{2p}}
		\hookrightarrow \underbrace{W^{1 - 1/p, (2,1)}_p(J \times \Sigma)}_{\textrm{ind} = \frac{1}{2} - \frac{n + 2}{2p}}
\end{equation*}
for $m = 1,\,2$.
According to \Thmref{Multiplier-Anisotropic}, these embeddings are valid (for all $m \in \bN$),
if $\mbox{ind}_1 > 0$.
Thus, the above embeddings are available, if $\mbox{ind}_1 > 0$ or, equivalently, if (\ref{reqp1}) is satisfied.
Hence, the non-linearity $G_u$ has the desired mapping properties, provided that (\ref{reqp1}) is satisfied.

Finally, the assertions $N(0) = 0$, and $DN(0) = 0$ are trivial.
Therefore, the proof is complete.
\end{proof}

Based on Propositions~\ref{prop:two-phase-nvs-lin}, \ref{prop:two-phase-nvs-nonlin-opt}
and the contraction mapping principle in the same way as in \cite{Pruess-Simonett:Two-Phase-Navier-Stokes} we obtain the following 
improvement of \cite[Theorem~6.3~(a)]{Pruess-Simonett:Two-Phase-Navier-Stokes}.
\begin{theorem}
	\label{prop:two-phase-nvs-full}
	Let $a > 0$, $\frac{n + 2}{2} < p < \infty$, $p \neq 3$,
	$\rho = \rho_\pm > 0$, $\mu = \mu_\pm > 0$, and $\sigma > 0$.
	Then there exists $\delta > 0$ such that the quasilinear problem \eqnref{two-phase-nvs-qlin} admits a unique solution $(u,\,q,\,h) \in \bX(a)$,
	provided the initial conditions $(u_0,\,h_0) \in \bZ$ satisfy
	\begin{equation*}
		\begin{array}{c}
			\|(u_0,\,h_0)\|_{\bZ} < \delta, \\[0.5em]
			\textup{div}\,u_0 = F_q(u_0,\,h_0), \qquad [\![u_0]\!]_\Sigma = 0, \\[0.5em]
			-[\![\mu \partial_n v_0]\!]_\Sigma - [\![\mu \nabla^\prime w_0]\!]_\Sigma = G_w(u_0,\,h_0),\ \mbox{if}\ p > 3.
		\end{array}
	\end{equation*}
	The solutions depend continuously on the data.
\end{theorem}

\begin{remark}
\remlabel{Example-NVS}
(a) Theorem~\ref{prop:two-phase-nvs-full} and its proof given here improves 
\cite{Pruess-Simonett:Two-Phase-Navier-Stokes} in two directions.
Combined with the Newton-Polygon approach to maximal regularity
for linearized mixed order systems developed in \cite{Denk-Saal-Seiler:Newton-Polygon},
it provides a systematic way to handle quasilinear problems such as 
the two-phase Navier-Stokes equations \eqnref{two-phase-nvs-qlin}.
Secondly, by the sharp embedding result Theorem~\ref{thm:Multiplication-Anisotropic}
it includes a considerable improvement of the range for admissible
$p$ from $p > n + 2$ as required in \cite[Theorem~6.3~(a)]{Pruess-Simonett:Two-Phase-Navier-Stokes}
to $p > (n + 2) / 2$.
That such an improvement is possible is already conjectured in \cite[Remark~1.2~(a)]{Pruess-Simonett:Two-Phase-Navier-Stokes}.

(b) We note that the statements of 
Proposition~\ref{prop:two-phase-nvs-nonlin-opt}, and Theorem~\ref{prop:two-phase-nvs-full} slightly differ
from those of \cite[Proposition~6.2 \& Theorem~6.3]{Pruess-Simonett:Two-Phase-Navier-Stokes} in the following way:
In \cite{Pruess-Simonett:Two-Phase-Navier-Stokes} the whole problem is considered in $\bR^{n^\prime + 1}$ and not in $\bR^n$.
For this reason, the results \cite[Proposition~6.2 \& Theorem~6.3]{Pruess-Simonett:Two-Phase-Navier-Stokes} require the condition $p > n^\prime + 3$,
which may cause confusion unless one notes that $n = n^\prime + 1$.

(c) Theorem~\ref{prop:two-phase-nvs-full} could be extended to include \cite[Theorem~6.3~(b)]{Pruess-Simonett:Two-Phase-Navier-Stokes},
which provides analyticity of the solutions.
Indeed, the main ingredient of its proof is the analyticity of the non-linearities,
which is also obtained in Proposition~\ref{prop:two-phase-nvs-nonlin-opt}.
Therefore, the proof given in \cite{Pruess-Simonett:Two-Phase-Navier-Stokes} carries over to the extended range of admissible $p$.

(d) The smallness condition on $h$ in $\bX_h(a)$ required in \ref{prop:two-phase-nvs-nonlin-opt}
can be replaced by a smallness condition on $\nabla_\Sigma$ in $\partial \bX_h(a) := W^{2 - 1/p, (2,1)}_p(J \times \Sigma)$,
cf.~\Remref{Example-Stefan}~(b).

(e) For the general theory of parabolic free boundary problems including the two-phase Navier-Stokes equations
we refer to the pertinent monograph \cite{Pruess-Simonett:Moving-Interfaces}.
\end{remark}
\newpage

\section*{Appendix}\renewcommand{\thesection}{A}\setcounter{theorem}{0}
In this appendix we collect elementary auxiliary results as well as embedding and interpolation properties of anisotropic vector-valued function spaces.
These are repeatedly employed in the proofs of our first main theorem, cf.~\Secref{Multiplication-Proof}.
However, they are not necessary for the formulation and application of the theorems,
which is the reason for collecting them in an appendix.
Recall that $[\rho]_\ominus = \min\,\{\,0,\,\rho\,\} \leq 0$ for $\rho \in \bR$;
analogously, we set $[\rho]_\oplus = \max\,\{0,\,\rho\,\} \geq 0$ for $\rho \in \bR$.

\subsection*{Elementary Auxiliary Results}
\subsection*{A Realization Lemma}
The following lemma plays a fundamental role in the proof of \Thmref{Multiplication-Anisotropic}.
It is used in several steps in order to choose suitable integrability parameters for function spaces
that appear on the right-hand side of an embedding.
\begin{lemma}
	\lemlabel{Appendix-Realization}
	Let $m \in \bN$ with $m \geq 2$.
	Let $0 \leq \sigma_j < \infty$ and $0 < \pi_j < 1$ for $j \in M := \{\,1,\,\dots,\,m\,\}$.
	Let $0 < \rho < 1$ such that
	\begin{equation}
		\eqnlabel{Realization-Target-Range}
		- \sum_{j \in M} \pi_j \leq - \rho \leq \sum_{j \in M} \big[ \sigma_j - \pi_j \big]_\ominus.
	\end{equation}
	Then there exist $0 \leq \rho_j < 1$ for $j \in M$ such that
	\begin{equation*}
		- \pi_j \leq - \rho_j \leq \sigma_j - \pi_j, \qquad j \in M, \qquad \qquad
		- \sum_{j \in M} \rho_j = - \rho.
	\end{equation*}
	In particular, $\rho_j > 0$ for all $j \in M$ with $\sigma_j < \pi_j$,
	and $- \rho_j < \sigma_j - \pi_j$ for all $j \in M$ with $\sigma_j > \pi_j$.
	If the second inequality in \eqnref{Realization-Target-Range} is strict,
	then the $\rho_j$ may be chosen such that $\rho_j > 0$ for all $j \in M$,
	and such that $- \rho_j < \sigma_j - \pi_j$ for all $j \in M$ with $\sigma_j \neq 0$.
\end{lemma}
\begin{proof}
If $\sigma_1 = \sigma_2 = \dots = \sigma_m = 0$, then the assertions are trivial.
Hence, we may assume $\sigma_j > 0$ for at least one $j \in M$.
We define the functions $\phi_j: [0,\,1] \longrightarrow [0,\,1)$ for $j \in M$ as
\begin{equation*}
	\phi_j(\theta) := (1 - \theta) \big[ \pi_j - \sigma_j \big]_\oplus + \theta \pi_j, \qquad 0 \leq \theta \leq 1.
\end{equation*}
Note that $0 \leq [ \pi_j - \sigma_j ]_\oplus \leq \pi_j < 1$,
which shows that the functions $\phi_j$ are well-defined and continuous by construction.
Moreover, $\phi_j$ is strictly increasing, if $\sigma_j > 0$, and constant, if $\sigma_j = 0$.
Now, we define $\phi: [0,\,1] \longrightarrow [0,\,m)$ as
\begin{equation*}
	\phi(\theta) := \sum_{j \in M} \phi_j(\theta), \qquad 0 \leq \theta \leq 1.
\end{equation*}
Then $\phi$ is well-defined, continuous, and strictly increasing.
We have
\begin{equation*}
	\min_{0 \leq \theta \leq 1} \phi(\theta) = \phi(0) = \sum_{j \in M} \big[ \pi_j - \sigma_j \big]_\oplus
		\leq \rho \leq \sum_{j \in M} \pi_j = \phi(1) = \max_{0 \leq \theta \leq 1} \phi(\theta)
\end{equation*}
and $\phi$ takes its minimum precisely at $\theta = 0$ and its maximum precisely at $\theta = 1$.
Therefore, there exists a unique $0 \leq \theta_0 \leq 1$ such that $\phi(\theta_0) = \rho$.
Now, we define $\rho_j := \phi_j(\theta_0)$ for $j \in M$.
Then the monotonicity of the $\phi_j$ implies
\begin{equation*}
	\max\,\{\,\pi_j - \sigma_j,\,0\,\} = \big[ \pi_j - \sigma_j \big]_\oplus \leq \rho_j \leq \pi_j < 1, \qquad j \in M,
\end{equation*}
and the $\rho_j$ sum up to $\rho$ by construction.

Now, if $\sigma_j = 0$ then $\phi_j$ is constant and $\rho_j = \pi_j > 0$.
Furthermore, if $0 < \sigma_j < \pi_j$, then $\rho_j \geq \phi_j(0) > 0$.
Finally, if $\sigma_j \geq \pi_j$, then $\rho_j = \phi_j(\theta_0) > 0$, provided that $\theta_0 > 0$;
analogously, if $\sigma_j < \pi_j$, then $\rho_j > \pi_j - \sigma_j = \phi_j(0)$, provided that $\theta_0 > 0$;
however, we have $\theta_0 > 0$ provided that the second inequality in \eqnref{Realization-Target-Range} is strict.
\end{proof}

\subsection*{A Minimization Lemma}
For convenience, we cite a lemma concerning the minimization of a certain class of functions,
which plays an important role in the proof of \Thmref{Multiplication-Anisotropic}.
This is a slightly modified version of \cite[Lemma~1.1]{Amann:Multiplication};
the prerequisites and assertions are simply multiplied by $-1$ as compared to \cite[Lemma~1.1]{Amann:Multiplication} for convenience.
\begin{lemma}
	\lemlabel{Appendix-Minimization}
	Let $m,\,n \in \bN$ with $m \geq 2$.
	Moreover, let $0 < \sigma_j,\,\pi_j < \infty$ for $j \in M := \{\,1,\,\dots,\,m\,\}$ and set
	\begin{equation*}
		\begin{array}{c}
			M_0 := \Big\{\,j \in M\,:\,\sigma_j - \pi_j = 0\,\Big\}, \qquad M_\pm := \Big\{\,j \in M\,:\,\pm(\sigma_j - \pi_j) > 0\,\Big\}, \\[1.0em]
			M_\bullet := \Big\{\,j \in M\,:\,\sigma_j - \pi_j = {\displaystyle{\min_{k \in M}}} (\sigma_k - \pi_k) =: \mu\,\Big\}.
		\end{array}
	\end{equation*}
	Then the function
	\begin{equation*}
		\phi: N := \Big\{\,\nu \in \bN^m_0\,:\,|\nu| \leq n\,\Big\} \longrightarrow \bR, \quad
		\phi(\nu) := \sum^m_{j = 1} \big[ \sigma_j - \nu_j - \pi_j \big]_\ominus, \ \nu \in N
	\end{equation*}
	satisfies
	\begin{equation*}
		\phi_\bullet := \min_{\nu \in N}\,\phi(\nu) = \left\{
			\begin{array}{ll}
				\mbox{\upshape min}\,\Big\{\,\big[ (\sigma_j - \pi_j) - n \big]_\ominus\,:\,j \in M\,\Big\}, & \quad \mbox{if} \ \mu \geq 0, \\[1.0em]
				{\displaystyle{\sum_{j \in M_-}}} (\sigma_j - \pi_j) - n,                                    & \quad \mbox{otherwise}
			\end{array} \right.
	\end{equation*}
	Furthermore,
	\begin{itemize}
		\item if $\mu \geq n$, then $\phi(\nu) = \phi_\bullet = 0$ for every $\nu \in N$;
		\item if $0 \leq \mu < n$, then $\phi(\nu) = \phi_\bullet$,
			if and only if $\nu \in N$ with $|\nu| = n$ and $\nu_j = 0$ for all $j \in M_+$ except for at most one $j \in M_\bullet$;
		\item if $\mu < 0$, then $\phi(\nu) = \phi_\bullet$,
			if and only if $\nu \in N$ with $|\nu| = n$ and $\nu_j = 0$ for all $j \in M_+$.
	\end{itemize}
\end{lemma}

\subsection*{Embeddings and Interpolation of Anisotropic Function Spaces}
\subsection*{Embeddings of Besov Spaces}
The scale of vector-valued anisotropic Besov spaces $B^{s, \omega}_{p, q}(\bR^n,\,A)$,
where $n,\,\omega \in \bN^\nu$ with $\nu \in \bN$, and where $A$ is a Banach space, has been defined in \Subsecref{Anisotropic-Spaces}
for the parameter range $- \infty < s < \infty$, $1 < p < \infty$, and $1 \leq q \leq \infty$.
Using the definition from \cite[Section~3.3]{Amann:Maximal-Regularity} instead,
one may extend the parameter range to $1 \leq p \leq \infty$.
Now, the known embedding theorems for this scale are usually presented as
\begin{equation*}
	\begin{array}{lc}
		\begin{array}{l} \textrm{embeddings w.\,r.\,t.\ {\itshape differentiability:}} \\[0.5em]
		B^{s_1, \omega}_{p, q_1}(\bR^n,\,A) \hookrightarrow B^{s_0, \omega}_{p, q_0}(\bR^n,\,A), \end{array} & \qquad - \infty < s_0 < s_1 < \infty,\ 1 \leq p \leq \infty,\ 1 \leq q_0,\,q_1 \leq \infty, \\[2.5em]
		\begin{array}{l} \textrm{embeddings w.\,r.\,t.\ the {\itshape micro-scale:}} \\[0.5em]
		B^{s, \omega}_{p, q_1}(\bR^n,\,A) \hookrightarrow B^{s, \omega}_{p, q_0}(\bR^n,\,A), \end{array} & \qquad - \infty < s < \infty,\ 1 \leq p \leq \infty,\ 1 \leq q_1 \leq q_0 \leq \infty, \\[2.5em]
		\begin{array}{l} \textrm{embeddings w.\,r.\,t.\ the {\itshape index:}} \\[0.5em]
		B^{s_1, \omega}_{p_1, q}(\bR^n,\,A) \hookrightarrow B^{s_0, \omega}_{p_0, q}(\bR^n,\,A), \end{array} & \qquad \begin{array}{c} - \infty < s_0 < s_1 < \infty,\ 1 \leq p_1 < p_0 \leq \infty,\ 1 \leq q \leq \infty, \\[0.5em] \mbox{ind}(B^{s_0, \omega}_{p_0, q}(\bR^n,\,A)) = \mbox{ind}(B^{s_1, \omega}_{p_1, q}(\bR^n,\,A)), \end{array}
	\end{array}
\end{equation*}
cf.~e.\,g.~\cite[Theorem~3.3.2]{Amann:Maximal-Regularity}.
However, these may be combined into the compact form
\begin{equation}
	\eqnlabel{Besov-Embedding}
	B^{s_1, \omega}_{p_1, q_1}(\bR^n,\,A) \hookrightarrow B^{s_0, \omega}_{p_0, q_0}(\bR^n,\,A), \qquad
	\begin{array}{c} - \infty < s_0 \leq s_1 < \infty, \\[0.5em] 1 \leq p_1 \leq p_0 \leq \infty,\ 1 \leq q_0,\,q_1 \leq \infty, \\[0.5em] \mbox{ind}(B^{s_0, \omega}_{p_0, q_0}(\bR^n,\,A)) \leq \mbox{ind}(B^{s_1, \omega}_{p_1, q_1}(\bR^n,\,A)), \\[0.5em] \textrm{where the last inequality is strict or}\ q_1 \leq q_0. \end{array}
\end{equation}
Indeed, if equality holds for the indices,
then $s_0 = s_1$ implies $p_0 = p_1$ and \eqnref{Besov-Embedding} reduces to the second embedding above.
If equality holds for the indices and $s_0 < s_1$, then \eqnref{Besov-Embedding} follows from the second and third embedding above.
Finally, if the inequality between the indices is strict, then $s_0 < s_1$ and the first and third embedding above imply
\begin{equation*}
	B^{s_1, \omega}_{p_1, q_1}(\bR^n,\,A) \hookrightarrow B^{s, \omega}_{p_1, q_0}(\bR^n,\,A) \hookrightarrow B^{s_0, \omega}_{p_0, q_0}(\bR^n,\,A),
\end{equation*}
where we choose $s_0 < s < s_1$ such that $\mbox{ind}(B^{s_0, \omega}_{p_0, q_0}(\bR^n,\,A)) = \mbox{ind}(B^{s, \omega}_{p_1, q_0}(\bR^n,\,A))$.

\subsection*{Interpolation of Besov Spaces}
The common interpolation results for vector-valued anisotropic Besov spaces
follow from the fact that, by its common definition as in \cite[Section~3.3]{Amann:Maximal-Regularity},
the space $B^{s, \omega}_{p, q}(\bR^n,\,A)$ is a retract of $\ell^s_q(\bN,\,L_p(\bR^n,\,A))$, cf.~\cite[Eq.~(3.3.5)]{Amann:Maximal-Regularity}.
Now, for $0 < \theta < 1$ and an arbitrary interpolation couple $(A_0,\,A_1)$ of Banach spaces we have
\begin{equation*}
	\begin{array}{lc}
		[\ell^{s_0}_{q_0}(\bN,\,A_0),\ \ell^{s_1}_{q_1}(\bN,\,A_1)]_\theta = \ell^s_q(\bN,\,[A_0,\,A_1]_\theta), \qquad
		\begin{array}{c} - \infty < s_0,\,s_1 < \infty,\ s = (1 - \theta) s_0 + \theta s_1, \\[0.5em] 1 \leq q_0,\,q_1 < \infty,\ \frac{1}{q} = \frac{1 - \theta}{q_0} + \frac{\theta}{q_1}, \end{array}
	\end{array}
\end{equation*}
which follows from the fact that $\ell^s_q(\bN,\,A) = \ell_q(\bN,\,(2^{sk} A)_{k \in \bN})$ for $- \infty < s < \infty$ and $1 \leq q < \infty$,
from \cite[Theorem~1.18.1 \& Remark~1.18.1/1]{Triebel:Interpolation}, and from $[2^{s_0 k} A_0,\,2^{s_1 k} A_1]_\theta = 2^{s k} [A_0,\,A_1]_\theta$
for $- \infty < s_0,\,s_1 < \infty$, $s = (1 - \theta) s_0 + \theta s_1$ and $k \in \bN$.
Moreover, \cite[Theorem~1.18.4]{Triebel:Interpolation} implies
\begin{equation}
	\eqnlabel{Lebesgue-Complex-Interpolation}
	\begin{array}{lc}
		[L_{p_0}(\bR^n,\,A),\ L_{p_1}(\bR^n,\,A)]_\theta = L_p(\bR^n,\,A), \qquad \qquad
		1 \leq p_0,\,p_1 < \infty,\ \frac{1}{p} = \frac{1 - \theta}{p_0} + \frac{\theta}{p_1}.
	\end{array}
\end{equation}
Thus, a retraction argument, cf.~\cite[Section~1.2.4]{Triebel:Interpolation}, yields
\begin{equation}
	\eqnlabel{Besov-Complex-Interpolation}
	[B^{s_0, \omega}_{p_0, q_0}(\bR^n,\,A),\ B^{s_1, \omega}_{p_1, q_1}(\bR^n,\,A)]_\theta \doteq B^{s, \omega}_{p, q}(\bR^n,\,A), \qquad
	\begin{array}{c} - \infty < s_0,\,s_1 < \infty, \\[0.5em] s = (1 - \theta) s_0 + \theta s_1, \\[0.5em] 1 \leq p_0,\,p_1 < \infty,\ \frac{1}{p} = \frac{1 - \theta}{p_0} + \frac{\theta}{p_1}, \\[0.5em] 1 \leq q_0,\,q_1 < \infty,\ \frac{1}{q} = \frac{1 - \theta}{q_0} + \frac{\theta}{q_1}. \end{array}
\end{equation}
Concerning the real interpolation method we have
\begin{equation*}
	\begin{array}{lc}
		(\ell^{s_0}_{q_0}(\bN,\,A),\ \ell^{s_1}_{q_1}(\bN,\,A))_{\theta, q} \doteq \ell^s_q(\bN,\,A), & \qquad
		\begin{array}{c} - \infty < s_0,\,s_1 < \infty,\ s = (1 - \theta) s_0 + \theta s_1, \\[0.5em] s_0 \neq s_1,\ 1 \leq q_0,\,q_1,\,q \leq \infty, \end{array} \\[2.5em]
		(\ell^{s_0}_{q_0}(\bN,\,A_0),\ \ell^{s_1}_{q_1}(\bN,\,A_1))_{\theta, q} \doteq \ell^s_q(\bN,\,(A_0,\,A_1)_{\theta, q}), & \qquad
		\begin{array}{c} - \infty < s_0,\,s_1 < \infty,\ s = (1 - \theta) s_0 + \theta s_1, \\[0.5em] 1 \leq q_0,\,q_1,\,q < \infty,\ \frac{1}{q} = \frac{1 - \theta}{q_0} + \frac{\theta}{q_1}, \end{array}
	\end{array}
\end{equation*}
where the first result is \cite[Theorem~1.18.2]{Triebel:Interpolation}.
The second result follows from the fact that $\ell^s_q(\bN,\,A) = \ell_q(\bN,\,(2^{sk} A)_{k \in \bN})$ for $- \infty < s < \infty$ and $1 \leq q < \infty$,
from \cite[Theorem~1.18.1]{Triebel:Interpolation}, and from $(2^{s_0 k} A_0,\,2^{s_1 k} A_1)_{\theta, q} = 2^{s k} (A_0,\,A_1)_{\theta, q}$
for $- \infty < s_0,\,s_1 < \infty$, $s = (1 - \theta) s_0 + \theta s_1$, $1 \leq q < \infty$ and $k \in \bN$.
Moreover, \cite[Theorem~1.18.4]{Triebel:Interpolation} implies
\begin{equation}
	\eqnlabel{Lebesgue-Real-Interpolation}
	\begin{array}{lc}
		(L_{p_0}(\bR^n,\,A),\ L_{p_1}(\bR^n,\,A))_{\theta, p} \doteq L_p(\bR^n,\,A), \qquad \qquad
		1 \leq p_0,\,p_1 < \infty,\ \frac{1}{p} = \frac{1 - \theta}{p_0} + \frac{\theta}{p_1}.
	\end{array}
\end{equation}
Thus, a retraction argument, cf.~\cite[Section~1.2.4]{Triebel:Interpolation}, yields
\begin{equation}
	\eqnlabel{Besov-Real-Interpolation}
	\begin{array}{lc}
		(B^{s_0, \omega}_{p, q_0}(\bR^n,\,A),\ B^{s_1, \omega}_{p, q_1}(\bR^n,\,A))_{\theta, q} \doteq B^{s, \omega}_{p, q}(\bR^n,\,A), & \qquad
		\begin{array}{c} - \infty < s_0,\,s_1 < \infty, \\[0.5em] s = (1 - \theta) s_0 + \theta s_1, \\[0.5em] s_0 \neq s_1,\ 1 \leq p,\,q_0,\,q_1,\,q \leq \infty, \end{array} \\[3.5em]
		(B^{s_0, \omega}_{p_0, q_0}(\bR^n,\,A),\ B^{s_1, \omega}_{p_1, q_1}(\bR^n,\,A))_{\theta, p} \doteq B^{s, \omega}_p(\bR^n,\,A), & \qquad
		\begin{array}{c} - \infty < s_0,\,s_1 < \infty, \\[0.5em] s = (1 - \theta) s_0 + \theta s_1, \\[0.5em] 1 \leq p_0,\,p_1 < \infty,\ \frac{1}{p} = \frac{1 - \theta}{p_0} + \frac{\theta}{p_1}, \\[0.5em] 1 \leq q_0,\,q_1 < \infty,\ \frac{1}{p} = \frac{1 - \theta}{q_0} + \frac{\theta}{q_1}. \end{array}
	\end{array}
\end{equation}
The first result -- and its derivation -- coincide with \cite[Eq.\ (3.3.12)]{Amann:Maximal-Regularity}.

\subsection*{Embeddings of Bessel Potential Spaces}
The scale of vector-valued anisotropic Bessel potential spaces $H^{s, \omega}_p(\bR^n,\,A)$,
where $n,\,\omega \in \bN^\nu$ with $\nu \in \bN$, and where $A$ is a UMD-space with property $(\alpha)$, if $\omega \neq \dot{\omega}(1,\,\dots,\,1)$,
has been defined in \Subsecref{Anisotropic-Spaces} for the parameter range $- \infty < s < \infty$, and $1 < p < \infty$.
The definition coincides with that from \cite[Section~3.7]{Amann:Maximal-Regularity}.
Now, the known embedding theorems for this scale are usually presented as
\begin{equation*}
	\begin{array}{lc}
		\begin{array}{l} \textrm{embeddings w.\,r.\,t.\ {\itshape differentiability:}} \\[0.5em]
		H^{s_1, \omega}_p(\bR^n,\,A) \hookrightarrow H^{s_0, \omega}_p(\bR^n,\,A), \end{array} & \qquad - \infty < s_0 \leq s_1 < \infty,\ 1 < p < \infty, \\[2.5em]
		\begin{array}{l} \textrm{embeddings w.\,r.\,t.\ the {\itshape index:}} \\[0.5em]
		H^{s_1, \omega}_{p_1}(\bR^n,\,A) \hookrightarrow H^{s_0, \omega}_{p_0}(\bR^n,\,A), \end{array} & \qquad \begin{array}{c} - \infty < s_0 < s_1 < \infty,\ 1 < p_1 < p_0 < \infty, \\[0.5em] \mbox{ind}(H^{s_0, \omega}_{p_0}(\bR^n,\,A)) = \mbox{ind}(H^{s_1, \omega}_{p_1}(\bR^n,\,A)), \end{array}
	\end{array}
\end{equation*}
cf.~e.\,g.~\cite[Theorem~3.7.5]{Amann:Maximal-Regularity}.
However, these may be combined into the compact form
\begin{equation}
	\eqnlabel{Bessel-Potential-Embedding}
	H^{s_1, \omega}_{p_1}(\bR^n,\,A) \hookrightarrow H^{s_0, \omega}_{p_0}(\bR^n,\,A), \qquad \qquad
	\begin{array}{c} - \infty < s_0 \leq s_1 < \infty,\ 1 < p_1 \leq p_0 < \infty, \\[0.5em] \mbox{ind}(H^{s_0, \omega}_{p_0}(\bR^n,\,A)) \leq \mbox{ind}(H^{s_1, \omega}_{p_1}(\bR^n,\,A)). \end{array}
\end{equation}
Indeed, if equality holds for the indices,
then $s_0 = s_1$ implies $p_0 = p_1$ and there is nothing to be proved.
If equality holds for the indices and $s_0 < s_1$, then \eqnref{Bessel-Potential-Embedding} follows from the second embedding above.
Finally, if the inequality between the indices is strict, then $s_0 < s_1$ and the two embeddings above imply
\begin{equation*}
	H^{s_1, \omega}_{p_1}(\bR^n,\,A) \hookrightarrow H^{s, \omega}_{p_1}(\bR^n,\,A) \hookrightarrow H^{s_0, \omega}_{p_0}(\bR^n,\,A),
\end{equation*}
where we choose $s_0 < s < s_1$ such that $\mbox{ind}(H^{s_0, \omega}_{p_0}(\bR^n,\,A)) = \mbox{ind}(H^{s, \omega}_{p_1}(\bR^n,\,A))$.

\subsection*{Interpolation of Bessel Potential Spaces}
The common interpolation results for vector-valued anisotropic Bessel potential spaces
follow from the fact that $H^{s, \omega}_p(\bR^n,\,A) = \cB^{- s, \omega} L_p(\bR^n,\,A)$ is a retract of $L_p(\bR^n,\,A)$.
For $0 < \theta < 1$ and a UMD-space $A$ with property $(\alpha)$, if $\omega \neq \dot{\omega}(1,\,\dots,\,1)$, we have
\begin{equation}
	\eqnlabel{Bessel-Potential-Complex-Interpolation}
	\begin{array}{lc}
		[H^{s_0, \omega}_p(\bR^n,\,A),\ H^{s_1, \omega}_p(\bR^n,\,A)]_\theta \doteq H^{s, \omega}_p(\bR^n,\,A), & \qquad
		\begin{array}{c} - \infty < s_0,\,s_1 < \infty, \\[0.5em] s = (1 - \theta) s_0 + \theta s_1, \\[0.5em] 1 < p < \infty, \end{array} \\[3.5em]
		[H^{s, \omega}_{p_0}(\bR^n,\,A),\ H^{s, \omega}_{p_1}(\bR^n,\,A)]_\theta \doteq H^{s, \omega}_p(\bR^n,\,A), & \qquad
		\begin{array}{c} - \infty < s < \infty, \\[0.5em] 1 < p_0,\,p_1 < \infty,\ \frac{1}{p} = \frac{1 - \theta}{p_0} + \frac{\theta}{p_1}, \end{array}
	\end{array}
\end{equation}
where the first result follows from the fact that the vector-valued anisotropic Bessel potential scale is a fractional power scale, cf.~\cite[Theorem~3.7.1]{Amann:Maximal-Regularity},
and the second result follows from \eqnref{Lebesgue-Complex-Interpolation} and a retraction argument, cf.~\cite[Section~1.2.4]{Triebel:Interpolation}.
Concerning the real interpolation method we note that
\begin{equation}
	\eqnlabel{Bessel-Potential-Sandwich}
	B^{s, \omega}_{p, 1}(\bR^n,\,A) \hookrightarrow H^{s, \omega}_p(\bR^n,\,A) \hookrightarrow B^{s, \omega}_{p, \infty}(\bR^n,\,A), \qquad
	- \infty < s < \infty,\ 1 < p < \infty,
\end{equation}
cf.~\cite[Theorem~3.7.1]{Amann:Maximal-Regularity}.
Therefore, we obtain
\begin{equation}
	\eqnlabel{Bessel-Potential-Real-Interpolation}
	\begin{array}{lc}
		(H^{s_0, \omega}_p(\bR^n,\,A),\ H^{s_1, \omega}_p(\bR^n,\,A))_{\theta, q} \doteq B^{s, \omega}_{p, q}(\bR^n,\,A), & \qquad
		\begin{array}{c} - \infty < s_0,\,s_1 < \infty, \\[0.5em] s_0 \neq s_1,\ s = (1 - \theta) s_0 + \theta s_1, \\[0.5em] 1 < p < \infty, \ 1 \leq q \leq \infty, \end{array} \\[3.5em]
		(H^{s, \omega}_{p_0}(\bR^n,\,A),\ H^{s, \omega}_{p_1}(\bR^n,\,A))_{\theta, p} \doteq H^{s, \omega}_p(\bR^n,\,A), & \qquad
		\begin{array}{c} - \infty < s < \infty, \\[0.5em] 1 < p_0,\,p_1 < \infty,\ \frac{1}{p} = \frac{1 - \theta}{p_0} + \frac{\theta}{p_1}, \end{array}
	\end{array}
\end{equation}
where the first result follows from \eqnref{Bessel-Potential-Sandwich} and \eqnref{Besov-Real-Interpolation}
and the second result follows from \eqnref{Lebesgue-Real-Interpolation} and a retraction argument, cf.~\cite[Section~1.2.4]{Triebel:Interpolation}.

\subsection*{Embeddings of Different Scales}
By definition we have $H^{0, \omega}_p(\bR^n,\,A) = L_p(\bR^n,\,A)$ for all $1 < p < \infty$.
Therefore, it holds that
\begin{equation}
	\eqnlabel{Bessel-Potential-Lebesgue-Embedding}
	H^{s, \omega}_p(\bR^n,\,A) \hookrightarrow L_r(\bR^n,\,A), \qquad \qquad
	\begin{array}{c} 0 \leq s < \infty,\ 1 < p < \infty,\ p \leq r \leq \infty, \\[0.5em] \mbox{ind}_\omega(L_r(\bR^n,\,A)) \leq \mbox{ind}(H^{s, \omega}_p(\bR^n,\,A)), \\[0.5em] \textrm{where the last inequality is strict or}\ r < \infty. \end{array}
\end{equation}
For $r < \infty$ this embedding is a direct consequence of \eqnref{Bessel-Potential-Embedding} and the definition
\begin{equation*}
	\mbox{ind}_\omega(L_p(\bR^n,\,A)) = - \frac{\omega \cdot n}{\dot{\omega}}\,\frac{1}{p}, \qquad \qquad
	1 \leq p \leq \infty,
\end{equation*}
of the anisotropic index, cf.~\Subsecref{Anisotropic-Spaces}.
For $r = \infty$ the embedding follows from \cite[Theorem~3.9.1]{Amann:Maximal-Regularity}.
Now, a similar result for the vector-valued anisotropic Besov scale may be obtained by \eqnref{Besov-Embedding}, \eqnref{Bessel-Potential-Sandwich}, and \eqnref{Bessel-Potential-Lebesgue-Embedding}.
However, \eqnref{Besov-Embedding} requires a strict inequality between the indices to allow for a regularity change on the micro-scale.
An optimized result can be obtained as follows, cf.~\cite[Remark~2.8.1/3]{Triebel:Interpolation}.
Given $- \infty < s_0 < s_1 < \infty$ and $1 < p_1 < p_0 < \infty$ with $\mbox{ind}(H^{s_0, \omega}_{p_0}(\bR^n,\,A)) = \mbox{ind}(H^{s_1, \omega}_{p_1}(\bR^n,\,A))$
we have by \eqnref{Bessel-Potential-Embedding} that
\begin{equation*}
	H^{s_1 \pm \epsilon, \omega}_{p_1}(\bR^n,\,A) \hookrightarrow H^{s_0, \omega}_{q_\pm}(\bR^n,\,A)
\end{equation*}
for all $\epsilon > 0$ such that $s_0 < s_1 - \epsilon$
and all $1 < q_\pm < \infty$ such that $\frac{\omega \cdot n}{q_\pm} = \frac{\omega \cdot n}{p_0} \mp \epsilon$.
Thus, by an embedding on the micro-scale and real interpolation $(\,\cdot\,,\,\cdot\,)_{\frac{1}{2}, p_0}$ of the above embeddings via \eqnref{Bessel-Potential-Real-Interpolation} we obtain
\begin{equation*}
	B^{s_1, \omega}_{p_1, q}(\bR^n,\,A) \hookrightarrow B^{s_1, \omega}_{p_1, p_0}(\bR^n,\,A) \hookrightarrow H^{s_0, \omega}_{p_0}(\bR^n,\,A), \qquad \qquad
	1 \leq q \leq p_0.
\end{equation*}
Similarly, \eqnref{Bessel-Potential-Embedding} implies that
\begin{equation*}
	H^{s_1, \omega}_{q_\pm}(\bR^n,\,A) \hookrightarrow H^{s_0 \pm \epsilon, \omega}_{p_0}(\bR^n,\,A)
\end{equation*}
for all $\epsilon > 0$ such that $s_0 + \epsilon < s_1$
and all $1 < q_\pm < \infty$ such that $\frac{\omega \cdot n}{q_\pm} = \frac{\omega \cdot n}{p_1} \mp \epsilon$.
Thus, by real interpolation $(\,\cdot\,,\,\cdot\,)_{\frac{1}{2}, p_1}$ of the above embeddings via \eqnref{Bessel-Potential-Real-Interpolation} and an embedding on the micro-scale we obtain
\begin{equation*}
	H^{s_1, \omega}_{p_1}(\bR^n,\,A) \hookrightarrow B^{s_0, \omega}_{p_0, p_1}(\bR^n,\,A) \hookrightarrow B^{s_0, \omega}_{p_0, q}(\bR^n,\,A), \qquad \qquad
	p_1 \leq q \leq \infty.
\end{equation*}
Therefore, using \eqnref{Besov-Embedding}, \eqnref{Bessel-Potential-Sandwich}, and the above embeddings we obtain
\begin{equation}
	\eqnlabel{Besov-Bessel-Potential-Embedding}
	B^{s_1, \omega}_{p_1, q}(\bR^n,\,A) \hookrightarrow H^{s_0, \omega}_{p_0}(\bR^n,\,A), \qquad \qquad
	\begin{array}{c} - \infty < s_0 < s_1 < \infty, \\[0.5em] 1 \leq p_1 \leq p_0 \leq \infty,\ 1 \leq q \leq \infty, \\[0.5em] \mbox{ind}(H^{s_0, \omega}_{p_0}(\bR^n,\,A)) \leq \mbox{ind}(B^{s_1, \omega}_{p_1, q}(\bR^n,\,A)), \\[0.5em] \textrm{where the last inequality is strict or}\ q \leq p_0, \end{array}
\end{equation}
as well as
\begin{equation}
	\eqnlabel{Bessel-Potential-Besov-Embedding}
	H^{s_1, \omega}_{p_1}(\bR^n,\,A) \hookrightarrow B^{s_0, \omega}_{p_0, q}(\bR^n,\,A), \qquad \qquad
	\begin{array}{c} - \infty < s_0 < s_1 < \infty, \\[0.5em] 1 < p_1 \leq p_0 \leq \infty,\ 1 \leq q \leq \infty, \\[0.5em] \mbox{ind}(B^{s_0, \omega}_{p_0, q}(\bR^n,\,A)) \leq \mbox{ind}(H^{s_1, \omega}_{p_1}(\bR^n,\,A)), \\[0.5em] \textrm{where the last inequality is strict or}\ p_1 \leq q. \end{array}
\end{equation}
Now, a combination of \eqnref{Besov-Bessel-Potential-Embedding} and \eqnref{Bessel-Potential-Lebesgue-Embedding} yields
\begin{equation}
	\eqnlabel{Besov-Lebesgue-Embedding}
	B^{s, \omega}_{p, q}(\bR^n,\,A) \hookrightarrow L_r(\bR^n,\,A), \qquad
	\begin{array}{c} 0 < s < \infty,\ 1 \leq p < \infty,\ p \leq r \leq \infty,\ r \neq 1, \\[0.5em] \mbox{ind}_\omega(L_r(\bR^n,\,A)) \leq \mbox{ind}(B^{s, \omega}_{p, q}(\bR^n,\,A)), \\[0.5em] \textrm{where the last inequality is strict or}\ q \leq p\ \textrm{and}\ r < \infty. \end{array}
\end{equation}

\bibliographystyle{plain}
\bibliography{references}
\end{document}